\documentclass[a4paper,10pt,utf8]{article}
\usepackage[figuresright]{rotating}
\usepackage{pdflscape} 	
\usepackage[utf8]{inputenc}
\usepackage[font=small, labelfont=bf, margin=1cm]{caption}
\usepackage{amsmath,amsfonts,amssymb}
\usepackage{tabularx,booktabs}
\usepackage{xspace}
\usepackage{etex}
\usepackage{enumitem}
\usepackage{dsfont}
\usepackage{mathtools}
\usepackage[linesnumbered,ruled,vlined]{algorithm2e}

\SetCommentSty{AlgoCommentStyle}
\usepackage{changes}
\usepackage{longtable}
\usepackage{geometry}
\usepackage{amsthm}
\usepackage{multirow}
\usepackage[numbers,sort&compress]{natbib}  
\usepackage[breaklinks,colorlinks,citecolor=blue,linkcolor=blue,hyperfootnotes=false]{hyperref}
\usepackage[capitalize]{cleveref}
\usepackage{pgfplots}
\usepackage{subcaption}
\usepackage{textcomp}
\usepgfplotslibrary{groupplots}
\usepgfplotslibrary{units}
\usepackage{graphicx}
\usepackage{microtype}
\usepackage{siunitx}
\usepackage{xparse}
\ExplSyntaxOn
\def\myround#1{\num{\fp_eval:n {round(#1, 2)}}}
\ExplSyntaxOff	
\usepackage{float}
\usepackage{tcolorbox}
\usetikzlibrary{arrows.meta}
\usetikzlibrary{math}
\usetikzlibrary{calc}
\usetikzlibrary{patterns,patterns.meta}

\newcommand{\LP}{{LP}\xspace}

\newcommand{\MIP}{{MIP}\xspace}

\newcommand{\MINLP}{{MINLP}\xspace}
\newcommand{\MILP}{{MILP}\xspace}

\DeclareMathOperator{\conv}{conv}

\DeclareMathOperator*{\argmin}{argmin}
\DeclareMathOperator*{\argmax}{argmax}

\newcommand{\MIR}{\text{MIR}\xspace}
\newcommand{\Z}{\mathbb{Z}}
\newcommand{\R}{\mathbb{R}}

\newcommand{\Fset}{\mathcal{X}}

\newcommand{\Kset}{\mathcal{X}^{\text{K}}}

\newcommand{\Tset}{\mathcal{X}^{\text{T}}}
\newcommand{\Cset}{\mathcal{X}^{\text{C}}}

\newcommand{\FsetS}{\Fset^{\text{S}}}
\newcommand{\FsetB}{\Fset^{\text{B}}}
\newcommand{\settingO}{\textbf{OA}\xspace}
\newcommand{\settingS}{$+$\textbf{Single-phase}\xspace}
\newcommand{\settingT}{$+$\textbf{Two-phase}\xspace}
\newcommand{\settingC}{$+$\textbf{Combination}\xspace}

\newcommand{\floor}[1]{\lfloor #1 \rfloor}
\newcommand{\ceil}[1]{\lceil #1 \rceil}
\providecommand{\keywords}[1]{\small \quad \quad \textbf{Keywords:} #1}
\theoremstyle{plain}
\newtheorem{theorem}{Theorem}[section]
\newtheorem{corollary}[theorem]{Corollary}
\newtheorem{remark}[theorem]{Remark}
\newtheorem{observation}[theorem]{Observation}
\newtheorem{proposition}[theorem]{Proposition}
\newtheorem{example}[theorem]{Example}
\newtheorem{lemma}[theorem]{Lemma}

\newtheorem{claim}[theorem]{Claim}
\newtheorem{fact}{Fact}
\newtheorem*{claim*}{Claim}
\crefname{figure}{Figure}{Figure}
\crefname{observation}{Observation}{Observations}
\crefname{fact}{Fact}{Facts}
\setlist[itemize]{leftmargin=3.45ex}
\setlist[itemize,1]{label=$-$,itemsep=0ex,topsep=0.9ex}
\setlist[itemize,2]{label=$\cdot$,topsep=0.5ex,leftmargin=2.75ex}
\setlist[enumerate]{leftmargin=3ex,itemsep=0.1ex,parsep=1ex,topsep=0.9ex}

\usepackage{titling}
\pretitle{\begin{center}\linespread{1.05}\Large}
	\posttitle{\par\end{center}\vskip 0.5em}
\preauthor{\begin{center}\linespread{1.15}\large \lineskip 0.6em\begin{tabular}[t]{c}}
		\postauthor{\end{tabular}\par\end{center}\vskip 0.6em}
\predate{\begin{center}}
	\postdate{\par\end{center}}
\thanksmarkseries{arabic}

\usepackage[]{titlesec}
\usepackage{etoolbox}
\makeatletter
\patchcmd{\ttlh@hang}{\parindent\z@}{\parindent\z@\leavevmode}{}{}
\patchcmd{\ttlh@hang}{\noindent}{}{}{}
\makeatother
\titleformat{\section}
{\normalfont\large\bfseries}{\thesection}{0.9ex}{}
\titleformat{\subsection}
{\normalfont\normalsize\bfseries}{\thesubsection}{0.9ex}{}
\titleformat{\subsubsection}
{\normalfont\normalsize\upshape}{\thesubsubsection}{0.9ex}{}
\titleformat{\paragraph}[runin]
{\normalfont\normalsize\itshape}{\theparagraph}{1em}{}
\titleformat{\subparagraph}[runin]
{\normalfont\normalsize\itshape}{\theparagraph}{1em}{}
\titlespacing*{\section}     {0pt}{21dd plus 8pt minus 4pt}{10.5dd}
\titlespacing*{\subsection}   {0pt}{21dd plus 8pt minus 4pt}{10.5dd}
\titlespacing*{\subsubsection}{0pt}{19dd plus 8pt minus 4pt}{10.5dd}
\titlespacing*{\paragraph}   {0pt}{13pt plus 8pt minus 4pt}{1em}
\titlespacing*{\subparagraph}   {0pt}{13pt plus 8pt minus 4pt}{1em}

\usepackage{supertabular}

\makeatletter

\providecommand{\fnm}[1]{\leavevmode\hbox{#1}}
\providecommand{\sur}[1]{\unskip~\nobreak\leavevmode\hbox{#1}}

\providecommand{\orgdiv}[1]{#1}
\providecommand{\orgname}[1]{#1}
\providecommand{\orgaddress}[1]{#1}

\providecommand{\postcode}[1]{#1}
\providecommand{\city}[1]{#1}

\providecommand{\country}[1]{#1}

\usepackage{orcidlink}  

\let\SN@orcidlink@raw\orcidlink
\renewcommand{\orcidlink}[1]{%
	\raisebox{-0ex}{\scalebox{0.90}{\SN@orcidlink@raw{#1}}}%
}

\newcommand\SN@authors{}
\newcommand\SN@affils{}
\newcommand\SN@corrEmails{}
\newcommand\SN@contribEmails{}
\newcommand\SN@authsep{}
\newif\ifSN@lastcorr
\SN@lastcorrfalse

\def\SN@mailto#1{\href{mailto:#1}{\textcolor{blue}{#1}}}

\newcount\SN@aucount
\RenewDocumentCommand{\author}{s o m}{%
	\ifnum\SN@aucount>0
	\g@addto@macro\SN@authors{,\ }%
	\fi
	\advance\SN@aucount by 1\relax
	\g@addto@macro\SN@authors{%
		#3%
		\IfValueT{#2}{\textsuperscript{#2}}%
		\IfBooleanT{#1}{\textsuperscript{*}}%
	}%
	\IfBooleanTF{#1}{\SN@lastcorrtrue}{\SN@lastcorrfalse}%
}

\NewDocumentCommand{\email}{m}{%
	\ifSN@lastcorr
	\g@addto@macro\SN@corrEmails{\SN@mailto{#1};\ }%
	\else
	\g@addto@macro\SN@contribEmails{\SN@mailto{#1};\ }%
	\fi
}

\NewDocumentCommand{\affil}{s o m}{%
	\g@addto@macro\SN@affils{%
		\par\noindent
		\IfValueT{#2}{\textsuperscript{#2\IfBooleanT{#1}{*}}}%
		#3.\par
	}%
}

\newcommand\SN@pacsTitle{}
\newcommand\SN@pacsBody{}
\NewDocumentCommand{\pacs}{O{PAC Codes} m}{%
	\gdef\SN@pacsTitle{#1}%
	\gdef\SN@pacsBody{#2}%
}
\newcommand{\printpacs}{%
	\par\medskip
	\noindent\textbf{\SN@pacsTitle:}~\SN@pacsBody\par
	\medskip
}

\newcommand{\makefrontmatter}{%
	\begin{center}
		{\LARGE \@title \par}
		\vspace{0.9em}
		{\large \SN@authors \par}
		\vspace{0.9em}
		{\large \SN@affils \par}
		\vspace{1.4em}
		\begin{minipage}{0.9\linewidth}
			\centering
			\ifx\SN@corrEmails\@empty\else
			*Corresponding author(s). E-mail(s): \SN@corrEmails\par
			\fi
			\ifx\SN@contribEmails\@empty\else
			Contributing authors: \SN@contribEmails\par
			\fi
		\end{minipage}
	\end{center}
}

\RenewDocumentCommand{\pacs}{O{PAC Codes} m}{%
	\small \quad \quad \textbf{#1:} #2%
}

\makeatother

\begin{document}

\title{On strong valid inequalities for a class of mixed-integer nonlinear sets with box constraints}

\author[1]{\fnm{Keyan} \sur{Li}}\email{likeyan@bit.edu.cn}
\author[1]{\fnm{Yan-Ru} \sur{Wang}}\email{wangyanru@bit.edu.cn}
\author*[1]{\fnm{Wei-Kun} \sur{Chen}\orcidlink{0000-0003-4147-1346}}\email{chenweikun@bit.edu.cn}
\author[2,3]{\fnm{Yu-Hong} \sur{Dai}\orcidlink{0000-0002-6932-9512}}\email{dyh@lsec.cc.ac.cn}

\affil*[1]{\orgdiv{School of Mathematics and Statistics}, \orgname{Beijing Institute of Technology}, \orgaddress{\city{Beijing}, \postcode{100081}, \country{China}}}
\affil[2]{\orgdiv{State Key Laboratory of Mathematical Sciences, Academy of Mathematics and Systems Science}, \orgname{Chinese Academy of Sciences}, \orgaddress{\city{Beijing}, \postcode{100190}, \country{China}}}
\affil[3]{\orgdiv{School of Mathematical Sciences}, \orgname{University of Chinese Academy of Sciences}, \orgaddress{\city{Beijing}, \postcode{100049}, \country{China}}}

\newgeometry{left=38mm,right=38mm,top=35mm}

\makefrontmatter

 \begin{abstract}
 	In this paper, 
 	we investigate the mixed-integer nonlinear set with box constraints $\Fset = \left\{ (w,x) \in \R \times \Z^n \,:\, w \leq f(a^\top x), ~0 \leq x \leq \mu \right\}$, 
 	where $f$ is a univariate concave function, $a \in \mathbb{R}^n$, and $\mu \in \Z^n_{++}$.
 	This set arises as a substructure in many mixed-integer nonlinear optimization models and encompasses, as special cases, several previously investigated mixed-integer sets---namely the submodular maximization set, the mixed-integer knapsack set, and the mixed-integer polyhedral conic set.
 	We present the first comprehensive polyhedral study of $\conv(\Fset)$.
 	In particular, we derive a class of seed inequalities for a two-dimensional restriction of $\Fset$, obtained by fixing all but one of the $x$ variables to their bounds in $\Fset$, and develop two lifting procedures to obtain strong valid inequalities for $\conv(\Fset)$.
 	In the first lifting procedure, we derive a subadditive approximation for the exact lifting function of the seed inequalities, and lift all fixed variables in a single phase.
	In the second lifting procedure, we first lift variables fixed at their lower bounds before those at their upper bounds (and vice versa), using subadditive exact and approximation lifting functions, respectively.
	The derived single- and two-phase lifted inequalities are shown to be facet-defining for $\conv(\Fset)$ under mild conditions.
	Moreover, for the aforementioned special cases of $\conv(\Fset)$, we show that the proposed lifted inequalities can either unify existing strong valid inequalities or yield {new} facet-defining inequalities.
 	Finally, extensive computational experiments on expected utility maximization and weapon target assignment problems demonstrate that the proposed lifted inequalities can substantially strengthen the continuous relaxations and significantly improve the overall computational performance of the branch-and-cut algorithms.
 	\vspace{8pt} \\
 	\vspace{8pt}\\
 	\noindent \pacs[Mathematics Subject Classification]{90C11 $\cdot$ 90C57}
 \end{abstract}

\section{Introduction}\label{sec:introduction}
In this paper, we provide a polyhedral study of the mixed-integer nonlinear set defined by a nonlinear constraint and box constraints:
\begin{equation}\label{Fset}
	\Fset : = \left\{ (w,x) \in \R \times \Z^n \,:\, w \leq f(a^\top x), ~ 
	0 \leq x_i \leq \mu_i, \ \forall \ i \in [n] \right\},
\end{equation}
where $f:\, \R \rightarrow \R$ is a concave function, $a \in \R^n$, and $\mu \in \Z_{++}^n$.
For simplicity, we assume without loss of generality that $a \in \R_+^n$ as otherwise, 
we can complement $x_i:= \mu_i - \bar{x}_i$ for  $i \in [n]$ with $a_i < 0$ (throughout, for a nonnegative integer $\tau$, we denote $[\tau] := \{1,\ldots,\tau\}$, with the convention that $[0] = \varnothing$).
The set $\Fset$ appears frequently as a substructure in many mixed-integer nonlinear optimization models, including expected utility maximization \cite{Ahmed2011,Li2019,Paul1982}, competitive facility location \cite{Berman1998,Aboolian2007,Ljubic2018}, combinatorial auctions \cite{Feige2011,Lehmann2006}, data summarization \cite{Dolhansky2016,El-Arini2009,Tschiatschek2014,Kai2015}, weapon target assignment \cite{Manne1958,Ahuja2007,Andersen2022,Lu20211,Bertsimas2025}, and discrete nonlinear nonseparable knapsack \cite{Klastorin1990,Sharkey2011}, among others.

In many practical applications, the decision vector $x$ in substructure $\Fset$ is binary (i.e., $\mu_i = 1$ for all $i \in [n]$).
In this case, $\Fset$ reduces to 
\begin{equation}\label{FsetB}
	\FsetB : = \left\{ (w,x) \in \R \times \{0, 1\}^n \,:\, 
	w \leq f( a^\top x) \right\}.
\end{equation}
It is known that the composition of a nonnegative linear function and a concave function is submodular \cite{Ahmed2011}, and therefore $\FsetB$ is a special case of the submodular maximization set: $\mathcal{Y} := \left\{ (w,x) \in \R \times \{0, 1\}^n \,:\, w \leq F(x) \right\}$.
Here, a function $F:\,\{0, 1\}^n \rightarrow \R$ is called \emph{submodular} if for any $x, y \in \{0, 1\}^n$ with $x \leq y$ (where the ``$\leq$'' is component-wise) {and any $i \in [n]$ with $x_i = y_i = 0$}, 
it follows $F(x + \mathbf{e}_i) - F(x) \geq F(y + \mathbf{e}_i) - F(y)$, where $\mathbf{e}_i$ is the $i$-th standard unit vector.
In the seminal work, \citet{Nemhauser1988} investigated the generic submodular maximization set $\mathcal{Y}$, 
arising in submodular maximization problems, 
and developed a class of submodular inequalities that is able to provide a linear characterization of the submodular maximization set $\mathcal{Y}$, 
thereby enabling mixed-integer linear programming (\MILP)-based approaches for solving related submodular maximization problems.
For the special case of $\mathcal{Y}$ where the submodular function $F(x)$ is of the form $f(a^\top x)$ (i.e., $\FsetB$),
\citet{Ahmed2011} strengthened the submodular inequalities via a sequence-independent lifting procedure
and provided sufficient conditions for the derived inequalities to be facet-defining for $\conv(\FsetB)$.
\citet{Shi2022} further showed that the exact lifting functions in \citet{Ahmed2011} are naturally subadditive (a function $h$ is subadditive on $\mathcal{D}$ if $h(\delta_1) + h(\delta_2) \geq h(\delta_1+\delta_2)$ holds for all $\delta_1, \delta_2, \delta_1+\delta_2 \in \mathcal{D}$), 
which enables to  develop two families of facet-defining inequalities for $\conv(\FsetB)$ that are stronger than those in \cite{Ahmed2011}.
Computational evidence in \cite{Ahmed2011,Shi2022} demonstrated that the enhanced submodular inequalities are effective in strengthening the continuous relaxations and improving the performance of \MILP-based approaches for submodular maximization problems. 
For $\FsetB$  with knapsack or cardinality constraints,
\citet{Yu2017} and \citet{Shi2022} exploited these constraints and derived strong valid inequalities using sequence-independent lifting procedures.
Another important and closely related variant is the submodular minimization set 
$\mathcal{Y}^{\geq} :=  \left\{ (w,x) \in \R \times \{0, 1\}^n \,:\, w \geq F(x) \right\}$, where $F$ is a submodular function.
\citet{Edmonds2003} showed that $\conv(\mathcal{Y}^{\geq})$ admits an explicit linear description via the so-called extended polymatroid inequalities; 
subsequently, \citet{Yu2017B} and \citet{Yu2023} strengthened 
these inequalities for $\mathcal{Y}^{\geq}$  with an additional cardinality constraint.
\citet{Atamturk2009} investigated the submodular knapsack set $\mathcal{Y}^{\geq} \cap \{w = c \}$ with $c \in \R$, and developed cover-based inequalities.

Beyond the binary case, the general integrality of the decision vector $x$ broadens the applicability of $\Fset$, and therefore related polyhedral investigations (for some special cases) were also provided in the literature.
In particular, when  $f(z) = \min\{0, b - z\}$ with $b \in \mathbb{R}$ is a piecewise linear concave function, $\Fset$ reduces to the well-known mixed-integer knapsack set:
\begin{equation}\label{Kset}
\begin{aligned}
	\Kset & : = \left\{ (w, x) \in \R \times \Z^n \,:\, w \leq \min\{0, b-a^\top x\} , ~ 
	0 \leq x_i \leq \mu_i, ~ \forall ~ i \in [n] \right\} \\
	& = \left\{ (w, x) \in \R \times \Z^n \,:\, w \leq 0, ~w \leq b-a^\top x , ~ 
	0 \leq x_i \leq \mu_i, ~ \forall ~ i \in [n] \right\}.
\end{aligned}
\end{equation}
This set arises as a fundamental substructure in \MILP problems---every constraint of an \MILP formulation can define a mixed-integer knapsack relaxation of the form \(\Kset\).
Consequently, strong valid inequalities for $\conv(\Kset)$ can be used as cutting planes for improving the computational performance of solving \MILP problems in a branch-and-cut framework.
A substantial body of literature develops strong valid inequalities for $\conv(\Kset)$, including the mixed-integer rounding (\MIR) \cite{Marchand2001,Nemhauser1990}, 2-{step} \MIR \cite{Dash2006}, $n$-{step} \MIR \cite{Kianfar2009}, lifted two-integer knapsack inequalities \cite{Agra2007}, mingling \cite{Atamturk2010a}, $n$-{step} mingling \cite{Atamturk2012},
and mixed-integer knapsack cover and pack inequalities \cite{Atamturk2003,Atamturk2005}. 
{For computational evidence of the efficacy of inequalities derived from $\Kset$ in solving \MILP problems, see \cite{Marchand2001,Agra2007,Atamturk2003,Dash2010,Kaparis2010,Fukasawa2011,Sanjeevi2012} and references therein.}
Another special case of $\Fset$ is the mixed-integer polyhedral conic set \cite{Atamturk2010b} (where $f(z) = -|b - z|$ with $b \in \R$ is a concave function):
\begin{equation}\label{Cset}
	\Cset := \left\{(w, x) \in \R \times \Z^n \,:\, w \leq - |b - a^\top x|, \ 0 \leq x_i \leq \mu_i, \ \forall \ i \in [n]\right\}.
\end{equation}
This set arises as a substructure in the polyhedral conic reformulation of mixed-integer second-order conic programs \cite{Atamturk2010b}.
\citet{Atamturk2010b} proposed conic \MIR inequalities for $\Cset$ and demonstrated that the derived inequalities can effectively reduce the integrality gap of continuous relaxations of mixed-integer second-order conic programs, thereby improving the overall computational efficiency of branch-and-cut approaches.

Similar to the special cases $\conv(\FsetB)$, $\conv(\Kset)$, and $\conv(\Cset)$, the general case $\conv(\Fset)$ is also polyhedral as 
$\Fset$ is the union of a finite set of rays with the same direction.
In addition to the three special cases $\FsetB$, $\Kset$, and $\Cset$,
the set $\Fset$ also arises as a substructure in many other practical applications. 
These applications include (i) the weapon target assignment problem \cite{Manne1958,Ahuja2007,Andersen2022,Lu20211,Bertsimas2025}, 
discrete nonlinear nonseparable knapsack problem \cite{Klastorin1990,Sharkey2011}, 
and multiple gradual cover location problem \cite{Berman2019}
in which $f(a^\top x) = 1 - \exp(-a^\top x)$,
(ii) competitive facility location and design \cite{Aboolian2007B} 
in which $f(a^\top x) = \left(1-\exp(-(a^\top x + b)) \right) \cdot \left( \frac{a^\top x}{a^\top x + b} \right)$ with $b \in \R$,
and (iii) resource allocation \cite{Moriguchi2011}, equilibrium analysis of network congestion games \cite{Fujishige2015}, and inventory and portfolio management \cite{Chen2021} in which $f(a^\top x) = f(\mathbf{1}^\top x)$, i.e., $a = \mathbf{1}$ is an all-ones vector.
Despite such a wide range of real-world applications, a comprehensive polyhedral study of $\conv(\Fset)$ and its application in a branch-and-cut framework to solve related problems is still missing in the literature.

Analyzing the polyhedral structure of the general case $\conv(\Fset)$, however, is much more challenging than those of the special cases $\conv(\FsetB)$, $\conv(\Kset)$, and $\conv(\Cset)$, 
because $\Fset$ simultaneously involves general integrality and nonlinearity.
On one hand, due to the general integrality, mixed-integer polyhedra are significantly harder to study than mixed 0-1 polyhedra. 
For instance, combinatorial arguments that yield strong valid inequalities for mixed 0-1 polyhedra often fail to provide high-dimensional faces for mixed-integer polyhedra  \cite{Atamturk2003,Atamturk2005}.
As noted in \cite{Atamturk2003,Atamturk2005}, one possible explanation for this difficulty is that, for mixed 0-1 sets, all feasible solutions lie on the ``surface'' of the continuous relaxation (obtained by removing the integrality constraints from the {mixed} 0-1 sets),
whereas for mixed-integer sets,  some feasible points generally lie deep in the continuous relaxation (i.e., not necessarily on some face of the continuous relaxation).
On the other hand, the nonlinearity $w\le f(a^\top x)$ makes it much more difficult to derive strong valid inequalities for $\conv(\Fset)$ than for its linear counterpart $\conv(\Kset)$ (or the nearly linear counterpart $\conv(\Cset)$).
In particular, in the context of deriving  lifted inequalities, 
the associated lifting functions for $\Kset$ and  $\Cset$ are defined by solving \MILP problems, making it easier to derive closed formulas or subadditive approximations  \cite{Atamturk2003,Atamturk2004} (and obtain strong valid inequalities).  
In sharp contrast, in the same context, the related lifting functions for $\Fset$ are defined by solving relatively ``harder''  {mixed-integer nonlinear programming (\MINLP)} problems, rendering it difficult to derive {closed formulas or subadditive approximations}.

\subsection{Main contributions}
The main motivation of this paper is to fill the research gap by providing an  in-depth polyhedral study of $\conv(\Fset)$.
In particular, we develop several classes of strong valid inequalities for  polyhedron $\conv(\Fset)$ by applying sequence-independent lifting techniques \cite{Wolsey1977,Richard2011} to the seed inequalities (valid for a two-dimensional restriction of $\Fset$), and apply the resultant inequalities as cutting planes to improve the computational performance of solving related problems.
To the best of our knowledge, this is the first time that the polyhedral structure of $\conv(\Fset)$ with an arbitrary concave function $f$ is investigated.
The main contributions are summarized as follows.
\begin{itemize}
	\item [$\bullet$] 
	We consider a two-dimensional restriction of $\Fset$, obtained by fixing all but one of the integer variables in $\Fset$ to their lower or upper bounds, and derive a class of valid inequalities. 
	These valid inequalities are able to not only provide a complete linear characterization of the convex hull of the two-dimensional restriction but also serve as a building block in developing strong valid inequalities for $\conv(\Fset)$ 
	(by leveraging lifting techniques). 
	\item [$\bullet$] 
	We derive a closed formula for the exact lifting function associated with the valid inequalities for the two-dimensional restriction of $\Fset$.
	The exact lifting function is, however, not subadditive on $\R$ in general.
	To bypass the difficulty, we construct a subadditive approximation for the exact lifting function, which allows for
	simultaneously lifting all fixed variables in a single phase, thereby obtaining a family of strong valid inequalities for $\conv(\Fset)$.
	We show that the single-phase lifted inequalities are facet-defining for $\conv(\Fset)$ under mild conditions.	
	\item [$\bullet$] 
	Although the exact lifting function is not subadditive on $\R$, it is subadditive on $\R_+$ and on $\R_-$ separately, which opens up the possibility to obtain other strong valid inequalities using a two-phase procedure. 
	Specifically, we can derive another two classes of strong valid inequalities (i) by first lifting the variables fixed at their lower bounds (using the exact lifting function) and then lifting those fixed at their upper bounds,
	and (ii) by first lifting the variables fixed at their upper bounds (using the exact lifting function) and then lifting those fixed at their lower bounds.
	We provide subadditive lower and upper bound approximations for the second-phase lifting functions. 
	With this result, we are able to not only efficiently compute  the two-phase lifted inequalities but also provide mild conditions for them to be facet-defining for $\conv(\Fset)$.
	As a byproduct of analysis, we develop two classes of  subadditive functions that encompass existing subadditive functions proposed in \cite{Atamturk2003,Shi2022} as special cases.
	\item[$\bullet$] 
	When applying the proposed single- and two-phase lifted inequalities to the three special cases $\conv(\FsetB)$, $\conv(\Kset)$, and $\conv(\Cset)$, our results show that
	the proposed inequalities can either unify existing ones in the literature or provide new facet-defining inequalities.
	In particular, for $\conv(\FsetB)$, the single-phase lifted inequalities can provide  new   facet-defining inequalities, and 
	the two-phase lifted ones are equivalent to the lifted inequalities in \citet{Shi2022}.
	For $\conv(\Kset)$, the single-phase lifted inequalities are equivalent to the well-known \MIR inequalities in \cite{Nemhauser1990,Marchand2001}, and the two-phase lifted ones are equivalent to the mixed-integer knapsack cover and pack inequalities in \cite{Atamturk2003,Atamturk2005}.
	For $\conv(\Cset)$, the single-phase lifted inequalities are equivalent to conic \MIR inequalities in \cite{Atamturk2010b} and the two-phase lifted ones can provide new facet-defining inequalities.
	It is worth noting that using a variable transformation technique, the proposed  single- and two-phase inequalities can also be applied to  the two-row mixed-integer set \cite{Nemhauser1990,Bodur2017}.
\end{itemize}

We apply the proposed single- and two-phase lifted inequalities as cutting planes within a branch-and-cut framework for solving the expected utility maximization \cite{Ahmed2011,Li2019,Paul1982} and weapon target assignment problems \cite{Manne1958,Ahuja2007,Andersen2022,Lu20211,Bertsimas2025}.
Extensive computational experiments demonstrate that the proposed single- and two-phase lifted inequalities can substantially strengthen the continuous relaxations and significantly improve the overall computational performance of the branch-and-cut framework to solve the two problems.

\subsection{Outline}
The remainder of this paper is organized as follows.
\cref{sec:preliminaries} presents some basic properties of concave functions
and proposes a family of facet-defining inequalities that can characterize the convex hull of a two-dimensional restriction of $\Fset$ (i.e., $\conv(\Fset)$ with $n = 1$).
\cref{sec:lift-one-phase,sec:twosteplifting} develop sequence-independent lifting procedures to derive the single- and two-phase lifted inequalities for $\conv(\Fset)$, respectively.
\cref{sec:special-case} analyzes the connections between the proposed inequalities and previously derived inequalities for some special cases of $\Fset$ in the literature.
\cref{sec:computational-result} presents the computational results.
Finally, \cref{sec:conslusion} concludes the paper.
\section{Preliminaries}\label{sec:preliminaries}
In this section, we first present some basic properties of concave functions that will be utilized throughout this paper. 
Then, we investigate the polyhedral structure of $\conv(\Fset)$ for the special case $n=1$. 
In particular, we provide
a family of valid inequalities that can characterize $\conv(\Fset)$ with $n=1$ and will serve as a building block in deriving strong valid inequalities for $\conv(\Fset)$ with an arbitrary positive integer $n$ in \cref{sec:lift-one-phase,sec:twosteplifting}.

\subsection{Basic properties of concave functions}\label{subsec:concave-func}\label{sect:basicpro}
\cref{lem:slope,cor:sum-pariwise-comparison} summarize some properties of concave functions.

\begin{lemma} \label{lem:slope}
	Let $f: \R \rightarrow \R$ be a concave function. 
	Then for any $a_1,a_2, b_1, b_2 \in \R$ 
	with $a_1 \leq b_1$, $a_2 \leq b_2$, $a_1 \leq a_2$, and $b_1 \le b_2$, the following inequality holds:
	\begin{equation}\label{ineq:concavity-slope}
		(b_2 - a_2)(f(b_1) - f(a_1)) \geq(b_1 - a_1) (f(b_2) - f(a_2)).
	\end{equation}
	In particular, if $b_1 - a_1 = b_2 - a_2$, then $f(b_1) - f(a_1) \geq f(b_2) - f(a_2)$.
\end{lemma}
\begin{proof}
	The proof follows from Remark 1.4.1 of \cite{Niculescu2025}.
\end{proof}

\begin{lemma}\label{cor:sum-pariwise-comparison}
	Let $f: \R \rightarrow \R$ be a concave function and $H = \sum_{i=1}^m \tau_i f(c_i)$ with $\tau, c \in \mathbb{R}^m$ satisfying  
	\begin{equation}\label{ccondition}
		\sum_{i=1}^m \tau_i c_i=0.
	\end{equation}
	If $H$ can be regrouped into the form of
	\begin{equation}\label{Hdef}
	H = \sum_{i\in I^+} k_i (f(b_i) - f(a_i))-   \sum_{i\in I^-} k_i (f(b_i) - f(a_i))
	\end{equation}
	with
	(i) $k_i \geq 0$ and $a_i \leq b_i $ for all $i \in I^+\cup I^-$ and (ii) $\max_{i\in I^+} b_i \le \min_{i\in I^-} b_i$ and
	$\max_{i\in I^+} a_i \le \min_{i\in I^-} a_i$,
	then $H \ge 0$.
\end{lemma}
\begin{proof}
	Without loss of generality, we assume that $k_i > 0$ and $a_i < b_i$ for all $i \in I^+ \cup I^-$ (as otherwise, we can remove the corresponding term $k_i (f(b_i) - f(a_i))$ from \eqref{Hdef}).
	If $I^+\cup I^- = \varnothing$, then $H=0$ and thus the statement follows. 
	Otherwise, it follows from \eqref{ccondition} that
	$\sum_{i \in I^+} k_i (b_i - a_i) - \sum_{i \in I^-} k_i (b_i - a_i) = 0$, 
	and thus  $I^+ \neq \varnothing$ and $I^- \neq \varnothing$ must hold.
	Combining  $a_i < b_i$ for all $i \in I^+ \cup I^-$, condition (ii), and \cref{lem:slope}, we obtain 
	\begin{equation}\label{tmpeq}
		\frac{f(b_i) - f(a_i)}{b_i - a_i} \ge \frac{f(b_j) - f(a_j)}{b_j - a_j}, \ \forall \ i \in I^+, \ j \in I^-.
	\end{equation}
	Letting $\lambda = \min_{i \in I^+} \frac{f(b_i) - f(a_i)}{b_i - a_i}$, then 
	by \eqref{ccondition} and \eqref{tmpeq}, we have
	\begin{align*}
		H 
		& = \sum_{i\in I^+} k_i \cdot \frac{f(b_i) - f(a_i)}{b_i - a_i} \cdot (b_i - a_i) 
		- \sum_{i\in I^-} k_i \cdot \frac{f(b_i) - f(a_i)}{b_i - a_i} \cdot (b_i - a_i)\\
		& \ge \lambda\left(\sum_{i \in I^+} k_i (b_i - a_i) - \sum_{i \in I^-} k_i (b_i - a_i)\right) = \lambda \sum_{i = 1}^m \tau_i c_i = 0.  \qedhere
	\end{align*}
\end{proof}
Observe that if $I^+ =\{1\}$, $I^-=\{2\}$, $k_1 =b_2 -a_2$, and $k_2 = b_1-a_1$, then \eqref{Hdef} reduces to 
\begin{equation*}
	H=(b_2 - a_2)(f(b_1) - f(a_1)) -(b_1 - a_1) (f(b_2) - f(a_2)). 
\end{equation*}
Therefore, the result in \cref{cor:sum-pariwise-comparison} generalizes that in \cref{lem:slope}.

\subsection{Polyhedral structure of $\conv(\Fset)$ with $n=1$} \label{subsec:convhull-F1}
Next, we consider the special case of $\Fset$ with $n=1$:
\begin{equation}\label{F1}
	\Fset_1 = \left\{ (w, x) \in \R \times \Z \,:\,
	w \leq f(a x), \ 0 \leq x \leq \mu \right\},
\end{equation}
and investigate the polyhedral structure of its convex hull.
Here, $a>0$ and $\mu \in {\Z_{++}}$. 
We first present a family of facet-defining inequalities for $\conv(\Fset_1)$.
\begin{proposition}\label{prop:2dim-facet}
	For any $k \in [\mu]$, inequality 
	\begin{equation}
		\label{eq:seedInq1}
		w \leq [f(k a) - f((k - 1) a)] (x - k) + f(k a)
	\end{equation}
	is facet-defining for $\conv(\Fset_1)$.
\end{proposition}

\begin{proof}
	Let $(w,x)$ be an arbitrary point in $\Fset_1$. 
	Then $w \leq f(a x)$ holds. 
	To prove the validity of \eqref{eq:seedInq1}, it suffices to show that
	\begin{equation*}
		H : = [f(k a) - f((k - 1) a)] (x - k) + f(k a) - f(a x)
		\geq 0.
	\end{equation*}
	Below we use \cref{cor:sum-pariwise-comparison} to show $H \geq 0$.
	First, observe that $ka(x - k)- (k-1)a(x-k) + k a - a x = 0$, and thus condition \eqref{ccondition} holds.
	If $x \geq k$, then
	\begin{equation*}
		\begin{aligned}
			H = (x - k) [f(k a) - f((k - 1) a)] - [f(a x) - f(k a)] \stackrel{(a)}{\geq} 0,
		\end{aligned}
	\end{equation*}
	where (a) follows from $ax \geq k a > (k - 1) a$ and \cref{cor:sum-pariwise-comparison}.
	Otherwise, $x \leq k - 1$ holds (as $x \in \Z$ and $x < k$) and 
	\begin{equation*}
		\begin{aligned}
			H = [f(k a) - f(a x)] - (k - x) [f(k a) - f((k - 1) a)] \stackrel{(a)}{\geq} 0,
		\end{aligned}
	\end{equation*}
	where (a) follows from $ax \leq (k - 1) a < k a$  and \cref{cor:sum-pariwise-comparison}.
	
	As the dimension of $\conv(\Fset_1)$ equals $2$ and $(f((k-1) a), k - 1)$ and $(f(k a), k)$ 
	are affinely independent points in $\Fset_1$ satisfying \eqref{eq:seedInq1} at equality, \eqref{eq:seedInq1} must be facet-defining for $\conv(\Fset_1)$.
\end{proof}

\cref{fig:convF1} illustrates inequalities \eqref{eq:seedInq1} for the two special cases $\{ (w, x)\in   \R \times \Z\, :\, w \leq -{(x-1.3)}^2,~0 \leq x \leq 3\}$ and $\{ (w, x) \in \R \times \Z \, :\, w \leq \min\{0, 1.5-x\},~0\leq x\leq 3\}$, respectively.
For these two special cases, inequalities \eqref{eq:seedInq1}, together with the box constraints $0 \leq x \leq 3$,
are sufficient to give a complete description of $\conv(\Fset_1)$.
This statement, however, also holds true for the general case, as detailed in the following proposition. 
\begin{figure}[htbp]
	\centering
	\begin{tikzpicture}[scale=\textwidth/9.8cm]
		\centering

		\draw[->, >=latex, line width=0.05mm] (-0.2, 0) -- (3.5, 0) node[below] {$x$};
		\draw[->, >=latex, line width=0.05mm] (0, -3.2) -- (0, 0.5) node[left] {$w$};
		
		\fill[red!30] (0, -0.845) -- (1, -0.045) -- (2, -0.245) -- (3, -1.45) -- (3, -2.9) -- (0, -2.9) -- cycle;
		
		\def\fx(#1){-0.5*(#1-1.3)^2}
		
		\fill[
		pattern={Lines[angle=60, distance=6pt, line width=0.3mm]},
		pattern color=gray!75
		]
		plot[domain=0:3, samples=200,line width=0.5mm]
		(\x,{\fx(\x)})
		-- (0,-0.845) -- (0,-2.9) -- (3, -2.9)  -- (3, -1.45) -- cycle;
		\draw[line width=0.5mm, gray, domain=-0.1:3.1, samples=100] plot (\x, {-0.5*(\x - 1.3)^2});	
		
		\draw[line width=0.5mm] (0, -0.845) -- (0, -3);
		\draw[line width=0.5mm] (1, -0.045) -- (1, -3);
		\draw[line width=0.5mm] (2, -0.245) -- (2, -3);
		\draw[line width=0.5mm] (3, -1.45) -- (3, -3);
		
		\draw[line width=0.05mm] (0, -0.05) -- (0, 0.05);
		\draw[line width=0.05mm] (1, -0.05) -- (1, 0.05);
		\draw[line width=0.05mm] (2, -0.05) -- (2, 0.05);
		\draw[line width=0.05mm] (3, -0.05) -- (3, 0.05);
		
		\draw[line width=0.5mm, blue] (-0.1, -0.925) -- (1.1, 0.035);
		\draw[line width=0.5mm,  blue] (0.9, -0.025) -- (2.1, -0.265);
		\draw[line width=0.5mm, blue] (1.9, -0.12636) -- (3.1, -1.55);

		\draw[line width=0.5mm, blue, -stealth] (0.45, -0.485) -- (0.65, -0.735);
		\draw[line width=0.5mm, blue, -stealth] (1.5, -0.145) -- (1.437212201087252, -0.45893899456374054);
		\draw[line width=0.5mm, blue,-stealth] (2.5, -0.8381818181818181) -- (2.255206242604927, -1.044521383802186);

		\node[above right] at (0, 0) {\small $0$};
		\node[above] at (1, 0) {\small $1$};
		\node[above] at (2, 0) {\small $2$};
		\node[above] at (3, 0) {\small $3$};		

		\draw[->, >=latex, line width=0.05mm] (5.25, 0) -- (8.95, 0) node[below] {$x$};
		\draw[->, >=latex, line width=0.05mm] (5.45, -3.2) -- (5.45, 0.5) node[left] {$w$};
		
		\fill[red!30] (5.45, 0) -- (6.45, 0) -- (7.45, -0.5) -- (8.45, -1.5) -- (8.45, -2.9) -- (5.45, -2.9) -- cycle;

		\fill[
		pattern={Lines[angle=60, distance=6pt, line width=0.3mm]},
		pattern color=gray!75
		]
		(5.45, 0) -- (6.95, 0) -- (7.45, -0.5) -- (8.45, -1.5) -- (8.45, -2.9) -- (5.45, -2.9) -- cycle;
		\draw[line width=0.5mm,gray] (5.45,0) -- (6.96, 0);
		\draw[line width=0.5mm,gray] (6.95, 0) -- (7.45, -0.5);
		
		\draw[line width=0.5mm] (5.45, 0) -- (5.45, -3);
		\draw[line width=0.5mm] (6.45, 0) -- (6.45, -3);
		\draw[line width=0.5mm] (7.45, -0.5) -- (7.45, -3);
		\draw[line width=0.5mm] (8.45, -1.5) -- (8.45, -3);
		
		\draw[line width=0.05mm] (5.45, 0) -- (5.45, 0.05);
		\draw[line width=0.05mm] (6.45, -0.05) -- (6.45, 0.05);
		\draw[line width=0.05mm] (7.45, -0.05) -- (7.45, 0.05);
		\draw[line width=0.05mm] (8.45, -0.05) -- (8.45, 0.05);
		
		\draw[line width=0.5mm, blue] (5.35, 0) -- (6.55, 0);
		\draw[line width=0.5mm, blue] (6.35, 0.05) -- (7.55, -0.55);
		\draw[line width=0.5mm, blue] (7.35, -0.4) -- (8.55, -1.6);

		\draw[line width=0.5mm, blue, -stealth] (5.95, 0) -- (5.95, -0.32015621187164245);
		\draw[line width=0.5mm, blue, -stealth] (6.95, -0.25) -- (6.819296773822016, -0.5114064523559687);
		\draw[line width=0.5mm, blue, -stealth] (7.95, -1) -- (7.723615371546565, -1.2263846284534354);

		\node[above right] at (5.45, 0) {\small $0$};
		\node[above] at (6.45, 0) {\small $1$};

		\node[above] at (7.45, 0) {\small $2$};
		\node[above] at (8.45, 0) {\small $3$};

		\draw[line width=0.5mm] (3.2,-1) -- (3.6,-1);
		\node[right] at (3.7,-1) {\footnotesize Feasible points};
		
		\draw[line width=0.5mm, blue] (3.2,-1.5) -- (3.6,-1.5);
		\node[right] at (3.7,-1.5) {\footnotesize Inequalities \eqref{eq:seedInq1}};
		
		\draw[fill=red!30] (3.2,-2.1) rectangle (3.6,-1.9);
		\node[right] at (3.7,-2) {\footnotesize Convex hull};

		\filldraw[
		pattern={Lines[angle=60, distance=3pt, line width=0.3mm]},
		pattern color=gray!75] (3.2,-2.6) rectangle (3.6,-2.4);
		\node[right] at (3.7,-2.5) {\footnotesize Continuous};
		\node[right] at (3.7,-2.8) {\footnotesize relaxation};
		
		\node[below] at (1.7, -3.2) {\footnotesize $\{ (w, x)\in \R \times \Z\, :\, w \leq {-{(x-1.3)}^2},~0 \leq x \leq 3\}$};
		
		\node[below] at (6.9, -3.225) {\footnotesize $\{ (w, x) \in \R \times \Z \, :\, w \leq \min\{0, 1.5-x\},~0 \leq x \leq 3\}$};
	\end{tikzpicture}
	\caption{Feasible points, inequalities \eqref{eq:seedInq1}, the convex hull, and the continuous relaxation of $\Fset_1$ for two special cases.}
	\label{fig:convF1}
\end{figure}
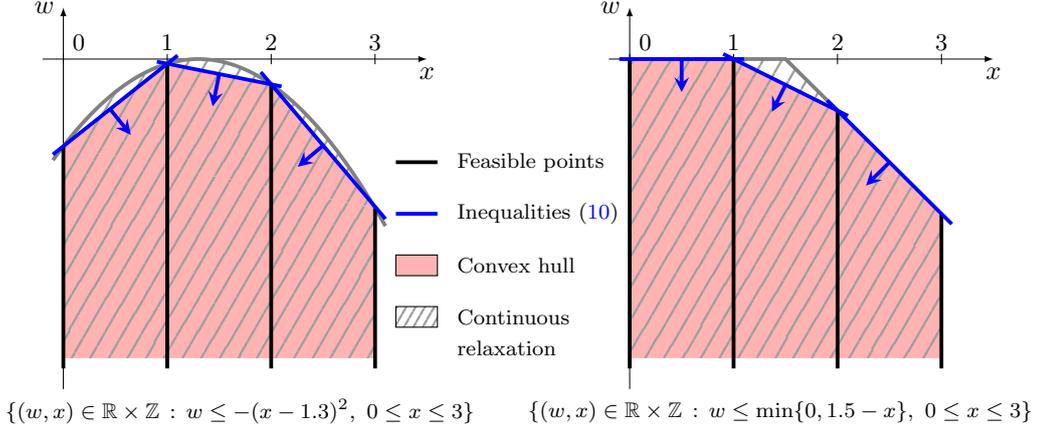

\begin{proposition} 
	Letting
	\begin{equation*}
		\begin{aligned}
			Q = \left\{
			(w, x) \in \R^2 \,:\, 
			w \leq 
			[f(k a) - f((k - 1) a)] (x - k) + f(k a), 
			\ \forall \ k \in [\mu], \ 0 \leq x \leq \mu
			\right\},
		\end{aligned}
	\end{equation*}
	then $\conv(\Fset_1) = Q$.
\end{proposition}
\begin{proof}
	By \cref{prop:2dim-facet}, $\conv(\Fset_1) \subseteq Q$.
	To establish the reverse inclusion, we will show that every point $(\bar{w}, \bar{x}) \in Q$ with $\bar{x} \notin \mathbb{Z}$
	can be represented as a convex combination of two points in $\Fset_1$.
	As $(\bar{w}, \bar{x}) \in Q$, inequality \eqref{eq:seedInq1} with $k= \ceil{\bar{x}}$ must hold, that is
	\begin{equation*}
		\begin{aligned}
			\bar{w} \leq [f(\ceil{\bar{x}} a) - f(\floor{\bar{x}} a)] (\bar{x} - \ceil{\bar{x}}) + f(\ceil{\bar{x}} a) 
			= (\ceil{\bar{x}} - \bar{x}) f(\floor{\bar{x}} a) + (\bar{x} - \floor{\bar{x}}) f(\ceil{\bar{x}} a).
		\end{aligned}
	\end{equation*}
	Let 
	$\hat{w} = (\ceil{\bar{x}} - \bar{x}) f(\floor{\bar{x}} a) + (\bar{x} - \floor{\bar{x}}) f(\ceil{\bar{x}} a)$ and $\lambda = \bar{x} - \floor{\bar{x}}$.
	It follows that $\bar{w} \leq \hat{w}$, $\lambda \in (0, 1)$, and
	$(\hat{w}, \bar{x}) = (1 - \lambda) (f(\floor{\bar{x}} a), \floor{\bar{x}}) + \lambda (f(\ceil{\bar{x}} a), \ceil{\bar{x}})$.
	Therefore, if $\bar{w} = \hat{w}$, then $(\bar{w}, \bar{x})$ is a convex combination of two points $(f(\floor{\bar{x}} a), \floor{\bar{x}})$ and $(f(\ceil{\bar{x}} a), \ceil{\bar{x}})$ in $\Fset_1$ and the statement follows.
	Otherwise, $\bar{w} < \hat{w}$. 
	The statement also follows since $(\bar{w}, \bar{x}) = (\hat{w} - (\hat{w} - \bar{w}), \bar{x})
	= (1 - \lambda) (f(\floor{\bar{x}} a) - (\hat{w} - \bar{w}), \floor{\bar{x}}) + \lambda (f(\ceil{\bar{x}} a) - (\hat{w} - \bar{w}), \ceil{\bar{x}})$
	and $(f(\floor{\bar{x}} a) - (\hat{w} - \bar{w}), \floor{\bar{x}})$, 
	$(f(\ceil{\bar{x}} a) - (\hat{w} - \bar{w}), \ceil{\bar{x}})$ are both in $\Fset_1$.
\end{proof}

In the next two sections, we will use inequalities \eqref{eq:seedInq1} to develop strong valid inequalities for $\conv(\Fset)$ with an arbitrary positive integer $n$.

\section{Single-phase lifted inequalities for $\conv(\Fset)$}\label{sec:lift-one-phase}
Let $s \in [n]$, and $S_0$ and $S_1$ be disjoint subsets of $[n] \backslash s$ such that $s \cup S_0 \cup S_1 = [n]$. 
By fixing $x_i = 0$ for $i \in S_0$ and $x_i = \mu_i$ for $ i \in S_1$, 
we obtain the following two-dimensional restriction of $\Fset$:
\begin{equation*}
\Fset_s(S_0, S_1) : = \left\{ (w,x_s) \in \R \times \Z \,:\, w \leq g \left(a_sx_s\right),  ~ 0 \leq x_s \leq \mu_s \right\},
\end{equation*}
where $g(z) := f\left(z+\sum_{i \in S_1}\mu_i a_i\right)$ is concave (as $f$ is concave).
The set $\Fset_s(S_0, S_1)$ is a form of $\Fset_1$ in \eqref{F1}, and thus 
by \cref{prop:2dim-facet},
the following inequality is facet-defining for $\conv(\Fset_s(S_0, S_1))$:
\begin{equation}\label{eq:seedInq}
		w \leq \rho_s(k)(x_s - k) 
		+ g \left(ka_s\right),
	\end{equation} 
where $k \in [\mu_s]$ and $\rho_s(k) := g(k a_s) - g((k - 1) a_s)$.
We refer to \eqref{eq:seedInq} as the \emph{seed inequality}.
For notational simplicity, we suppress the dependence of $g$ and $\rho_s(k)$ on $S_1$ throughout the paper.

To lift variables $\{x_i\}_{i \in [n] \backslash s}$ into the seed inequality \eqref{eq:seedInq} and obtain strong valid inequalities for $\conv(\Fset)$, 
we consider the following lifting function:
\begin{equation}
	\label{prob:first-part-lifting}
	\begin{aligned}
		\zeta(\delta) : = \max_{w,\, x_s}  \quad 
		& w - \rho_s(k) (x_s - k) - g (ka_s) \\
		\text{s.t.} \quad & w \leq 
		g\left(\delta+a_s x_s\right), \\
		& 0 \leq x_s \le \mu_s, \\
		& w \in \R, \ x_s \in \Z,
	\end{aligned}
\end{equation}
where $\delta \in \R$.
In the following,  we will provide a closed formula for the lifting function $\zeta$ and derive a subadditive approximation.
The latter enables to develop a family of strong valid inequalities for $\conv(\Fset)$ using a single-phase (sequence-independent) lifting procedure.

\subsection{A closed formula for the lifting function $\zeta$}
Observe that for any optimal solution $(\bar{w}, \bar{x}_s)$ of the above lifting problem, 
$\bar{w} = g\left(\delta + a_s \bar{x}_s\right)$ must hold.
Thus, the lifting problem \eqref{prob:first-part-lifting} is equivalent to 
\begin{equation} 
	\tag{\ref{prob:first-part-lifting}'}
	\label{prob:first-part-lifting2}
	\zeta(\delta) = \max_{x_s \in \Z} \left\{
	p(\delta, x_s) \,:\, 0 \leq x_s \leq \mu_s
	\right\},
\end{equation}
where $p(\delta, x_s) := g(\delta + a_s x_s) - \rho_s(k)(x_s - k) - g\left(ka_s\right)$.
In what follows, we also refer to problem \eqref{prob:first-part-lifting2} as the lifting problem.

The following lemma shows that for a fixed $\delta$, $p(\delta, x_s)$ first increases with $x_s$ and then decreases with $x_s$.
\begin{lemma} 
	\label{lemma1}
	For $\delta \in \R$ and $j \in \Z$,
	\begin{itemize}
		\item [(i)] if $\delta \le (k - j - 1)a_s$, then $p(\delta, j+1) \geq p(\delta, j)$;  
		\item [(ii)] if $\delta \ge (k - j - 1)a_s$, then $p(\delta, j+1) \leq p(\delta, j)$.
	\end{itemize}
\end{lemma}
\begin{proof}
	By the definitions of $p$ and $\rho_s(k)$, we have
	\begin{subequations}
		\begin{align*}
			p(\delta, j+1) - p(\delta, j)
			& = g\left(\delta + (j+1)a_s\right) - g\left(\delta + j a_s\right) - \rho_s(k) \\
			&= [g\left(\delta + (j+1)a_s\right) - g\left(\delta + j a_s\right)] 
			- [g(ka_s) - g((k-1)a_s)].
		\end{align*}
	\end{subequations}
	By \cref{lem:slope}, it follows that
	\begin{itemize}
		\item [(a)] if $\delta + (j+1) a_s \leq ka_s$, then $p(\delta, j+1) \geq p(\delta, j)$;
		\item [(b)] if $\delta + (j+1) a_s \geq ka_s$, then $p(\delta, j+1) \leq p(\delta, j)$.
	\end{itemize}
	Thus, statements (i) and (ii) hold.
\end{proof}

Using \cref{lemma1} and the fact that $x_s \in \Z$ and $0 \leq x_s \leq \mu_s$, an optimal solution 
of the lifting problem \eqref{prob:first-part-lifting2} is given by:
\begin{equation*} 
	x_s =  \left\{
	\begin{aligned}
		&\mu_s, && ~\text{if}~\delta < (k-\mu_s-1)a_s,\\
		&k - \ell - 1, && ~\text{if}~\ell a_s \le \delta < (\ell + 1) a_s, \\
		&&& ~~ \ell = k - \mu_s-1,\ldots, k - 1, \\
		&0, && ~\text{if}~\delta \ge ka_s, 
	\end{aligned}
	\right.
\end{equation*}
and hence a closed formula for $\zeta$ is given by:
\begin{equation}\label{calc:zeta}
	\begin{small}
		\zeta(\delta) =  \left\{
		\begin{aligned}
			&g(\delta + \mu_sa_s)+(k-\mu_s)\rho_s(k)-g(ka_s), &&\text{if}~\delta < (k-\mu_s  - 1)a_s,\\
			&g \left(\delta + (k - \ell - 1) a_s \right) + (\ell + 1) \rho_s(k)-g(ka_s), &&\text{if}~\ell a_s \le \delta < (\ell + 1) a_s, \\
			& && ~ \ell = k - \mu_s-1,\ldots, k - 1, \\
			&g \left(\delta\right) + k\rho_s(k)-g(ka_s), &&\text{if}~\delta \ge ka_s.
		\end{aligned}
		\right.
	\end{small}
\end{equation}
By simple computations, $\zeta$ is continuous on $\R$.

\subsection{Subadditive approximation for the lifting function $\zeta$}
The exact lifting function $\zeta$ is, however, not subadditive on $\R$ in general (an illustrative example will be provided in \cref{ex1} later in this section).
To construct a subadditive lifting function \cite{Atamturk2004,Richard2011,Wolsey1977} (that enables the sequence-independent lifting of the fixed variables $\{x_i\}_{i \in [n]\backslash s}$), 
we drop the box constraint $0\leq x_s \leq \mu_s$ from the lifting problem \eqref{prob:first-part-lifting2} and consider the following relaxation:
\begin{equation}\label{prob:first-part-lifting-ub}
	Z(\delta) : = \max_{x_s \in \Z} \ p(\delta, x_s).
\end{equation}
By \cref{lemma1} and $x_s \in \Z$ in \eqref{prob:first-part-lifting-ub}, 
we can give a closed formula for $Z(\delta)$:
\begin{equation}\label{funZ}
	Z(\delta) =  g(\delta + (k - \ell - 1) a_s) + (\ell + 1) \rho_s(k) - g(ka_s),
	\ \text{if} \ \ell a_s \le \delta < (\ell + 1)a_s, \ \ell \in \Z.
\end{equation}
The following proposition establishes the subadditivity of function $Z$.
\begin{proposition}
	\label{lem:Z-subadditive-on-R}
	Function $Z$ is subadditive on $\R$; that is, $	Z(\delta_1) + Z(\delta_2) \geq Z(\delta_1 + \delta_2)$ holds for all $\delta_1, \delta_2 \in \R$.
\end{proposition}
\begin{proof}
	Given $\delta_1, \delta_2 \in \R$, let $\ell_1, \ell_2, \ell \in \Z$ 
	be such that $\ell_1 a_s \leq \delta_1 < (\ell_1 + 1)a_s$, $\ell_2 a_s \leq \delta_2 < (\ell_2 + 1)a_s$, and 
	$\ell a_s \leq \delta_1 + \delta_2 < (\ell+1)a_s$.
	It follows that $\ell \in \{ \ell_1 + \ell_2, \ell_1+\ell_2+1 \}$.
	We prove that $Z(\delta_1) + Z(\delta_2) - Z(\delta_1 + \delta_2) \geq 0$ by considering the two cases (i) $\ell = \ell_1 + \ell_2$ and (ii) $\ell = \ell_1 + \ell_2 + 1$, separately.
	\begin{itemize}
		\item [(i)] {$\ell = \ell_1 + \ell_2$}. 
		From the closed formula for $Z$ in \eqref{funZ}, we have
		\begin{equation*}
			\begin{aligned}
				& Z(\delta_1) + Z(\delta_2) - Z(\delta_1 + \delta_2) \\
				& \quad 
				= g(\delta_1+(k - \ell_1 - 1) a_s) 
				+ g(\delta_2+(k - \ell_2 - 1) a_s) \\
				& \quad\quad 
				- g(\delta_1 + \delta_2 + (k - \ell_1-\ell_2 - 1) a_s) 
				+ \rho_s(k) - g(k a_s) \\
				& \quad \stackrel{(a)}{=} g(\delta_1 + (k - \ell_1 - 1) a_s) 
				+ g(\delta_2 + (k - \ell_2 - 1) a_s) \\
				& \quad\quad 
				- g(\delta_1 + \delta_2 + (k -  \ell_1-\ell_2 - 1) a_s) 
				- g((k-1)a_s) \\
				& \quad = [g(\delta_1 - \ell_1 a_s + (k-1)a_s) - g((k-1)a_s)] \\
				& \quad \quad 
				- \left[ g(\delta_1 - \ell_1 a_s + \delta_2 - \ell_2  a_s + (k-1)a_s) 
				- g(\delta_2 - \ell_2  a_s + (k-1)a_s)
				\right]
				\stackrel{(b)}{\geq} 0,
			\end{aligned}
		\end{equation*}
		where (a) follows from $\rho_s(k) = g(k a_s) - g((k - 1) a_s)$,
		and (b) follows from $\delta_1 - \ell_1 a_s \ge 0$,  $\delta_2 - \ell_2 a_s \ge 0$, and 
		\cref{lem:slope}.
		\item [(ii)] {$\ell = \ell_1 + \ell_2 + 1$}. 
		From the closed formula for $Z$ in \eqref{funZ}, we have
		\begin{equation*}
			\begin{aligned}
				& Z(\delta_1) + Z(\delta_2) - Z(\delta_1 + \delta_2) \\
				& \quad = g(\delta_1 + (k - \ell_1 - 1) a_s) 
				+ g(\delta_2 + (k - \ell_2 - 1) a_s) \\
				& \quad\quad 
				- g(\delta_1 + \delta_2 + (k - \ell_1-\ell_2 - 1 - 1) a_s) 
				- g(ka_s) \\
				& \quad = \left[
				g(ka_s+\delta_1 - (\ell_1 + 1) a_s )
				- g( ka_s+ \delta_1 - (\ell_1 + 1) a_s + \delta_2 - (\ell_2+1) a_s ) 
				\right] \\
				& \quad\quad - \left[
				g(ka_s)
				- g(ka_s+\delta_2 - (\ell_2 + 1) a_s )
				\right]
				\stackrel{(a)}{\geq} 0,
			\end{aligned}
		\end{equation*}
		where (a) follows from $\delta_1 - (\ell_1 + 1) a_s \le 0$, $\delta_2 - (\ell_2 + 1) a_s \le 0$, and \cref{lem:slope}.\qedhere 
	\end{itemize}
\end{proof}

\subsection{Single-phase lifted inequalities}
Since problem \eqref{prob:first-part-lifting-ub} is a relaxation of the lifting problem \eqref{prob:first-part-lifting2}, $\zeta(\delta)\leq Z(\delta)$ must hold for $\delta \in \R$.
This, together with \cref{lem:Z-subadditive-on-R}, implies that
\begin{equation*}
	\begin{aligned}
		\zeta\left( 
		\sum_{i\in S_0} a_i x_i 
		+ \sum_{i\in S_1} - a_i (\mu_i - x_i) 
		\right)
		& \le 
		Z \left( 	\sum_{i\in S_0} a_i x_i 
		+ \sum_{i\in S_1} - a_i (\mu_i - x_i) 
		\right)\\
		& \le \sum_{i\in S_0} Z(a_i) x_i + \sum_{i \in S_1} Z(-a_i) (\mu_i - x_i).
	\end{aligned}
\end{equation*}
Therefore, we obtain the single-phase lifted inequality
\begin{equation}\label{eq:single-phase}
	w \leq \sum_{i \in S_0}Z(a_i)x_i + \sum_{ i \in S_1}Z(-a_i)(\mu_i - x_i)
	+ \rho_s(k) (x_s - k) + g(k a_s),
\end{equation}
which is valid for $\conv(\Fset)$.
The following classic result illustrates the strength of the above single-phase lifted inequalities.
\begin{theorem}{\cite[Theorem 5]{Atamturk2004}}
	\label{thm:subadd}
	If $Z(a_i) = \zeta(a_i)$ for all $i \in S_0$ and $Z(-a_i) = \zeta(-a_i)$ for all $i \in S_1$, then
	inequality \eqref{eq:single-phase} is facet-defining for $\conv(\Fset)$.
\end{theorem}
\noindent By \eqref{calc:zeta} and \eqref{funZ}, we have
$Z(\delta) = \zeta(\delta)$ for $\delta$ with $(k - \mu_s-1) a_s \le \delta \le ka_s$.
Thus,
\begin{corollary}\label{cor:facet-cond} 
	If $0 \le a_i \le ka_s$ for all $i \in S_0$ and $ (k - \mu_s-1) a_s \le - a_i \le 0$ for all $ i \in S_1$, 
	then
	inequality \eqref{eq:single-phase} is facet-defining for $\conv(\Fset)$.
\end{corollary}

The following example shows that inequality \eqref{eq:single-phase} can still be facet-defining even if {the conditions in \cref{cor:facet-cond} do not hold. }
\begin{example}\label{ex1}
	Consider 
	$\Fset = \{(w,x) \in \R \times \{0,1\}^4 \,:\,
	w \leq f(x_1+2x_2+2x_3+3x_4)\}$,
	where $f$ is a concave function. 
	Let $s = 1$, $S_0 = \{2, 3\}$, $S_1 = \{4\}$, and $k=1$. Then inequality \eqref{eq:seedInq} reads
	\begin{equation*}
		w \leq \rho_1(1) x_1 + g(0),
	\end{equation*}
	where $g(z) = f\left(z + \sum_{i \in S_1} a_i \mu_i\right) = f(z + 3)$
	and $\rho_1(1) = g(1) - g(0)$.
	The exact lifting function $\zeta$ in \eqref{calc:zeta} reduces to 
	\begin{equation}\label{zeta-example}
		\zeta(\delta) =  \left\{
		\begin{aligned}
			&g(\delta + 1)-g(1), &&\text{if}~\delta <  -1,\\
			&g \left( \delta - \ell \right) +\ell (g(1)-g(0)) - g(0), 
			&&\text{if}~\ell \le \delta < (\ell + 1), ~ \ell = -1,0, \\
			&g \left(\delta\right)  - g(0), &&\text{if}~\delta \ge 1.
		\end{aligned}
		\right.
	\end{equation}
	Note that $\zeta$ is not subadditive on $\R$ in general. 
	For instance, if $f(z) = -e^{-(z-3)}$, then $g(z) = -e^{-z}$  
	and $\zeta(2) + \zeta(-1) = g(2) - g(1) = e^{-1}- e^{-2}<  1 - e^{-1}=g(1) - g(0) = \zeta(1)$.
	
	The subadditive approximation $Z$ in \eqref{funZ} reads
	$$Z(\delta) = g(\delta - \ell) + \ell (g(1)-g(0)) - g(0),
	\ \text{if} \ \ell \leq \delta < (\ell + 1),~\ell \in \mathbb{Z}.$$
	By $Z(\ell) = \ell (g(1) - g(0)) $ for $\ell \in \mathbb{Z}$,
	the single-phase lifted inequality \eqref{eq:single-phase} can be written as
	\begin{equation}\label{eq:facet-ex}
		w \leq g(0) + (g(1)-g(0)) \left(
		x_1 + 2x_2 + 2x_3 - 3 (1 - x_4)
		\right).
	\end{equation}
	{Since $a_i = 2 > 1 = k a_s$ for $i \in S_0 = \{2 ,3\}$, the conditions in \cref{cor:facet-cond} are not satisfied.}
	However, it can be verified that there are five affinely independent points that lie on the face defined by \eqref{eq:facet-ex}:
	\begin{align*}
		& (g(0), 0, 0, 0, 1), ~(g(1), 1, 0, 0, 1),~ (g(0), 1, 1, 0, 0) \\
		&  (g(0), 1, 0, 1, 0),~ (g(1), 0, 1, 1, 0).
	\end{align*}
	Thus, \eqref{eq:facet-ex} defines a facet of $\conv(\Fset)$.
\end{example}

\section{Two-phase lifted inequalities for $\conv(\Fset)$}\label{sec:twosteplifting}
In this section, we develop two more families of strong valid inequalities for $\conv(\Fset)$.
The motivation is that although the exact lifting function $\zeta$ is not subadditive on $\R$,
it is subadditive on $\R_+$ and on $\R_-$ separately.
\begin{proposition}
	\label{liftzeta-new}
	The function $\zeta$ is subadditive on $\R_+$ and on $\R_-$ separately.
\end{proposition}
\begin{proof}
	The proof can be found in \cref{proof_4.1}.
\end{proof}
\noindent \cref{liftzeta-new} enables to develop strong valid inequalities for $\conv(\Fset)$ by lifting variables $\{x_i\}_{i \in [n] \backslash s}$ into the seed inequality \eqref{eq:seedInq} in a two-phase manner where in the first phase, variables $\{x_i\}_{i \in S_0}$ or variables $\{x_i\}_{i \in S_1}$ are lifted using the exact lifting function $\zeta$ and in the second phase, the remaining variables are lifted using some subadditive lifting functions.
In \cref{subsec:ineq-class-1}, we present a family of two-phase lifted inequalities by lifting the seed inequality \eqref{eq:seedInq} first with variables $\{x_i\}_{i \in S_0}$ and then with variables $\{x_i\}_{i \in S_1}$, while in \cref{subsec:ineq-class-2}, we present another family of inequalities where the variables are lifted in the reversed manner.

\subsection{Two-phase lifted inequalities: Type I }\label{subsec:ineq-class-1}
Using \cref{liftzeta-new}, we can first lift variables $\{x_i\}_{i \in S_0}$ into the seed inequality  \eqref{eq:seedInq}, and obtain the 
following inequality:
\begin{equation}\label{eq:low-dim-facet-1-new}
	w \leq \sum_{i \in S_0}\zeta(a_i) x_i + \rho_s(k) (x_s - k) + g(ka_s).
\end{equation}
As inequality \eqref{eq:seedInq} is facet-defining for $\conv(\Fset_s(S_0, S_1))$, we obtain that
\begin{proposition}\label{prop:low-dim-facet-new}
	Inequality \eqref{eq:low-dim-facet-1-new} is facet-defining for $\conv(\Fset_{s}(\varnothing, S_1))$, where 
	\begin{equation*}
		\Fset_s(\varnothing, S_1) := \left\{
		(w, x) \in \R \times \Z^{|S_0 \cup s|} \,:\, 
		w \leq g\left( \sum_{i \in S_0 \cup s} a_i x_i
		\right),
		\ 0 \leq x_i \leq \mu_i, \ \forall \ i \in S_0 \cup s 
		\right\}.
	\end{equation*}
\end{proposition}
	
To lift inequality \eqref{eq:low-dim-facet-1-new} with variables $\{x_i\}_{i \in S_1}$, we consider the corresponding lifting function:
\begin{align*}
	\eta(\delta) : = \max_{w, \, x_{S_0 \cup s}} \quad 
	& w - \sum_{i \in S_0}\zeta(a_i)x_i + (k - x_s) \rho_s(k) - g(ka_s) \\
	\text{s.t.} \quad 
	& w \leq g\left(\delta + \sum_{i \in S_0 \cup s}a_ix_i\right), \\
	& 0 \leq x_i \le \mu_i, \ \forall \ i \in S_0 \cup s, \label{cons:bound-xs}\\
	& w \in \R, \ x_i \in \Z, \ \forall \ i \in S_0 \cup s,
\end{align*}
where $\delta \in \R_-$.
This problem  is equivalent to the following \MINLP problem: 
\begin{equation}\label{prob:twophase1}
	\begin{aligned}
		\eta(\delta) = \max_{x_{S_0 \cup s}} \quad & g\left(\delta + \sum_{i \in S_0 \cup s}a_ix_i\right) - \sum_{i \in S_0}\zeta(a_i)x_i + (k - x_s) \rho_s(k) - g(ka_s) \\
		\text{s.t.} \quad
		& 0 \leq x_i \le \mu_i, \ \forall \ i \in S_0 \cup s, \\
		& x_i \in \Z, \ \forall \ i \in S_0 \cup s.
	\end{aligned}
\end{equation}
Giving an explicit description of $\eta$ is challenging, 
even for the special case that $f(z)= \min\{0, b-z\}$ (i.e., $g(z)= \min\{ 0, b - z - \sum_{i \in S_1} \mu_i a_i\}$) \cite{Atamturk2003}.
To bypass this difficulty, we attempt to derive lower and upper subadditive approximations for $\eta$, thus enabling to develop strong valid inequalities  for $\conv(\Fset)$.
The approach we take is as follows.
We first consider a restriction of problem \eqref{prob:twophase1}
obtained by setting $x_i = 0$ for all $i \in S_0$ with $a_i < ka_s$, and provide a closed-form optimal solution for this restriction, thereby obtaining 
a lower bound for $\eta$.
This lower bound enables to provide a closed formula for $\eta$ for some special cases.
For the general case, we solve a relaxation of problem \eqref{prob:twophase1} (obtained by removing $x_s \leq \mu_s $ from the problem) to derive an upper bound for $\eta$.
We show that the proposed lower and upper bounds for $\eta$ are subadditive on $\R_-$.
Using this result, we can develop two families of sequence-independent lifted inequalities for $\conv(\Fset)$ and establish sufficient conditions under which  they are facet-defining for $\conv(\Fset)$.

\subsubsection{A lower bound for $\eta$}\label{subsubsec:S0-equals-S0+}
By setting $x_i = 0$ for all $i \in S_0$ with $a_i < k a_s$, we obtain the following restriction of problem \eqref{prob:twophase1}:
\begin{equation}\label{prob:restriction1}
	\begin{aligned}
		\eta^L(\delta) : = \max_{x_{S_0^+ \cup s}} \quad & g\left(\delta + \sum_{i \in S_0^+ \cup s}a_ix_i\right) - \sum_{i \in S_0^+}\zeta(a_i)x_i + (k - x_s) \rho_s(k) - g(ka_s) \\
		\text{s.t.} \quad
		& 0 \leq x_i \le \mu_i, \ \forall \ i \in S_0^+ \cup s, \\
		& x_i \in \Z, \ \forall \ i \in S_0^+ \cup s,
	\end{aligned}
\end{equation}
where we let $S_0^+ := \{ i \in S_0 \,:\, a_i \geq ka_s\}$.
Solving problem \eqref{prob:restriction1} provides a lower bound for problem \eqref{prob:twophase1}.
In order to derive a closed-form optimal solution of problem \eqref{prob:restriction1},
we first consider a special case in which $\mu_i = 1$ for all $i \in S_0^+$. 
For any subset $\Lambda \subseteq S_0^+$ and $\gamma\in \mathbb{Z}_+$, let $a(\Lambda) = \sum_{ i \in \Lambda} a_i$
and $h(\Lambda, \gamma, \delta)$ denote the objective value of problem \eqref{prob:restriction1} at point $x_{S^+_0 \cup s}$ 
where $x_s = \gamma$, $x_i = 1$ for $i \in \Lambda$, and $x_i = 0$ for $i \in S_0^+ \backslash \Lambda$, i.e., 
\begin{equation}\label{hdef}
	\begin{small}
		\begin{aligned}
			h(\Lambda, \gamma, \delta) & = g(\delta + a(\Lambda) + \gamma a_s) - \sum_{i \in \Lambda}\zeta(a_i) + (k - \gamma) \rho_s(k) - g(ka_s)\\
			& =g(\delta + a(\Lambda) + \gamma a_s) - \sum_{i \in \Lambda}\zeta(a_i) + (k- \gamma)[g(ka_s)- g((k-1)a_s)] - g(ka_s).
		\end{aligned}
	\end{small}
\end{equation}
Here, $\rho_s(k) = g(ka_s) - g((k-1)a_s)$.
Thus, problem \eqref{prob:restriction1} with $\mu_i = 1$ for all $i \in S_0^+$ is equivalent to 
\begin{equation}\label{setopt1}
	\eta^L(\delta) =\max_{\Lambda, \,\gamma} \left\{ h(\Lambda, \gamma, \delta) \,:\, \Lambda \subseteq S_0^+, \ 0 \leq \gamma \leq \mu_s, \ \gamma \in \Z \right\}.	
\end{equation}
The following lemma states that if $\Lambda \subseteq S_0^+$ is fixed, maximizing $h(\Lambda, \gamma, \delta)$ over $\gamma \in \mathbb{Z}\cap [0,\mu_s]$ is easy.
\begin{lemma}\label{lem:gamma-choose}
	Let $\delta \in \R_-$ and $\Lambda \subseteq S_0^+$. 
	Then $\gamma_{\Lambda, \delta}$ is an optimal solution of 
	\begin{equation*}
		\max_{\gamma} \left\{ h(\Lambda, \gamma, \delta) \,:\, 0 \leq \gamma \leq \mu_s, \ \gamma \in \Z \right\},
	\end{equation*}
	where
	\begin{equation}\label{def-gamma}
		\gamma_{\Lambda, \delta} = \left\{
		\begin{aligned}
			& 0, && ~\text{if}~ k a_s - a(\Lambda) < \delta, \\
			& {\ell}, && ~\text{if}~ (k - {\ell} - 1)a_s - a(\Lambda) < \delta \leq ( k - {\ell})a_s - a(\Lambda), \ {\ell} = 0, \ldots, \mu_s, \\
			& \mu_s, && ~\text{if}~ \delta \leq (k - \mu_s - 1)a_s - a(\Lambda).
		\end{aligned}
		\right.
	\end{equation}
\end{lemma}

\begin{proof}
	By the definition of $h$ in \eqref{hdef}, we have
	\begin{equation*}
		\begin{aligned}
			h(\Lambda, \gamma + 1, \delta) - h(\Lambda, \gamma, \delta)
			& =  [g(\delta + a(\Lambda) + (\gamma + 1) a_s) -  g(\delta + a(\Lambda) + \gamma a_s)] \\
			& \quad - [g(ka_s) - g((k-1)a_s)].
		\end{aligned}
	\end{equation*}
	By \cref{lem:slope}, it follows that 
	\begin{itemize}
		\item [(a)] if $\delta + a(\Lambda) + \gamma a_s \leq (k - 1) a_s$, then $h(\Lambda, \gamma + 1, \delta) \geq h(\Lambda, \gamma, \delta)$;
		\item [(b)] if $\delta + a(\Lambda) + \gamma a_s \geq (k - 1) a_s$, then $h(\Lambda, \gamma + 1, \delta) \leq h(\Lambda, \gamma, \delta)$.
	\end{itemize}
	This, together with $0 \leq \gamma \leq \mu_s$ and $\gamma \in \Z$, yields the desired result.
\end{proof}
\noindent The following two observations follow from the definition of $\gamma_{\Lambda,\delta}$ in \eqref{def-gamma}.
\begin{observation}\label{obs:gamma-range-sum}
	Let $\delta \in \R_-$ and $\Lambda \subseteq S_0^+$.
	Then the following statements are true: 
	(i) if $\gamma_{\Lambda, \delta} \ge 1$, then $\delta + a(\Lambda) + \gamma_{\Lambda, \delta} a_s  \le  k a_s$;
	and (ii) if $\gamma_{\Lambda, \delta} \le \mu_s - 1$, 
	then $\delta + a(\Lambda) + \gamma_{\Lambda, \delta} a_s > (k-1)a_s$.
\end{observation}
\begin{observation}\label{obs:gamma}
	For any $\Lambda_1, \Lambda_2 \subseteq S_0^+$ with $a(\Lambda_1) \leq a(\Lambda_2)$,
	it follows that $\gamma_{\Lambda_1, \delta} \geq \gamma_{\Lambda_2, \delta}$.
\end{observation}

Using \cref{lem:gamma-choose}, we can equivalently simplify problem \eqref{setopt1} as 
\begin{equation}\label{setopt2}
	\eta^L(\delta) =\max_{\Lambda \subseteq S_0^+}  h(\Lambda, \gamma_{\Lambda, \delta},\delta).
\end{equation}
\cref{lem:remove-element,lem:alternate-element,lem:add-element} below characterize structural properties of optimal solutions for problem \eqref{setopt2}.
Specifically, they provide sufficient conditions under which a given solution $\Lambda\subseteq S_0^+ $ can potentially be improved through a removal, swap, or addition operation.
These results will enable to provide a closed-form optimal solution for problem \eqref{setopt2} (and thus also for problem \eqref{prob:restriction1}).

\begin{lemma}\label{lem:remove-element}
	Let $\delta \in \R_-$, $\Lambda \subseteq S_0^+$, $i \in \Lambda$,  and $\bar{\Lambda} = \Lambda \backslash i$. 
	If $\delta + a(\Lambda) + \gamma_{\Lambda, \delta} a_s > a_i$, 
	then $h(\bar{\Lambda}, \gamma_{\bar{\Lambda}, \delta}, \delta) \geq h(\Lambda, \gamma_{\Lambda, \delta}, \delta)$.
\end{lemma}

\begin{proof}
	Since $\delta + a(\Lambda) + \gamma_{\Lambda, \delta} a_s > a_i $ and  $a_i \geq ka_s$ (as $i \in \Lambda \subseteq S_0^+$), we have $\delta + a(\Lambda) + \gamma_{\Lambda, \delta} a_s > ka_s$, 
	which together with \cref{obs:gamma-range-sum} (i), implies that $\gamma_{\Lambda, \delta} = 0$. 
	By the definition of $\zeta$ in \eqref{calc:zeta}, $a_i \geq ka_s$, and $\rho_s(k) = g(k a_s) - g((k - 1) a_s)$,
	we have $\zeta(a_i)=g(a_i) + k \rho_s(k)-g(ka_s)=g(a_i) + k [g(ka_s) - g((k-1)a_s) ]- g(k a_s)$. 
	It then follows from the definition of $h$ in \eqref{hdef} and $\gamma_{\Lambda, \delta} = 0$ that
	\begin{equation} \label{obj-diff3}
		\begin{aligned}
			H & := h(\bar{\Lambda}, \gamma_{\bar{\Lambda}, \delta}, \delta) - h(\Lambda, \gamma_{\Lambda, \delta}, \delta)\\
			& = h(\bar{\Lambda}, \gamma_{\bar{\Lambda}, \delta}, \delta) - h(\Lambda, 0, \delta)\\
			& = g(\delta + a(\bar{\Lambda}) + \gamma_{\bar{\Lambda}, \delta}a_s ) - g(\delta + a(\Lambda)) + \zeta(a_i) 
			-\gamma_{\bar{\Lambda}, \delta} (g(ka_s) - g((k-1)a_s)) \\
			&  = [g(a_i)  - g((k-1)a_s)]
			+ [(k - 1 - \gamma_{\bar{\Lambda}, \delta})(g(ka_s) - g((k-1)a_s))] \\
			& \quad
			- [g(\delta + a(\Lambda))
			- g(\delta + a(\bar{\Lambda}) + \gamma_{\bar{\Lambda}, \delta}a_s)].
		\end{aligned}
	\end{equation}
	We shall complete the proof by using \cref{cor:sum-pariwise-comparison} to prove $H\geq 0$.
	Observe that 
	\begin{equation*}
		a_i - (k-1)a_s + (k-1 - \gamma_{\bar{\Lambda},\delta}) a_s - [(\delta + a(\Lambda)) -(\delta + a(\bar{\Lambda}) + \gamma_{\bar{\Lambda}, \delta}a_s)]=0
	\end{equation*}
	and thus condition \eqref{ccondition} in \cref{cor:sum-pariwise-comparison} holds.
	We next show 
	\begin{equation}
		\label{tmpineq2-1}
		\gamma_{\bar{\Lambda}, \delta} \le k - 1.
	\end{equation}
	Indeed, if $\gamma_{\bar{\Lambda}, \delta}=0$, \eqref{tmpineq2-1} holds naturally as $k\geq 1$; otherwise,  from \cref{obs:gamma-range-sum} (i), we have 
	$\delta + a(\bar{\Lambda}) + \gamma_{\bar{\Lambda}, \delta} a_s \leq ka_s$, which
	together with $a_i < \delta + a(\Lambda) + \gamma_{\Lambda, \delta} a_s  = \delta + a(\Lambda)$ (i.e., $\delta +  a(\bar{\Lambda}) >0$) 
	and $ \gamma_{\bar{\Lambda}, \delta}  \in \mathbb{Z}$, implies \eqref{tmpineq2-1}. 
	It follows that $\gamma_{\bar{\Lambda}, \delta} \le k - 1 \leq \mu_s - 1$;
	thus, by \cref{obs:gamma-range-sum} (ii), we have
	\begin{equation}\label{ineq-ii-1}
		\delta + a(\bar{\Lambda}) + \gamma_{\bar{\Lambda}, \delta} a_s > (k-1)a_s.
	\end{equation}
	From \eqref{tmpineq2-1} and $a_i \geq ka_s > (k-1)a_s$, we obtain 
	\begin{equation}
		\label{tmpineq2-2}
		\delta + a(\bar{\Lambda}) + \gamma_{\bar{\Lambda}, \delta} a_s 
		\leq \delta + a(\bar{\Lambda}) +(k-1)a_s 
		< \delta + a(\bar{\Lambda}) + a_i=  \delta + a(\Lambda).
	\end{equation}
	From \eqref{tmpineq2-1}, \eqref{ineq-ii-1}, \eqref{tmpineq2-2}, 
	and $\delta + a(\Lambda) 
	= \delta + a(\Lambda) + \gamma_{\Lambda, \delta} a_s 
	> a_i \geq ka_s$, conditions (i) and (ii) in \cref{cor:sum-pariwise-comparison} must hold.
	Therefore, $H \geq 0$. \qedhere
\end{proof}

\begin{lemma}\label{lem:alternate-element}
	Let $\delta \in \R_-$, $\Lambda \subseteq S_0^+$, $i \in \Lambda$, $j \in S_0^+ \backslash \Lambda$ with $a_i \leq a_j$, and 
	$\bar{\Lambda} = \Lambda \cup j \backslash i$. 
	If $\delta + a(\Lambda) + \gamma_{\Lambda, \delta} a_s \leq a_i$, then $h(\bar{\Lambda}, \gamma_{\bar{\Lambda}, \delta}, \delta) \geq h(\Lambda, \gamma_{\Lambda, \delta}, \delta)$.
\end{lemma}
\begin{proof}
	By the definition of $\zeta$ in \eqref{calc:zeta} and $a_i, a_j \ge ka_s$ (as $i, j \in S_0^+$),
	we have $\zeta(a_j) - \zeta(a_i) = g(a_j) - g(a_i)$,
	which together with the definition of $h$ in \eqref{hdef}, implies 
	\begin{equation}\label{obj-diff-2}
		\begin{aligned}
			H & := h(\bar{\Lambda}, \gamma_{\bar{\Lambda}, \delta}, \delta) - h(\Lambda, \gamma_{\Lambda, \delta}, \delta) \\
			& = 
			g\left( \delta + a(\bar{\Lambda}) + \gamma_{\bar{\Lambda}, \delta} a_s \right)
			- g\left( \delta + a(\Lambda) + \gamma_{\Lambda, \delta} a_s \right)
			- \zeta(a_j) + \zeta(a_i) \\
			& \quad - (\gamma_{\bar{\Lambda}, \delta} - \gamma_{{\Lambda}, \delta})
			(g(ka_s) - g((k-1)a_s)) \\
			& = 
			g\left( \delta + a(\bar{\Lambda}) + \gamma_{\bar{\Lambda}, \delta} a_s \right)
			- g\left( \delta + a(\Lambda) + \gamma_{\Lambda, \delta} a_s \right)
			- g(a_j) + g(a_i) \\
			& \quad - (\gamma_{\bar{\Lambda}, \delta} - \gamma_{{\Lambda}, \delta})
			(g(ka_s) - g((k-1)a_s)).
		\end{aligned}
	\end{equation}
	From $a_i \leq a_j$, we have $a(\Lambda) \leq a(\Lambda) + a_j - a_i = a(\bar{\Lambda})$. 
	By  \cref{obs:gamma},
	it follows that $\gamma_{\Lambda, \delta} \geq \gamma_{\bar{\Lambda}, \delta} $.
	Below we will prove $H\geq 0$ by considering the two cases (i) $\gamma_{\Lambda, \delta} = \gamma_{\bar{\Lambda}, \delta}$ and (ii) $\gamma_{\Lambda, \delta} \ge \gamma_{\bar{\Lambda}, \delta} + 1$, separately.
	\begin{itemize}
		\item [(i)] $\gamma_{\Lambda, \delta} = \gamma_{\bar{\Lambda}, \delta}$. 
		Then $\Delta :=
		\delta + a(\bar{\Lambda}) + \gamma_{\bar{\Lambda}, \delta} a_s
		- (\delta + a(\Lambda) + \gamma_{\Lambda, \delta} a_s) 
		= a_j - a_i \ge 0$.
		Thus,
		\begin{equation*}
			H =  [g\left( \delta + a(\Lambda) + \gamma_{\Lambda, \delta} a_s + \Delta \right)  
			- g(\delta +  a(\Lambda) + \gamma_{\Lambda, \delta} a_s)]
			- [g(a_i + \Delta) - g(a_i)] 
			\stackrel{(a)}{\geq} 0,
		\end{equation*}
		where (a) follows from $\delta + a(\Lambda) + \gamma_{\Lambda, \delta} a_s \le a_i$ and \cref{lem:slope}.
		\item [(ii)] $\gamma_{\Lambda, \delta} \ge \gamma_{\bar{\Lambda}, \delta} + 1$.
		Then $\gamma_{\Lambda, \delta} \geq 1$ and $\gamma_{\bar{\Lambda}, \delta} \leq \mu_s - 1$.
		By \cref{obs:gamma-range-sum}, we obtain 
		\begin{equation}\label{lb-bar-Lambda}
			\delta + a(\Lambda) + \gamma_{\Lambda, \delta} a_s \leq ka_s
			~ \text{and} ~
			\delta + a(\bar{\Lambda}) + \gamma_{\bar{\Lambda}, \delta} a_s > (k-1) a_s.
		\end{equation}
		Regrouping \eqref{obj-diff-2}, we obtain 
		\begin{equation*}
			\begin{aligned}
				H = & [(\gamma_{\Lambda, \delta} - \gamma_{\bar{\Lambda}, \delta} - 1)(g(ka_s) - g((k-1)a_s))]
				+ [g(ka_s) - g\left( \delta + a(\Lambda) + \gamma_{\Lambda, \delta} a_s \right)] \\
				&  + [g(\delta + a(\bar{\Lambda}) + \gamma_{\bar{\Lambda}, \delta} a_s ) - g((k-1)a_s)]
				- [g(a_j) - g(a_i)].
			\end{aligned}
		\end{equation*}
		Below we use \cref{cor:sum-pariwise-comparison} to show $H \geq 0$.
		Observe that 
		\begin{small}
			\begin{equation*}
				(\gamma_{\Lambda, \delta} - \gamma_{\bar{\Lambda}, \delta} - 1) a_s 
				+ [k a_s - (\delta + a(\Lambda) + \gamma_{\Lambda, \delta} a_s) ]
				+ [\delta +  a(\bar{\Lambda}) + \gamma_{\bar{\Lambda}, \delta} a_s - (k - 1) a_s ]
				- ( a_j - a_i ) = 0,
			\end{equation*}
		\end{small}
		and thus condition \eqref{ccondition} in \cref{cor:sum-pariwise-comparison} holds.
		From $\delta + a(\Lambda) + \gamma_{\Lambda, \delta} a_s \le a_i$ and {$\gamma_{\Lambda, \delta} \ge \gamma_{\bar{\Lambda}, \delta}+1$}, we have 
		\begin{equation}\label{ineq-lem3.7-i-1}
			\delta + a(\bar{\Lambda}) + \gamma_{\bar{\Lambda}, \delta} a_s 
			=  a_j + (\delta + a(\Lambda) + \gamma_{\Lambda, \delta} a_s
			- a_i) - (\gamma_{\Lambda, \delta} - \gamma_{\bar{\Lambda}, \delta}) a_s \leq a_j.
		\end{equation}
		From \eqref{lb-bar-Lambda}, \eqref{ineq-lem3.7-i-1}, and $a_j \ge a_i \ge ka_s \ge (k-1)a_s$, 
		conditions (i) and (ii) in \cref{cor:sum-pariwise-comparison} must hold. 
		Therefore, $H \ge 0$. \qedhere
	\end{itemize}
\end{proof}

\begin{lemma}\label{lem:add-element}
	Let $\delta \in \R_-$, $\Lambda \subseteq S_0^+$, $i \in S_0^+ \backslash \Lambda$,  and $\bar{\Lambda} = \Lambda \cup i$. 
	If $\delta + a(\bar{\Lambda}) + \gamma_{\bar{\Lambda}, \delta} a_s \leq a_i$, 
	then $h(\bar{\Lambda}, \gamma_{\bar{\Lambda}, \delta}, \delta) \geq h(\Lambda, \gamma_{\Lambda, \delta}, \delta)$.
\end{lemma}

\begin{proof}
	By the definition of $\zeta$ in \eqref{calc:zeta}, $a_i \geq ka_s$ (as $i \in S_0^+$), and $\rho_s(k) = g(k a_s) - g((k - 1) a_s)$, 
	we have $\zeta(a_i)=g(a_i) + k \rho_s(k)-g(ka_s)=g(a_i) + k [g(ka_s) - g((k-1)a_s) ]- g(k a_s)$, which together with the definition of $h$ in \eqref{hdef}, implies
	\begin{equation} \label{obj-diff4}
		\begin{aligned}
			H & :=  h(\bar{\Lambda}, \gamma_{\bar{\Lambda}, \delta}, \delta) - h(\Lambda, \gamma_{\Lambda, \delta}, \delta) \\
			& = g(\delta + a(\bar{\Lambda}) + \gamma_{\bar{\Lambda}, \delta}a_s)
			- g(\delta + a(\Lambda) + \gamma_{\Lambda, \delta}a_s) -  \zeta(a_i) \\
			& \quad -(\gamma_{\bar{\Lambda}, \delta} - \gamma_{\Lambda, \delta})(g(ka_s) - g((k-1)a_s)) \\
			& =g(\delta + a(\bar{\Lambda}) + \gamma_{\bar{\Lambda}, \delta}a_s)
			- g(\delta + a(\Lambda) + \gamma_{\Lambda, \delta}a_s)- g(a_i)  + g(ka_s) \\
			& \quad 
			+ (\gamma_{\Lambda, \delta} - \gamma_{\bar{\Lambda}, \delta} - k)(g(ka_s) - g((k-1)a_s)).
		\end{aligned}
	\end{equation}
	Below we will prove $H \geq 0$ by considering the two cases (i) $\gamma_{\Lambda, \delta} - \gamma_{\bar{\Lambda}, \delta} = k$
	and (ii) $\gamma_{\Lambda, \delta} - \gamma_{\bar{\Lambda}, \delta} \neq k$, separately.
	\begin{itemize}
		\item [(i)]$\gamma_{\Lambda, \delta} - \gamma_{\bar{\Lambda}, \delta} = k$. 
		Then $ \Delta :=  (\delta + a(\bar{\Lambda}) + \gamma_{\bar{\Lambda}, \delta}a_s) - (\delta + a(\Lambda) + \gamma_{\Lambda, \delta}a_s) =  a_i - ka_s \geq 0$.
		Thus,
		\begin{equation*}
			H =[g(\delta + a(\bar{\Lambda}) + \gamma_{\bar{\Lambda}, \delta}a_s )
			-g(\delta + a(\bar{\Lambda}) + \gamma_{\bar{\Lambda}, \delta}a_s - \Delta)]-  [g(a_i)  - g(a_i - \Delta)]  
			\stackrel{(a)}{\geq} 0, 
		\end{equation*}
		where (a) follows from $\delta + a(\bar{\Lambda}) + \gamma_{\bar{\Lambda}, \delta} a_s \le a_i$ and \cref{lem:slope}.
		\item [(ii)] $\gamma_{\Lambda, \delta} - \gamma_{\bar{\Lambda}, \delta} \neq k$.
		We shall use  \cref{cor:sum-pariwise-comparison} to prove $H \geq 0$.
		Observe that 
		\begin{equation*}
			\delta + a(\bar{\Lambda}) + \gamma_{\bar{\Lambda}, \delta}a_s
			- (\delta + a(\Lambda) + \gamma_{\Lambda, \delta}a_s)
			- a_i + ka_s  
			+ (\gamma_{\Lambda, \delta} - \gamma_{\bar{\Lambda}, \delta} - k)a_s= 0,
		\end{equation*}
		and thus condition \eqref{ccondition} in \cref{cor:sum-pariwise-comparison} holds.
		We further consider the following two cases.
		\begin{itemize}
			\item [(ii.1)] $\gamma_{\Lambda, \delta} - \gamma_{\bar{\Lambda}, \delta} \geq k + 1$.
			Then $\gamma_{\Lambda, \delta} \geq 1$ and $\gamma_{\bar{\Lambda}, \delta} \leq \mu_s - 1$.
			By \cref{obs:gamma-range-sum}, we obtain
			\begin{equation}\label{cond-4.8}
				\delta + a(\Lambda) + \gamma_{\Lambda, \delta}a_s \leq ka_s
				~\text{and}~
				\delta + a(\bar{\Lambda}) + \gamma_{\bar{\Lambda}, \delta}a_s > (k - 1) a_s.
			\end{equation}
			Regrouping \eqref{obj-diff4}, we obtain 
			\begin{equation*}
				\small
				\begin{aligned}
					& H = 
					[(\gamma_{\Lambda, \delta} - \gamma_{\bar{\Lambda}, \delta} - (k+1))
					(g(ka_s) - g((k-1)a_s)) ] 
					+ [g(ka_s) -g(\delta + a(\Lambda)+ \gamma_{\Lambda, \delta}a_s)] \\
					& \qquad
					+ [g(\delta + a(\bar{\Lambda}) + \gamma_{\bar{\Lambda}, \delta}a_s) - g((k-1)a_s)] - [g(a_i)  - g(ka_s)].
				\end{aligned}
			\end{equation*}
			Together with $\gamma_{\Lambda, \delta} - \gamma_{\bar{\Lambda}, \delta} \geq k + 1$, 
			$a_i \geq \delta + a(\bar{\Lambda}) + \gamma_{\bar{\Lambda}, \delta}$, $a_i \geq k a_s$, and \eqref{cond-4.8},
			conditions (i) and (ii) in \cref{cor:sum-pariwise-comparison} must hold.
			Therefore, $H \geq 0$.
			\item [(ii.2)] $\gamma_{\Lambda, \delta} - \gamma_{\bar{\Lambda}, \delta} \le k - 1$.
			Then
			\begin{equation}
				\label{tmpineq1-1}
				(\delta + a(\bar{\Lambda}) + \gamma_{\bar{\Lambda}, \delta}a_s) 
				- (\delta + a(\Lambda) + \gamma_{\Lambda, \delta}a_s) 
				= a_i -  (\gamma_{\Lambda, \delta} - \gamma_{\bar{\Lambda}, \delta}) a_s\geq a_i -(k-1)a_s.
			\end{equation}
			Together with $\delta + 	a(\bar{\Lambda}) + \gamma_{\bar{\Lambda}, \delta} \le a_i$, we have
			\begin{equation}
				\label{tmpineq1-2}
				\delta + a(\Lambda) + \gamma_{\Lambda, \delta}a_s \leq (k-1)a_s.
			\end{equation}
			From \cref{obs:gamma-range-sum} (ii),  $\gamma_{\Lambda, \delta} = \mu_s$ must hold, and thus $\gamma_{\bar{\Lambda}, \delta} \ge \gamma_{\Lambda, \delta} - (k-1) = (\mu_s - k) + 1 \ge 1$.
			Applying \cref{obs:gamma-range-sum} (i) yields 
			\begin{equation}\label{ineq-Delta-neg-lt-k-2}
				\delta + a({\bar{\Lambda}}) + \gamma_{{\bar{\Lambda}}, \delta} a_s \le ka_s.
			\end{equation}
			We now regroup \eqref{obj-diff4} as 
			\begin{equation*}
				\begin{aligned}
					& H = [g(\delta + a(\bar{\Lambda})+ \gamma_{\bar{\Lambda}, \delta}a_s)
					- g(\delta + a(\Lambda) + \gamma_{\Lambda, \delta}a_s) ]  \\
					& \qquad 
					-[g(a_i)  - g((k-1)a_s)] 
					-(k - 1 -  (\gamma_{\Lambda, \delta} -  \gamma_{\bar{\Lambda}, \delta}) )[ g(ka_s) - g((k-1)a_s)) ].
				\end{aligned}
			\end{equation*}
			Combining $a_i \geq ka_s > (k-1)a_s$, \eqref{tmpineq1-1}, \eqref{tmpineq1-2}, and \eqref{ineq-Delta-neg-lt-k-2}, conditions (i) and (ii) in \cref{cor:sum-pariwise-comparison} must hold.
			Therefore, $H \geq 0$. \qedhere
		\end{itemize}
	\end{itemize}
\end{proof}

We are now ready to provide a closed-form optimal solution of the problem \eqref{setopt2}.
Without loss of generality, we assume $S_0^+ =[r]$ (where $r \in \mathbb{Z}_+$) with $a_1 \geq a_2 \geq \cdots \geq a_r$.
Define the partial sums $A_0 = 0$ and $A_{i} = \sum_{j =1}^{i} a_{j}$ for $ i \in [r]$.
\begin{theorem}\label{thm:special-case-optsol}
	For $\delta \in \R_-$, define 
	\begin{equation}\label{def-t-xs}
		t = \left\{ \
		\begin{aligned}
			& i + 1 , && ~\text{if}~ - A_{i + 1} < \delta \leq -A_{i}, \ 
			i = 0 ,\ldots, r -1, \\
			& r, && ~\text{if}~
			\delta \leq -A_{r}.
		\end{aligned}
		\right.
	\end{equation}
	Then $[t] \subseteq S_0^+ $ is an optimal solution of problem \eqref{setopt2}.
\end{theorem}

\begin{proof}
	Let $\Lambda \subseteq S_0^+ $ be an optimal solution of  problem \eqref{setopt2} 
	and
	\begin{equation}
		\label{eq:imaxdef}
		i_{\max} = \left\{\begin{array}{ll}
			\max_{i \in \Lambda} i,& \text{if}~\Lambda \neq \varnothing, \\
			0, & \text{otherwise}.
		\end{array}\right.
	\end{equation}
	If $i_{\max} = 0$, it follows that $\Lambda =\varnothing$ and we let $a_{i_{\max}}=+\infty$.
	\begin{itemize}
		\item [(i)] If $\delta + a(\Lambda) + \gamma_{\Lambda, \delta} a_s > a_{i_{\max}}$, 
		then from \cref{lem:remove-element}, it follows that
		$h(\bar{\Lambda}, \gamma_{\bar{\Lambda}, \delta}, \delta) \ge h(\Lambda, \gamma_{\Lambda, \delta}, \delta)$ where $\bar{\Lambda} = \Lambda \backslash i_{\max}$. 
		This implies that $\bar{\Lambda}$ is also an optimal solution.
		Therefore, we can assume without loss of generality that $\delta + a(\Lambda) + \gamma_{\Lambda, \delta} a_s \le a_{i_{\max}}$ in the following. 
		\item [(ii)]  If $\Lambda \subsetneq [i_{\max}]$, then $i_{\max} \geq 1$ and we let $j \in [i_{\max}] \backslash \Lambda$.
		By $a_1 \geq \cdots \geq a_r $ and \eqref{eq:imaxdef}, we have $a_j \ge a_{i_{\max}}$.
		Using step (i), it follows that $\delta + a(\Lambda) + \gamma_{\Lambda, \delta} a_s \leq a_{i_{\max}}$, which together with  \cref{lem:alternate-element}, 
		indicates that $h(\bar{\Lambda}, \gamma_{\bar{\Lambda}, \delta}, \delta) \ge h(\Lambda, \gamma_{\Lambda, \delta}, \delta)$, where  $\bar{\Lambda} = \Lambda \cup j \backslash i_{\max}$.
		Therefore, $\bar{\Lambda}$ is also an optimal solution.
	\end{itemize}
	Recursively applying the above two steps, 
	we can obtain an optimal solution $\Lambda \subseteq S_0^+ = [r]$ satisfying $\Lambda = [i_{\max}]$ 
	and $\delta + a(\Lambda) + \gamma_{\Lambda, \delta} a_s \leq a_{i_{\max}}$.
	
	Now, if $i_{\max}=t$, we are done.
	Otherwise, if $i_{\max} \ge t+1$, then $t \leq i_{\max}-1 \leq r - 1$. 
	From the definition of $t$ in \eqref{def-t-xs}, we have $\delta > - A_t$, 
	which, together with $a(\Lambda) = A_{i_{\max}}$ and $\gamma_{\Lambda, \delta} \geq 0$, indicates 
	$\delta + a(\Lambda) + \gamma_{\Lambda, \delta} a_s 
	> - A_t + a(\Lambda) + \gamma_{\Lambda, \delta} a_s 
	\geq - A_t + A_{i_{\max}} = \sum_{i = t + 1}^{i_{\max}} a_i \ge a_{i_{\max}}$, 
	contradicting with $\delta + a(\Lambda) + \gamma_{\Lambda, \delta} a_s \leq a_{i_{\max}}$.
	Thus, $i_{\max}\leq t-1$ must hold.
	Let $\bar{\Lambda} = \Lambda \cup \{i_{\max}+ 1\}$. 
	Then, the following inequality
	\begin{equation}\label{tmp-ineq}
		\delta + a(\bar{\Lambda}) + \gamma_{\bar{\Lambda}, \delta}a_s \leq a_{i_{\max}+1}
	\end{equation}
	must hold.
	Indeed, if $\gamma_{\bar{\Lambda}, \delta} \ge 1$, \eqref{tmp-ineq} follows from \cref{obs:gamma-range-sum} (i) and $ka_s \leq a_{i_{\max}+1}$ (as $i_{\max}+1 \in S_0^+$).
	Otherwise, $\gamma_{\bar{\Lambda}, \delta}=0$.
	From the definition of $t$ in \eqref{def-t-xs}, we have $\delta \le- A_{t-1}$.
	Together with ${i_{\max}} \le {t-1}$, we have 
	$\delta + a(\bar{\Lambda}) + \gamma_{\bar{\Lambda}, \delta} a_s 
	=\delta + a(\bar{\Lambda}) 
	\le - A_{t-1} + a(\bar{\Lambda}) 
	= - A_{t-1} + A_{i_{\max}} + a_{i_{\max} + 1} \le a_{i_{\max} + 1}$.
	From \cref{lem:add-element} and \eqref{tmp-ineq}, we obtain 
	$h(\bar{\Lambda}, \gamma_{\bar{\Lambda}, \delta}, \delta) \geq h(\Lambda, \gamma_{\Lambda, \delta}, \delta)$, which implies 
	that $\bar{\Lambda}$ is also an optimal solution of problem \eqref{setopt2}.
	Recursively applying the above argument, we can show that $[t]$ is an optimal solution of problem \eqref{setopt2}.	
\end{proof}

For the special case where $\mu_i = 1$ for all $i \in S_0^+$, \cref{thm:special-case-optsol} yields
\begin{equation*}
	\eta^L(\delta) = h\left([t], \gamma_{[t], \delta}, \delta\right),
\end{equation*}
where $t$ is defined in \eqref{def-t-xs}.
Using the definitions of $h$ and $\gamma_{\Lambda, \delta}$ in  \eqref{hdef} and \eqref{def-gamma}, respectively, we can derive the following closed formula for $\eta^L$:
\begin{equation}
	\label{eta-explicit}
	\begin{footnotesize}
		\eta^L(\delta) = \left\{
		\begin{aligned}
			& g(\delta+A_{i + 1}) - \sum_{j = 1}^{i + 1}\zeta(a_j) + k\rho_s(k) - g(ka_s), \\
			& \hspace{4.5cm} ~\text{if}~ k a_s - A_{i + 1} < \delta \leq - A_{i}, \\
			& \hspace{4.5cm} ~~ i = 0,\ldots,r - 1, \\
			& g(\delta+A_{i+1} + \ell a_s) - \sum_{j = 1}^{i+1}\zeta(a_j) + (k - \ell) \rho_s(k) - g(ka_s), \\
			& \hspace{4.5cm} ~\text{if}~ (k - \ell - 1) a_s - A_{i+1} < \delta \leq (k - \ell)a_s - A_{i+1}, \\
			& \hspace{4.5cm} ~~ i = 0,\ldots,r - 1 \ \text{and} \ \ell = 0,\ldots, k-1, \\
			& \hspace{4.5cm} ~~ \text{or} \ i = r-1 \ \text{and} \ \ell = k,\ldots, \mu_s, \\
			& g(\delta+A_r + \mu_s a_s) - \sum_{j = 1}^{r}\zeta(a_j) + (k - \mu_s) \rho_s(k) - g(ka_s), \\
			& \hspace{4.5cm} ~\text{if}~ \delta \leq (k - \mu_s - 1)a_s - A_r. \\
		\end{aligned}
		\right.
	\end{footnotesize}
\end{equation}
We now consider the general case where $\mu_i > 1$ may hold for some $i \in S_0^+$. 
In such a case, we can binarize the general integer variables and also derive a closed formula for $\eta^L$.
In particular, for each $i \in S_0^+$, let $x_i = x_{i1} + \cdots + x_{i \mu_i}$ where $x_{i q}\leq1$ for  $q = 1, \ldots, \mu_i$.
Define $A_{i q} = \sum_{j=1}^{i-1}\mu_j a_j + q a_i$ for $i \in [r]$ and $q \in \{0,1,\ldots,\mu_i\}$.
Then, the closed formula for $\eta^L$  with arbitrary positive integers $\{\mu_i\}_{i \in [n]}$ is given by:
\begin{equation}
	\label{etaL}
	\begin{footnotesize}
		\eta^L(\delta) = \left\{
		\begin{aligned}
			&g(\delta+A_{(i + 1) q}) - \sum_{j=1}^{i}\mu_j \zeta(a_j) 
			- q \zeta(a_{i + 1}) + k\rho_s(k) - g(ka_s), \\
			& \hspace{3.5cm}
			~\text{if}~ k a_s - A_{(i + 1) q} < \delta \leq -A_{(i + 1) (q-1)}, \\
			& \hspace{3.5cm} 
			~~ i = 0, \ldots, r - 1 \ \text{and} \ q = 1, \ldots, \mu_{i+1}, \\
			&g(\delta+A_{(i + 1) q} + \ell a_s) - \sum_{j=1}^{i}\mu_j \zeta(a_j) 
			- q \zeta(a_{i + 1})+ (k - \ell) \rho_s(k) - g(ka_s), \\
			& \hspace{3.5cm}
			~\text{if}~ (k-\ell - 1)a_s - A_{(i + 1) q} < \delta \leq 
			(k-\ell )a_s -A_{(i + 1) q}, \\
			& \hspace{3.5cm}
			~~ i = 0, \ldots, r - 1, \ q = 1, \ldots, \mu_{i+1}, \ \text{and} \ \ell = 0,\ldots, k-1,\\
			& \hspace{3.5cm}
			~~ \text{or} \ i = r - 1, \ q = \mu_{r}, \ \text{and} \ \ell = k,\ldots, \mu_s,\\
			&g(\delta+A_{r\mu_r} + \mu_s a_s) - \sum_{j=1}^{r}\mu_j \zeta(a_j) 
			+ (k - \mu_s) \rho_s(k) - g(ka_s), \\
			& \hspace{3.5cm}
			~\text{if}~ \delta \leq 
			(k - \mu_s - 1) a_s -A_{r \mu_r}. \\
		\end{aligned}
		\right.
	\end{footnotesize}
\end{equation}

We end this subsection by identifying sufficient conditions under which a closed formula for $\eta$ can be derived using the closed formula for $\eta^{L}$ in \eqref{etaL}.
\begin{itemize}
	\item [(i)] $S_0 = S_0^+$. 
	In this case, problem \eqref{prob:restriction1} reduces to the original lifting problem \eqref{prob:twophase1}, and thus  $\eta(\delta) = \eta^L(\delta)$ holds for all $\delta \in \R_-$, where $\eta^L$ is defined by \eqref{etaL}.
	\item [(ii)] $k = 1$. By the definition of $\zeta$ in \eqref{calc:zeta}, we have $\zeta(\delta) = g(\delta) - g(0)$ for $\delta \in \R_+$.
	Together with $\rho_s(1) = g(a_s)-g(0)$,
	we can simplify the lifting problem \eqref{prob:twophase1} to
	\begin{equation}\label{prob:k=10}
		\begin{aligned}
			\eta(\delta) = \max_{x_{S_0\cup s}} \quad 
			& g\left(\delta + \sum_{i \in S_0 \cup s}a_ix_i\right) - \sum_{i \in S_0 \cup s} \left( g(a_i) - g(0) \right) x_i - g(0) \\
			\text{s.t.} \quad
			& 0 \leq x_i \le \mu_i, \ \forall \ i \in S_0 \cup s, \\
			& x_i \in \Z, \ \forall \ i \in S_0 \cup s.
		\end{aligned}
	\end{equation}

	Letting $\bar{s} \in \argmin_i \{a_i \,:\, i \in S_0 \cup s\}$ and  $\bar{S}_0 := S_0 \cup s \backslash \bar{s}$, we can rewrite problem \eqref{prob:k=10} as 
	\begin{equation}\label{prob:k=1}
		\begin{aligned}
			\eta(\delta) = \max_{x_{\bar{S}_0 \cup \bar{s}}} \quad 
			& g\left(\delta + \sum_{i \in \bar{S}_0 \cup \bar{s}}a_ix_i\right) - \sum_{i \in \bar{S}_0 \cup \bar{s}} \left( g(a_i) - g(0) \right) x_i - g(0) \\
			\text{s.t.} \quad
			& 0 \leq x_i \le \mu_i, \ \forall \ i \in \bar{S}_0 \cup \bar{s}, \\
			& x_i \in \Z, \ \forall \ i \in \bar{S}_0 \cup \bar{s}.
		\end{aligned}
	\end{equation}

	In this problem,  $\bar{S}_0^+ = \{i \in \bar{S}_0 \,:\, a_i \geq k a_{\bar{s}}\} = \{i \in \bar{S}_0 \,:\, a_i \geq a_{\bar{s}}\} = \bar{S}_0$, and by \eqref{etaL}, 
		\begin{equation}\label{etaL-k=1}
		\begin{footnotesize}
			\bar{\eta}^L(\delta) = \left\{
			\begin{aligned}
				&g(\delta+A_{(i + 1) q}) - \sum_{j=1}^{i}\mu_j \zeta(a_j) 
				- q \zeta(a_{i+1}) - g(0), \\
				& \hspace{1cm} \text{if}~ - A_{(i+1) q} < \delta \leq -A_{(i+1) (q-1)},
				~ i = 0, \ldots, r \ \text{and} \ q = 1, \ldots, \mu_{i+1}, \\
				&g(\delta+A_{(r + 1) \mu_{r + 1}}) - \sum_{j=1}^{r + 1}\mu_j \zeta(a_j) - g(0), \\
				 & \hspace{1cm}
				 \text{if}~ \delta \leq - A_{(r + 1) \mu_{r + 1} },
			\end{aligned}
			\right.
		\end{footnotesize}
	\end{equation}
	where, by abuse of notation, we assume that $\bar{S}_0 \cup \bar{s} = [r + 1]$ (where $r \in \mathbb{Z}_{+}$) with $a_1 \geq a_2 \geq \cdots \geq a_{r + 1}$ and redefine $A_{i q} = \sum_{j = 1}^{i - 1} \mu_j a_j + q a_i$ for $i \in [r + 1]$ and $q \in \{0, 1, \ldots, \mu_i\}$.
	As a result, it follows from case (i) that $\eta(\delta) = \bar{\eta}^L(\delta)$ holds for all $\delta \in \R_-$. 

\end{itemize}
Summarizing the above discussion, we obtain that
\begin{proposition}\label{lem:eta}
	The following two statements hold.
	\begin{itemize}
		\item [(i)] If $S_0 = S_0^+$, then $\eta(\delta) = \eta^L(\delta)$ for $\delta \in \R_-$, where $\eta^L$ is defined in \eqref{etaL};
		\item [(ii)] If $k = 1$, then $\eta(\delta) = \bar{\eta}^L(\delta)$ for $\delta \in \R_-$, where $\bar{\eta}^L$ is defined in \eqref{etaL-k=1}.
	\end{itemize}
\end{proposition}

\subsubsection{An upper bound for $\eta$} \label{subsubsec:S_0-neq-S_0+}
Next, we consider the following relaxation of problem \eqref{prob:twophase1}:
\begin{equation}\label{prob:relaxation}
	\begin{aligned}
		\eta^U(\delta) : = \max\limits_{x_{S_0 \cup s}} \quad 
		& g\left(\delta + \sum_{i \in S_0 \cup s}a_ix_i\right)
		- \sum_{i \in S_0}\zeta(a_i)x_i + (k - x_s) \rho_s(k) - g(ka_s) \\
		\text{s.t.} \quad 
		& 0 \leq x_i \leq \mu_i, \ \forall \ i \in S_0, ~  0 \leq x_s, \\
		& x_i \in \Z, \ \forall \ i \in S_0 \cup s,
	\end{aligned}
\end{equation}
obtained by removing the upper bound constraint $x_s \leq \mu_s$ from \eqref{prob:twophase1}.
Solving this relaxation provides an upper bound $\eta^U(\delta)$ for problem \eqref{prob:twophase1}.
To identify a closed formula for $\eta^U$, we need the following result.
\begin{theorem}\label{thm:prob-simplify}
	The problem \eqref{prob:relaxation} has an optimal solution $x^*_{S_0 \cup s}$ for which $x^*_i = 0$ holds for all $i \in S_0 \backslash S^+_0$, where  $S_0^+ =\{ i \in S_0 \,:\, a_i \geq ka_s\}$.
\end{theorem}
The proof of \cref{thm:prob-simplify} is detailed further.
Using \cref{thm:prob-simplify}, we can set $x^*_i = 0$ for all $i \in S_0 \backslash S^+_0$ in problem \eqref{prob:relaxation}, thereby simplifying problem \eqref{prob:relaxation} as:
\begin{equation}\label{lift2}
\begin{aligned}
		\eta^U(\delta)  = \max\limits_{x_{S_0^+ \cup s}} \quad 
	& g\left(\delta + \sum_{i \in S_0^+ \cup s}a_ix_i\right) - \sum_{i \in S_0^+ }\zeta(a_i)x_i + (k - x_s) \rho_s(k) - g(ka_s) \\
	\text{s.t.} \quad 
	& 0 \leq x_i \leq \mu_i, \ \forall \ i \in S_0^+, ~0 \leq x_s, \\
	& x_i \in \Z, \ \forall \ i \in S_0^+ \cup s.
\end{aligned}
\end{equation}
Problem \eqref{lift2} is a special case of problem \eqref{prob:restriction1} in which $\mu_s = + \infty$.
Therefore, using the results in \cref{subsubsec:S0-equals-S0+}, we can derive 
the closed formula for $\eta^U$ from that for $\eta^L$ in \eqref{etaL} as follows:
\begin{equation}\label{etaU}
	\begin{footnotesize}
		\eta^U(\delta) = \left\{ \ 
		\begin{aligned}
			&g(\delta+A_{(i+1) q}) - \sum_{j=1}^{i}\mu_j \zeta(a_j) 
			- q \zeta(a_{i+1}) + k \rho_s(k) - g(ka_s), \\
			& \hspace{3cm}
			~\text{if}~ k a_s - A_{(i+1) q} < \delta \leq -A_{(i+1) (q-1)}, \\
			& \hspace{3cm} 
			~~ i = 0,\ldots,r-1 \ \text{and} \ q = 1, \ldots, \mu_{i+1}, \\
			&g(\delta+A_{(i+1) q} + \ell a_s) - \sum_{j=1}^{i}\mu_j \zeta(a_j) 
			- q \zeta(a_{i+1}) + (k - \ell) \rho_s(k) - g(ka_s), \\
			& \hspace{3cm}
			~\text{if}~ (k-\ell - 1)a_s - A_{(i+1) q} < \delta \leq 
			(k-\ell )a_s -A_{(i+1) q}, \\
			& \hspace{3cm}
			~~ i = 0,\ldots,r-1, \  q = 1, \ldots, \mu_{i+1}, \ \text{and} \ \ell = 0,\ldots, k-1,\\
			& \hspace{3cm}
			~~ \text{or} \ i = r-1, \  q = \mu_{r}, \ \text{and} \ \ell = k,k+1,\ldots,\\
		\end{aligned}
		\right.
	\end{footnotesize}
\end{equation}
where we recall that $S_0^+ = [r]$ with $a_1 \geq a_2 \geq \cdots \geq a_r$ (without loss of generality) and $A_{i q} = \sum_{j=1}^{i-1}\mu_j a_j + q a_i$ for $i \in [r]$ and $q \in \{0,1,\ldots,\mu_i\}$.

\begin{proof}[{Proof of \cref{thm:prob-simplify}}]
	It suffices to consider the case where $\mu_i = 1$ for all $i \in S_0$ (otherwise, we can apply the binarization technique in \cref{subsubsec:S0-equals-S0+} to transform problem \eqref{prob:relaxation} into a problem that satisfies this assumption).  
	Similar to problem \eqref{prob:restriction1}, we use $h(\Lambda, \gamma, \delta)$ to denote the objective value of problem \eqref{prob:relaxation} at point $x_{S_0 \cup s}$
	where $x_s = \gamma$, $x_i = 1$ for $i \in \Lambda \subseteq S_0$, 
	and $x_i = 0$ for $i \in S_0 \backslash \Lambda$.
	Then problem \eqref{prob:relaxation} with $\mu_i = 1$ for all $i \in S_0$ is equivalent to
	\begin{equation}\label{opt2}
		\max\limits_{\Lambda, \, \gamma} \left\{ h(\Lambda, \gamma, \delta) \,:\,
		\Lambda \subseteq S_0, \ \gamma \in \Z_+ \right\},
	\end{equation}
	where $h(\Lambda, \gamma, \delta)$ is defined in \eqref{hdef}.
	For a fixed $\Lambda \subseteq S_0$, using a similar argument as that in the proof of 
	\cref{lem:gamma-choose}, we obtain $\max_{\gamma \in \Z_+} h(\Lambda, \gamma, \delta) = h(\Lambda, \gamma_{\Lambda, \delta}, \delta)$, where
	\begin{equation}\label{def-gamma2}
		\gamma_{\Lambda, \delta} = \left\{
		\begin{aligned}
			& 0, && \text{ if } k a_s - a(\Lambda) < \delta, \\[5pt]
			& \ell, && \text{ if } (k - \ell - 1)a_s - a(\Lambda) < \delta \leq (k - \ell) a_s - a(\Lambda), \ \ell = 0, 1, \ldots.
		\end{aligned}
		\right.
		\end{equation}
	As a result, 	
	\begin{fact}\label{fact1}
		For any $\Lambda_1, \Lambda_2 \subseteq S_0$ with $a(\Lambda_1) \leq a(\Lambda_2)$, it follows that $\gamma_{\Lambda_1, \delta} \geq \gamma_{\Lambda_2, \delta}$.
	\end{fact}
	\begin{fact}\label{fact2}
		For $\delta \in \R_-$ and $\Lambda \subseteq S_0$, we have:
		(i) $\delta + a(\Lambda) + \gamma_{\Lambda, \delta} a_s > (k - 1) a_s$;
		and (ii) if $\gamma_{\Lambda, \delta} \geq 1$, then $\delta + a(\Lambda) + \gamma_{\Lambda, \delta} a_s \leq k a_s$.
	\end{fact}
	\noindent Using \eqref{def-gamma2}, we can equivalently rewrite problem \eqref{opt2} as
	\begin{equation}\label{setopt3}
		\max\limits_{\Lambda \subseteq S_0} h(\Lambda, \gamma_{\Lambda, \delta}, \delta).
	\end{equation}
	To prove the statement in \cref{thm:prob-simplify}, it suffices to show that for any given $\Lambda\subseteq S_0$ with $i \in \Lambda\backslash S_0^+$ and $\bar{\Lambda} := \Lambda \backslash i$, it follows 
	\begin{equation}
		\label{eq:cond-4.10}
		H := h(\bar{\Lambda}, \gamma_{\bar{\Lambda}, \delta}, \delta) - h(\Lambda, \gamma_{\Lambda, \delta}, \delta) \geq 0.
	\end{equation}	
	
	By $S_0^+ = \{i \in S_0 \,:\, a_i \geq ka_s\}$ and $i \in \Lambda \backslash S_0^+$, we have $0 \leq a_i < ka_s$, and thus there exists some
	$\ell \in \{0,\dots, k-1\}$ for which $\ell a_s \leq a_i < (\ell + 1)a_s$ holds.
	By the definition of $\zeta$ in \eqref{calc:zeta} and $\rho_s(k) = g(k a_s) - g((k - 1) a_s)$, we have 
	\begin{align*}
		\zeta(a_i) &= g(a_i + (k - \ell - 1)a_s) + (\ell + 1) \rho_s(k) - g(ka_s)  \\
		& = g(a_i + (k - \ell - 1)a_s) + (\ell + 1) (g(ka_s) - g((k-1)a_s)) - g(k a_s).
	\end{align*}
	Together with the definitions of $h$ and $H$ in \eqref{hdef} and \eqref{eq:cond-4.10}, respectively, we obtain
	\begin{equation}\label{obj-diff5}
		\begin{aligned}
			H 
			& = g(\delta + a(\bar{\Lambda}) + \gamma_{\bar{\Lambda}, \delta} a_s) 
			- g(\delta + a(\Lambda) + \gamma_{\Lambda, \delta} a_s) 
			+ \zeta(a_i) \\
			& \quad - (\gamma_{\bar{\Lambda}, \delta} - \gamma_{\Lambda, \delta}) 
			(g(ka_s) - g((k-1)a_s)) \\
			& = g(\delta + a(\bar{\Lambda}) + \gamma_{\bar{\Lambda}, \delta} a_s)
			- g(\delta + a(\Lambda) + \gamma_{\Lambda, \delta} a_s) \\
			& \quad 
			+ g(a_i + (k - \ell - 1)a_s) - g(ka_s) \\
			& \quad 
			+ (\ell + 1 - (\gamma_{\bar{\Lambda}, \delta} - \gamma_{\Lambda, \delta})) 
			(g(ka_s) - g((k-1)a_s)).
		\end{aligned}
	\end{equation}
	To prove $H \geq 0$, we first show that
	\begin{equation}\label{gamma-diff-condition}
		\gamma_{\bar{\Lambda}, \delta} - \gamma_{\Lambda, \delta} \leq \ell + 1.
	\end{equation}
	From $a(\bar{\Lambda}) = a(\Lambda) - a_i \leq a(\Lambda)$ and \cref{fact1}, we have $\gamma_{\bar{\Lambda}, \delta} \geq \gamma_{\Lambda, \delta}$.
	If $\gamma_{\bar{\Lambda}, \delta} = 0$, then by $\gamma_{\Lambda, \delta} \geq0$, it follows $\gamma_{\Lambda, \delta} = 0$, and thus
	\eqref{gamma-diff-condition} must hold.
	Otherwise,  $\gamma_{\bar{\Lambda}, \delta} \geq 1$,
	which together with \cref{fact2} (ii), implies that
	$\delta + a(\bar{\Lambda}) + \gamma_{\bar{\Lambda}, \delta} a_s \leq ka_s$.
	From \cref{fact2} (i), we have $\delta + a(\Lambda) + \gamma_{\Lambda, \delta} a_s > (k - 1)a_s$.
	As a result,
	\begin{align*}
		(k - 1)a_s < \delta + a(\Lambda) + \gamma_{\Lambda, \delta} a_s 
		& = (\delta + a(\bar{\Lambda}) + \gamma_{\bar{\Lambda}, \delta} a_s) + a_i + (\gamma_{\Lambda, \delta} - \gamma_{\bar{\Lambda}, \delta}) a_s \\
		& \leq ka_s + a_i + (\gamma_{\Lambda, \delta} - \gamma_{\bar{\Lambda}, \delta}) a_s.
	\end{align*}
	This, together with $a_i < (\ell + 1)a_s$,
	implies $(\gamma_{\bar{\Lambda}, \delta} - \gamma_{\Lambda, \delta}) a_s < a_s + a_i < (\ell + 2)a_s$.
	Combining it with $\gamma_{\bar{\Lambda}, \delta} - \gamma_{\Lambda, \delta} \in \mathbb{Z}$ yields \eqref{gamma-diff-condition}.
	We complete the proof of $H \geq 0$ by considering the following three cases (i) $\gamma_{\bar{\Lambda}, \delta} - \gamma_{\Lambda, \delta} = \ell + 1$, 
	(ii) $\gamma_{\bar{\Lambda}, \delta} - \gamma_{\Lambda, \delta} = \ell$,
	and (iii) $\gamma_{\bar{\Lambda}, \delta} - \gamma_{\Lambda, \delta} \leq \ell - 1$, separately.
	\begin{itemize}
		\item [(i)] $\gamma_{\bar{\Lambda}, \delta} - \gamma_{\Lambda, \delta} = \ell + 1$.
		Then $\gamma_{\bar{\Lambda}, \delta} \geq 1$
		and $\Delta := \delta + a(\bar{\Lambda}) + \gamma_{\bar{\Lambda}, \delta} a_s - (\delta + a(\Lambda) + \gamma_{\Lambda, \delta} a_s)
		= -a_i + (\ell + 1) a_s > 0$.
		Thus,
		\begin{equation*}
			\begin{aligned}
				H & = [g(\delta + a(\bar{\Lambda}) + \gamma_{\bar{\Lambda}, \delta} a_s)
				- g(\delta + a(\bar{\Lambda}) + \gamma_{\bar{\Lambda}, \delta} a_s - \Delta)] 
				- [g(k a_s) - g(k a_s - \Delta)]
				\stackrel{(a)}{\geq} 0,
			\end{aligned}
		\end{equation*}
		where (a) follows from $\delta + a(\bar{\Lambda}) + \gamma_{\bar{\Lambda}, \delta} a_s \leq ka_s$ (obtained by $\gamma_{\bar{\Lambda}, \delta} \geq 1$ and \cref{fact2} (ii)) and \cref{lem:slope}.
		\item [(ii)] $\gamma_{\bar{\Lambda}, \delta} - \gamma_{\Lambda, \delta} = \ell$.
		Then $\Delta^\prime := \delta + a(\Lambda) + \gamma_{\Lambda, \delta} a_s - (\delta + a(\bar{\Lambda}) + \gamma_{\bar{\Lambda}, \delta} a_s)
		= a_i - \ell a_s \geq 0$.
		Thus,
		\begin{equation*}
			\begin{aligned}
				H & = [g((k-1)a_s + \Delta^\prime) - g((k-1)a_s)] \\
				& \quad - [g(\delta + a(\bar{\Lambda}) + \gamma_{\bar{\Lambda}, \delta} a_s + \Delta^\prime) - g(\delta + a(\bar{\Lambda}) + \gamma_{\bar{\Lambda}, \delta} a_s)]
				\stackrel{(a)}{\geq} 0,
			\end{aligned}
		\end{equation*}
		where (a) follows from  $\delta + a(\bar{\Lambda}) + \gamma_{\bar{\Lambda}, \delta} a_s > (k-1)a_s$ (by \cref{fact2} (i)) and \cref{lem:slope}.
		\item [(iii)] $\gamma_{\bar{\Lambda}, \delta} - \gamma_{\Lambda, \delta} \leq \ell - 1$.
		We regroup \eqref{obj-diff5} as
		\begin{equation*}
			\small
			\begin{aligned}
				H & = [(\ell + \gamma_{\Lambda, \delta} - \gamma_{\bar{\Lambda}, \delta}) (g(ka_s) - g((k-1)a_s))]
				+ [g(a_i + (k - \ell - 1)a_s) - g((k-1)a_s)] \\
				& \quad
				- [g(\delta + a(\Lambda) + \gamma_{\Lambda, \delta} a_s) 
				- g(\delta + a(\bar{\Lambda}) + \gamma_{\bar{\Lambda}, \delta} a_s)].
			\end{aligned}
		\end{equation*}
		Below we use \cref{cor:sum-pariwise-comparison} to show $H \geq 0$.
		Observe that 
		\begin{small}
			\begin{equation*}
				(\ell + \gamma_{\Lambda, \delta} - \gamma_{\bar{\Lambda}, \delta}) a_s
				+ [ a_i + (k - \ell - 1)a_s - (k-1) a_s ]
				- [\delta + a(\Lambda) + \gamma_{\Lambda, \delta} a_s - (\delta + a(\bar{\Lambda}) + \gamma_{\bar{\Lambda}, \delta} a_s)]
				= 0,
			\end{equation*}
		\end{small}%
		and thus condition \eqref{ccondition} in \cref{cor:sum-pariwise-comparison} holds.
		By \cref{fact2} (i),
		we have
		\begin{equation} \label{case3-cond1} 
			\delta + a(\bar{\Lambda}) + \gamma_{\bar{\Lambda}, \delta} a_s > (k-1)a_s,
		\end{equation}
		which, together with
		\begin{equation}\label{case3-cond2}
			\delta + a(\Lambda) + \gamma_{\Lambda, \delta} a_s 
			- (\delta + a(\bar{\Lambda}) + \gamma_{\bar{\Lambda}, \delta} a_s )
			= a_i - (\gamma_{\bar{\Lambda}, \delta} - \gamma_{\Lambda, \delta}) a_s \geq \ell a_s - (\ell - 1)a_s = a_s,
		\end{equation}
		implies that
		\begin{equation}\label{case3-cond3}
			\delta + a(\Lambda) + \gamma_{\Lambda, \delta} a_s > ka_s.
		\end{equation}
		Combining $(k-1)a_s \leq a_i + (k - \ell - 1) a_s < ka_s$ (from $\ell a_s \leq a_i < (\ell + 1)a_s$), \eqref{case3-cond1}, \eqref{case3-cond2}, and \eqref{case3-cond3},
		conditions (i) and (ii) in \cref{cor:sum-pariwise-comparison} must hold.
		Therefore, $H \geq 0$. \qedhere
	\end{itemize}
\end{proof}

\subsubsection{Two-phase lifted inequalities}\label{subsubsec:2phase-lifted-ineq-1}
Next, we develop two families of strong valid inequalities for $\conv(\Fset)$, which are based on the subadditivity of $\eta^L$ and $\eta^U$ on $\mathbb{R}_-$.
\begin{proposition}\label{prop:sub-of-etaLU}
	Both $\eta^L$ and $\eta^U$ are subadditive on $\R_-$.
\end{proposition}
\noindent The way to prove \cref{prop:sub-of-etaLU} is to show that $\eta^L$ and $\eta^U$ are {special cases} of the subadditive functions $\bar{\omega}$ defined in \eqref{baromega} of \cref{subsec:sub-func}; see \cref{proof_4.12} for a complete proof.

Using the subadditivity of $\eta^L$ in \cref{prop:sub-of-etaLU} and the exact property of function $\eta^L$ in \cref{lem:eta} (under the condition $S_0^+ = S_0$ or $k = 1$), we can provide a closed formula for $\eta$ (i.e., \eqref{etaL} or \eqref{etaL-k=1}) and develop the first family of facet-defining inequalities for $\conv(\Fset)$.
\begin{theorem}\label{thm:facet1}
	If $S_0^+ = S_0$ or $k = 1$, then inequality
	\begin{equation}
		\label{facet-ineq1}
		w \leq \sum_{i \in S_0} \zeta(a_i) x_i + \sum_{i \in S_1}\eta(-a_i) (\mu_i - x_i) + \rho_s(k) (x_s - k) + g(k a_s)
	\end{equation}	
	is facet-defining for $\conv(\Fset)$.
\end{theorem}
For the general case where neither $S_0^+ = S_0$ nor $k = 1$ holds, 
we instead use the upper approximation $\eta^U$ to derive strong valid inequalities for $conv(\Fset)$, and give sufficient conditions under which the derived inequalities are facet-defining for $\conv(\Fset)$.

\begin{theorem}\label{thm:facet1-cond}
	Inequality
	\begin{equation}\label{eq:lifted-ineq1}
		w \leq 
		\sum_{i \in S_0} \zeta(a_i)x_i
		+ \sum_{i \in S_1} \eta^U(-a_i)(\mu_i - x_i)
		+ \rho_s(k) (x_s - k) + g(ka_s) 
	\end{equation}
	is valid for $\Fset$.
	It is facet-defining for $\conv(\Fset)$ if $(k - \mu_s - 1)a_s - A_{r \mu_{r}} \leq -a_i \leq 0$ holds for all $i \in S_1$.
\end{theorem}
\begin{proof}
	The validity of inequality \eqref{eq:lifted-ineq1} follows from $\eta(\delta) \leq \eta^U(\delta)$ for all $\delta \in \R_-$
	and the subadditivity of $\eta^U$ in \cref{prop:sub-of-etaLU}.
	Using the closed formulas for $\eta^L$ and $\eta^U$ in \eqref{etaL} and \eqref{etaU}, respectively,
	we observe that  for $\delta \in \mathbb{R}_-$ with
	$(k - \mu_s - 1) a_s - A_{r \mu_r} \leq \delta \leq 0$, $\eta^L(\delta) = \eta^U(\delta)$ holds, which 
	together with $\eta^L(\delta) \leq \eta(\delta) \leq \eta^U(\delta)$,
	it must follow that $\eta^L(\delta) = \eta(\delta) = \eta^U(\delta)$. 
	Therefore, \eqref{eq:lifted-ineq1} is facet-defining for $\conv(\Fset)$ when
	$(k - \mu_s - 1) a_s - A_{r \mu_r} \leq - a_i \leq 0$ holds for all $i \in S_1$.
\end{proof}

Observe that when $S_0^+ = S_0$ or $k = 1$, inequality \eqref{facet-ineq1} dominates inequality \eqref{eq:lifted-ineq1} as $\eta(-a_i) \leq \eta^U(-a_i)$ holds for all $i \in S_1$.  
The following example further illustrates that this dominance can be strict. 

\begin{example}
	Consider 
	$\Fset = \{(w,x) \in \R \times \{0,1\}^4 \,:\,
	w \leq f(x_1+2x_2+2x_3+6x_4)\}$,
	where $f$ is a concave function. 
	Let $s = 1$, $S_0 = \{2, 3\}$, $S_1 = \{4\}$, and $k=1$. 
	It follows that $g(z) = f(z+6)$, $\rho_1(1) = g(1) - g(0)$,
	and $S_0^+ = S_0 = \{2, 3\}$.
	Therefore, $\eta(\delta) = \eta^L(\delta)$ for all $\delta \in \R_-$,
	with closed formula derived from \eqref{etaL}:
	\begin{equation*}
		\eta(\delta) = \eta^L(\delta) = \left\{
		\begin{aligned}
			& g(\delta + 2) - \zeta(2) - g(0)
			&& \text{if}~ - 2 < \delta \leq 0, \\
			& g(\delta + 4) - 2 \zeta(2) - g(0)
			&& \text{if}~ - 4 < \delta \leq -2, \\
			& g(\delta + 5) - 2\zeta(2) - g(1)
			&& \text{if}~ \delta \leq -4,
		\end{aligned}
		\right.
	\end{equation*}
	where $\zeta$ is the first-phase lifting function defined in \eqref{calc:zeta}.
	The subadditive approximation $\eta^U$ derived from    \eqref{etaU} reads
	\begin{equation*}
		\eta^U(\delta) = \left\{
		\begin{small}
			\begin{aligned}
				& g(\delta + 2) - \zeta(2) - g(0)
				&& \text{if}~ - 2 < \delta \leq 0, \\
				& g(\delta + 4) - 2 \zeta(2) - g(0)
				&& \text{if}~ - 4 < \delta \leq -2, \\
				& g(\delta + 4 + \ell) - 2\zeta(2) + (1 - \ell) \rho_1(1) - g(1)
				&& \text{if}~ -\ell - 4 < \delta \leq -\ell -3, \ \ell = 1, 2, \ldots. \\
			\end{aligned}
		\end{small}
		\right.
	\end{equation*}	
	The corresponding inequalities \eqref{facet-ineq1} and \eqref{eq:lifted-ineq1} read:
	\begin{align}
		& w \leq \zeta(2) (x_2 + x_3) + \eta(-6) (1 - x_4) + \rho_1(1) (x_1 - 1) + g(1), \label{tmeq1} \\
		& w \leq \zeta(2) (x_2 + x_3) + \eta^U(-6) (1 - x_4) + \rho_1(1) (x_1 - 1) + g(1) .\label{tmeq2}
	\end{align}
	Letting $f(z)= -e^{-(z-6)}$, then $g(z) = - e ^{-z}$ and 
	\begin{align*}
		\eta^U(-6) - \eta(-6) & = (g(1) - 2\zeta(2) - 2 \rho_1(1) - g(1)) - (g(-1) - 2 \zeta(2) - g(1)) \\
		& = 2 g(0) - g(1) -g(-1)=-2+e^{-1}+e > 0.
	\end{align*}
	Thus, inequality \eqref{tmeq1} strictly dominates inequality \eqref{tmeq2}.
\end{example}
\subsection{Two-phase lifted inequalities: Type II}\label{subsec:ineq-class-2}
Next, we derive the symmetric family of two-phase lifted inequalities obtained by lifting  the seed inequality \eqref{eq:seedInq} using a reversed lifting order, that is, first with variables $\{x_i\}_{i \in S_1}$ and then with $\{x_i\}_{i \in S_0}$. 
We can conduct this in a similar manner as in \cref{subsec:ineq-class-1}:
(i) use the exact lifting function $\zeta$ (whose subadditivity on $\R_-$ is established in \cref{liftzeta-new}) to perform sequence-independent lifting of the seed inequality \eqref{eq:seedInq} with variables $\{x_i\}_{i \in S_1}$; and
(ii) lift the resultant inequality with variables $\{x_i\}_{i \in S_0}$ using the similar approximation or exact lifting function (i.e., $\eta^U$ or $\eta$).
Rather than repeating the lifting-based derivation,
we adopt a simpler approach by considering the symmetric counterpart of $\Fset$ obtained by complementing $x$ using $y = \mu - x$:
\begin{equation}
	\FsetS : = \left\{ (w, y) \in \R \times \Z^n \,:\, w \leq f\left( a^\top (\mu - y) \right)=\bar{f}(a^\top y ), ~ 
	0 \leq y_i \leq \mu_i, ~ \forall ~ i \in [n] \right\},
\end{equation}
where $\bar{f}(z) := f(-z + a^\top \mu)$.
Note that as $f$ is concave on $\R$, $\bar{f}$ is concave on $\R$ as well. 
Thus, $\FsetS$ is a form of $\Fset$ and we can construct the two-phase lifted inequalities in \cref{subsubsec:2phase-lifted-ineq-1} for $\conv(\FsetS)$. 
In the following, we first establish a one-to-one correspondence between the faces of $\conv(\Fset)$ and $\conv(\FsetS)$, as well as between their lifting procedures.
Using these correspondences, we can map the two-phase lifted inequalities in \cref{subsubsec:2phase-lifted-ineq-1} for $\conv(\FsetS)$ back to obtain the symmetric two-phase lifted inequalities for $\conv(\Fset)$.

\subsubsection{The relations between $\conv(\Fset)$ and $\conv(\FsetS)$}
First, we note that there exists a one-to-one correspondence between the faces of $\conv(\Fset)$ and $\conv(\FsetS)$.
\begin{proposition}\label{prop:faces}
	Inequality $w \leq \alpha_0 + \sum_{i = 1}^{n} \alpha_i x_i $
	defines a $k$-dimensional face of $\conv(\Fset)$ if and only if
	$w \leq \alpha_0 + \sum_{i = 1}^{n} \alpha_i (\mu_i - y_i)$
	defines a $k$-dimensional face of $\conv(\FsetS)$.
\end{proposition}
\begin{proof}
	The statement follows immediately from the facts that 
	(i) $(w, x) \in \Fset$ if and only if $(w, \mu - x) \in \FsetS$;
	(ii) inequality $w \leq \alpha_0 + \sum_{i = 1}^{n} \alpha_i x_i $ holds with equality at point $(w, x)$ if and only if inequality $w \leq \alpha_0 + \sum_{i = 1}^{n} \alpha_i (\mu_i - y_i) $ holds with equality at point $(w, \mu - x)$;
	and (iii) points $(w^1, x^1), \ldots, (w^{k+1}, x^{k+1})$ are affinely independent if and only if points 
	$(w^1, \mu - x^1), \ldots, (w^{k+1}, \mu - x^{k+1})$ are affinely independent.
\end{proof}
\noindent Second,  there also exists a one-to-one correspondence between the lifting procedures for $\conv(\Fset)$ and $\conv(\FsetS)$:
lifting an inequality $w \leq \alpha_0 + \sum_{i \in C} \alpha_i x_i$ (where $C \subseteq [n]$) with variable $x_j$ fixed at $0$ (respectively, fixed at $\mu_j$)  for $\conv(\Fset)$
is equivalent to lifting the corresponding inequality $w \leq \alpha_0 + \sum_{i \in C} \alpha_i (\mu_i - y_i)$ with variable $y_j$ fixed at $\mu_j$ (respectively, fixed at $0$) for $\conv(\FsetS)$.

The above two correspondences allow for deriving the symmetric two-phase lifted inequalities for $\conv(\Fset)$ (obtained by first lifting inequality \eqref{eq:seedInq} with variables $\{x_i\}_{i \in S_1}$ and then with $\{x_i\}_{i \in S_0}$) using the following two steps.
In the first step, we lift the corresponding seed inequality $w \leq \bar{\rho}_{s}(\bar{k}) (y_{s} - \bar{k}) + \bar{g}(\bar{k} a_{s})$ with variables $\{y_i\}_{i \in \bar{S}_0}$ and then with $\{y_i\}_{i \in \bar{S}_1}$, to obtain a two-phase lifted inequality (\eqref{facet-ineq1} or \eqref{eq:lifted-ineq1}) for $\conv(\FsetS)$.
Here, $\bar{S}_0 := S_1$, $\bar{S}_1 := S_0$,  $\bar{k} := \mu_s - k + 1$,
$\bar{g}(z) : = \bar{f}(z + \sum_{i \in \bar{S}_1} \mu_i a_i)$, and
$\bar{\rho}_{s}(\bar{k}) : = \bar{g}(\bar{k} a_{s}) - \bar{g}((\bar{k} - 1) a_{s})$.
In the second step, we use \cref{prop:faces} to obtain the symmetric two-phase lifted inequalities for $\conv(\Fset)$ from those for $\conv(\FsetS)$.
In the following, we will conduct the second step in detail.

\subsubsection{Symmetric two-phase lifted inequalities}
We first consider the two-phase lifted inequality \eqref{eq:lifted-ineq1} for $\conv(\FsetS)$:
\begin{equation}\label{eq:ineq-for-barFset}
	w \leq \sum_{i \in \bar{S}_0} \bar{\zeta}(a_i) y_i + \sum_{i \in \bar{S}_1} \bar{\eta}^U(-a_i) (\mu_i - y_i) 
	+ \bar{\rho}_{s}(\bar{k}) (y_{s} - \bar{k}) + \bar{g}(\bar{k} a_{s}),
\end{equation}
where $\bar{\zeta}$ and  $\bar{\eta}^U$ are defined in a similar manner to ${\zeta}$ and  ${\eta}^U$ in \eqref{calc:zeta} and \eqref{etaU}:
\begin{equation}\label{eq:bar-zeta}
	\begin{small}
		\bar{\zeta}(\delta) : = \left\{
		\begin{aligned}
			&\bar{g}(\delta+\mu_{s}a_{s})+(\bar{k}-\mu_{s}) \bar{\rho}_{s}(\bar{k})-\bar{g}(\bar{k}a_{s}), 
			&&\text{if}~\delta < (\bar{k}-\mu_{s}  - 1)a_{s},\\
			&\bar{g} \left(\delta+ (\bar{k} - \ell - 1) a_{s} \right) + (\ell + 1) \bar{\rho}_{s}(\bar{k})-\bar{g}(\bar{k}a_{s}), 
			&&\text{if}~\ell a_{s} \le \delta < (\ell + 1) a_{s}, \\
			& && ~ \ell = \bar{k} - \mu_{s}-1,\ldots, \bar{k} - 1, \\
			&\bar{g} \left(\delta\right) + \bar{k}\bar{\rho}_{s}(\bar{k})-\bar{g}(\bar{k}a_{s}), &&\text{if}~\delta \ge \bar{k}a_{s},
		\end{aligned}
		\right.
	\end{small}
\end{equation}
and
\begin{equation}\label{eq:bar-eta}
	\begin{small}
		\bar{\eta}^U(\delta) : = \left\{ 
		\begin{aligned}
			&\bar{g}(\delta+A_{(i+1) q}) - \sum_{j=1}^{i}\mu_j \bar{\zeta}(a_j) 
			- q \bar{\zeta}(a_{i+1}) + \bar{k} \bar{\rho}_{s}(\bar{k}) - \bar{g}(\bar{k} a_{s}), \\
			& \hspace{2cm}
			~\text{if}~ \bar{k} a_{s} - A_{(i+1) q} < \delta \leq -A_{(i+1) (q-1)}, \\
			& \hspace{2cm} 
			~~ i = 0,\ldots,d-1 \ \text{and} \ q = 1, \ldots, \mu_{i+1}, \\
			&\bar{g}(\delta+A_{(i+1) q} + \ell a_{s}) - \sum_{j=1}^{i}\mu_j \bar{\zeta}(a_j) 
			- q \bar{\zeta}(a_{i+1}) + (\bar{k} - \ell) \bar{\rho}_{s}(\bar{k}) - \bar{g}(\bar{k} a_{s}), \\
			& \hspace{2cm}
			~\text{if}~ (\bar{k}-\ell - 1)a_{s} - A_{(i+1) q} < \delta \leq 
			(\bar{k}-\ell )a_{s} -A_{(i+1) q}, \\
			& \hspace{2cm}
			~~ i = 0,\ldots,d-1, \  q = 1, \ldots, \mu_{i+1}, \ \text{and} \ \ell = 0,\ldots, \bar{k}-1,\\
			& \hspace{2cm}
			~~ \text{or} \ i = d-1, \  q = \mu_{d}, \ \text{and} \ \ell = \bar{k}, \bar{k}+1, \ldots.
		\end{aligned}
		\right. 
	\end{small}
\end{equation}
Here, we define $S_1^+ := \bar{S}_0^+ = \{i \in \bar{S}_0 \,:\, a_i \geq \bar{k} a_s\} = \{i \in S_1 \,:\, a_i \geq (\mu_s - k + 1) a_s\}$
and assume that $S_1^+ = [d]$ (where $d \in \Z_+$) with $a_1 \geq a_2 \geq \cdots \geq a_d$.
Moreover, we let $A_{iq} = \sum_{j = 1}^{i - 1} \mu_j a_j + q a_i$ for $i \in [d]$ and $q \in \{0, 1, \ldots, \mu_i\}$.

We next apply \cref{prop:faces} on inequality \eqref{eq:ineq-for-barFset} for $\conv(\FsetS)$ to obtain the symmetric two-phase lifted inequality for $\conv(\Fset)$.
To do this, we first simplify $\bar{g}$, $\bar{\rho}_s(\bar{k})$, $\bar{\zeta}$, and $\bar{\eta}^U$ using the previous notation $g$, $\rho_s(k)$, and $\zeta$.
By $\bar{g}(z) = \bar{f}(z+\sum_{i \in \bar{S}_1} \mu_i a_i)$, $\bar{f}(z) = f(-z+\mu^\top a)$, and $\bar{S}_1 = S_0$, we have
\begin{equation}\label{barg}
	\begin{aligned}
		\bar{g}(z) 
		= f\left(-z -\sum_{i \in S_0} \mu_i a_i +\mu^\top a\right) 
		= f\left(-z+\sum_{i \in S_1} \mu_i a_i + \mu_s a_s\right) = g(- z+\mu_s a_s).
	\end{aligned}
\end{equation}
Together with $\bar{\rho}_{s}(\bar{k}) = \bar{g}(\bar{k} a_{s}) - \bar{g}((\bar{k} - 1) a_{s})$
and $\bar{k} = \mu_s - k + 1$, we obtain
\begin{equation}\label{barg-relations}
	\begin{aligned}
		\bar{\rho}_{s}(\bar{k}) 
		= \bar{g}((\mu_s - k + 1) a_s) - \bar{g}((\mu_s - k) a_s) 
		= g((k - 1) a_s) - g(k a_s) = - \rho_s(k).
	\end{aligned}
\end{equation}
The following two lemmas further provide formulas for $\bar{\zeta}$ and $\bar{\eta}^U$.

\begin{lemma}\label{lem:barzeta}
	$\bar{\zeta}(\delta) = \zeta(-\delta)$ for $\delta \in \R$, where $\zeta$ is defined in \eqref{calc:zeta}.
\end{lemma}
\begin{proof}
	By the definition of $\bar{\zeta}$ in \eqref{eq:bar-zeta}, $\bar{k} = \mu_s - k + 1$, and \eqref{barg}--\eqref{barg-relations},
	we have
	\begin{equation*}
		\bar{\zeta}(\delta) = \left\{
		\begin{aligned}
			&g(- \delta)+k \rho_s(k) - g(k a_s), 
			&&\text{if}~\delta < - k a_s, \\
			&g\left( - \delta + (k + \ell) a_s \right) - \ell \rho_s(k) -g(k a_s), 
			&&\text{if}~\ell a_s \le \delta < (\ell + 1) a_s, \\
			& && ~ \ell = -k,\ldots, \mu_s - k, \\
			&g \left(- \delta + \mu_s a_s\right) + (k - \mu_s) \rho_s(k) - g(k a_s),
			&&\text{if}~\delta \geq (\mu_s - k + 1) a_s.
		\end{aligned}
		\right. 
	\end{equation*}
	We complete the proof by considering the following three cases.
	\begin{itemize}
		\item [(i)] $\delta < -k a_s$. Then $-\delta > k a_s$. 
		By the definition of $\zeta$ in \eqref{calc:zeta}, we have
		\begin{equation*}
			\zeta(-\delta) = g(- \delta)+k \rho_s(k) - g(k a_s) = \bar{\zeta}(\delta).
		\end{equation*}
		\item [(ii)] $\ell a_s \leq \delta < (\ell + 1) a_s$ for some $\ell \in \{-k, \ldots, \mu_s - k\}$.
		Then $ \bar{\ell} a_s < - \delta \leq (\bar{\ell} + 1) a_s$,
		where $\bar{\ell} := - (\ell + 1) \in \{k - \mu_s - 1, \ldots, k - 1\}$.
		By the definition of $\zeta$ in \eqref{calc:zeta}, we have
		\begin{align*}
			\zeta(-\delta) & = g(-\delta + (k - \bar{\ell} - 1) a_s) + (\bar{\ell} + 1) \rho_s(k) - g(k a_s) \\
			& = g\left(-\delta + (k + \ell) a_s \right) - \ell \rho_s(k) -g(k a_s)
			= \bar{\zeta}(\delta). 
		\end{align*}
		\item [(iii)] $\delta \geq (\mu_s - k + 1) a_s$. 
		Then $-\delta \leq (k - \mu_s- 1) a_s$. By the definition of $\zeta$ in \eqref{calc:zeta}, we have
		\begin{equation*}
			\zeta(-\delta) = g(-\delta +\mu_s a_s)+ (k - \mu_s) \rho_s(k) - g(k a_s) = \bar{\zeta}(\delta).\qedhere
		\end{equation*}
	\end{itemize}
\end{proof}

\begin{lemma}\label{lem:bareta}
	$\bar{\eta}^U(\delta) = \phi^U(-\delta)$ for $\delta \in \R_-$, where
	\begin{equation}\label{omegaU}
		\begin{footnotesize}
			\phi^U(\delta) : = \left\{
			\begin{aligned}
				& g(\delta + \mu_s a_s - A_{(i+1) q}) - \sum_{j = 1}^{i}\mu_j\zeta(-a_j) -q \zeta(-a_{i+1}) + (k - \mu_s) \rho_s(k) - g(ka_s), \\
				& \hspace{1.5cm} 
				\text{if}~ A_{(i+1) (q - 1)} \leq \delta < A_{(i+1) q} + (k - \mu_s - 1)a_s, \\
				& \hspace{1.5cm}
				\ i = 0,\ldots,d-1 \ \text{and} \ q = 1, \ldots, \mu_{i+1} , \\
				& g(\delta + (\mu_s - \ell) a_s - A_{(i+1) q}) - \sum_{j = 1}^{i}\mu_j\zeta(-a_j) - q \zeta(-a_{i+1})
				+ (k - \mu_s + \ell) \rho_s(k) - g(ka_s), \\
				& \hspace{1.5cm} \text{if}~ A_{(i+1) q} + (k - \mu_s + \ell - 1)a_s \leq \delta < A_{(i+1) q} + (k - \mu_s + \ell)a_s,\\
				& \hspace{1.5cm} 
				\ i = 0,\ldots,d-1, \ q = 1, \ldots, \mu_{i+1}, \ \text{and} \ \ell = 0,\ldots, \mu_s - k, \\
				& \hspace{1.5cm}
				~ \text{or} \ i = d-1, \  q = \mu_{d}, \ \text{and} \ \ell = \mu_s - k + 1, \mu_s - k + 2,\ldots.
			\end{aligned}
			\right.
		\end{footnotesize}
	\end{equation}
\end{lemma}

\begin{proof}
	By the definition of $\bar{\eta}^U$ in \eqref{eq:bar-eta}, $\bar{k} = \mu_s - k + 1$, \eqref{barg}--\eqref{barg-relations}, and $\bar{\zeta}(\delta) = \zeta(-\delta)$, we have
	\begin{equation*}
		\begin{footnotesize}
			\bar{\eta}^U(\delta) = \left\{ 
			\begin{aligned}
				&g(- \delta + \mu_s a_s - A_{(i+1) q}) - \sum_{j=1}^{i}\mu_j \zeta(- a_j) 
				- q \zeta(- a_{i+1}) + (k - \mu_s) \rho_s(k) - g(k a_s), \\
				& \hspace{3cm}
				~\text{if}~ (\mu_s-  k + 1) a_s - A_{(i+1) q} < \delta \leq -A_{(i+1) (q-1)}, \\
				& \hspace{3cm} 
				~~ i = 0,\ldots,d-1 \ \text{and} \ q = 1, \ldots, \mu_{i+1}, \\
				&g(- \delta + (\mu_s - \ell) a_s - A_{(i+1) q}) - \sum_{j=1}^{i}\mu_j \zeta(-a_j) 
				- q \zeta(-a_{i+1}) + (k - \mu_s + \ell) \rho_s(k)- g(k a_s), \\
				& \hspace{3cm}
				~\text{if}~ (\mu_s - k -\ell)a_s - A_{(i+1) q} < \delta \leq 
				(\mu_s - k -\ell + 1)a_s -A_{(i+1) q}, \\
				& \hspace{3cm}
				~~ i = 0,\ldots,d-1, \  q = 1, \ldots, \mu_{i+1}, \ \text{and} \ \ell = 0,\ldots, \mu_s - k,\\
				& \hspace{3cm}
				~~ \text{or} \ i = d-1, \  q = \mu_{d}, \ \text{and} \ \ell = \mu_s-k+1,\mu_s-k+2,\ldots.
			\end{aligned}
			\right.
		\end{footnotesize}
	\end{equation*}
	It immediately follows that $\bar{\eta}^U(\delta) = \phi^U(-\delta)$ for $\delta \in \R_-$.
\end{proof}

Using the above representations for $\bar{g}$, $\bar{\rho}_s$, $\bar{\zeta}$, and $\bar{\eta}^U$, 
we can rewrite inequality \eqref{eq:ineq-for-barFset} for $\conv(\FsetS)$ as 
\begin{equation*}
	w \leq \sum_{i \in \bar{S}_0} \zeta(-a_i) y_i + \sum_{i \in \bar{S}_1} \phi^U(a_i) (\mu_i - y_i) 
	+ \rho_s(k) (\mu_s - y_s - k) + g(k a_s).
\end{equation*}
Together with $\bar{S}_0 = S_1$, $\bar{S}_1 = S_0$, $\bar{k} = \mu_{s} - k + 1$, \cref{thm:facet1-cond}, and \cref{prop:faces}, we obtain the symmetric family of two-phase lifted inequalities for $\conv(\Fset)$ and their facet-defining conditions,
as stated in the following theorem.
\begin{theorem}
	Inequality 
	\begin{equation}\label{eq:lifted-ineq2}
		w \leq \sum_{i \in S_1} \zeta(-a_i) (\mu_i - x_i) + \sum_{i \in S_0} \phi^U(a_i) x_i
		+ \rho_s(k) (x_s - k) + g(k a_s)
	\end{equation}
	is valid for $\Fset$.
	If $0 \leq a_i \leq A_{d \mu_d} + k a_s$ holds for all $i \in S_0$, then \eqref{eq:lifted-ineq2} is facet-defining for $\conv(\Fset)$.
\end{theorem}
Similarly, we can derive a strengthened version of inequality \eqref{eq:lifted-ineq2}
from the facet-defining inequality \eqref{facet-ineq1} for $\conv(\FsetS)$,
as stated in the following theorem.
The proof is similar and omitted for brevity.

\begin{theorem}\label{thm:facet2}
	If $S_1^+ = S_1$ or $k = \mu_s$, then inequality
	\begin{equation}\label{facet-ineq2}
		w \leq \sum_{i \in S_1}\zeta(-a_i) (\mu_i - x_i) + \sum_{i \in S_0} \phi(a_i) x_i + \rho_s(k) (x_s - k) + g(k a_s)
	\end{equation}	
	is facet-defining for $\conv(\Fset)$.
	Here, (i) if $S_1^+ = S_1$, then $\phi$ is given by:
	\begin{equation}\label{phiL}
		\begin{footnotesize}
			\phi(\delta) = \left\{
			\begin{aligned}
				& g(\delta + \mu_s a_s - A_{(i+1) q}) - \sum_{j = 1}^{i}\mu_j\zeta(-a_j) -q \zeta(-a_{i+1}) + (k - \mu_s) \rho_s(k) - g(ka_s), \\
				& \hspace{1cm} 
				\text{if}~ A_{(i+1) (q - 1)} \leq \delta < A_{(i+1) q} + (k - \mu_s-1)a_s, \\
				& \hspace{1cm}  \ i = 0,\ldots,d-1 \ \text{and} \ q = 1, \ldots, \mu_{i+1} , \\
				& g(\delta + (\mu_s - \ell) a_s - A_{(i+1) q}) - \sum_{j = 1}^{i}\mu_j\zeta(-a_j) - q \zeta(-a_{i+1})
				+ (k - \mu_s + \ell) \rho_s(k) - g(ka_s), \\
				& \hspace{1cm} \text{if}~ A_{(i+1) q} + (k - \mu_s + \ell - 1)a_s \leq \delta < A_{(i+1) q} + (k - \mu_s + \ell)a_s,\\
				& \hspace{1cm} \ i = 0,\ldots,d-1, \ q = 1, \ldots, \mu_{i+1}, \ \text{and} \ \ell = 0,\ldots, \mu_s - k, \\
				& \hspace{1cm} 
				\ \text{or} \ i = d-1, \ q = \mu_{d}, \ \text{and} \ \ell = \mu_s-k+1,\ldots, \mu_s, \\
				& g(\delta - A_{d \mu_d}) - \sum_{j = 1}^{d}\mu_j\zeta(-a_j) 
				+ k \rho_s(k) - g(ka_s), \\
				& \hspace{1cm} \text{if} \ \delta \geq A_{d \mu_d} + k a_s,
			\end{aligned}
			\right.
		\end{footnotesize}
	\end{equation}
	and (ii) if $k = \mu_s$, then $\phi$ is given by:
	\begin{equation}\label{eq:phi}
		\begin{footnotesize}
			\phi(\delta) = \left\{
			\begin{aligned}
				& g\left(\delta + \mu_s a_s - A_{(i+1) q} \right) - \sum_{j = 1}^{i}\mu_j\zeta(-a_j) -q \zeta(-a_{i+1}) - g\left( \mu_s a_s \right), \\
				& \hspace{2cm}
				\text{if}~ A_{(i+1) (q - 1)} \leq \delta < A_{(i+1) q}, 
				\ \ i = 0,\ldots,d \ \text{and} \ q = 1, \ldots, \mu_{i+1}, \\
				& g\left(\delta +  \mu_s a_s - A_{(d + 1) \mu_{d + 1}} \right) - \sum_{j = 1}^{d + 1}\mu_j\zeta(-a_j) - g\left( \mu_s a_s \right), \\
				& \hspace{2cm}
				\text{if} \ \delta \geq A_{(d + 1) \mu_{d + 1}},
			\end{aligned}
			\right.
		\end{footnotesize}
	\end{equation}	
	where, by abuse of notation, we assume that $S_1 \cup s = [d + 1]$ (where $d \in \Z_+$) 
	with $a_1 \geq a_2 \geq \cdots \geq a_{d + 1}$
	and redefine $A_{i q} = \sum_{j = 1}^{i - 1} \mu_j a_j + q a_i$ for $i \in [d + 1]$ and $q \in \{0, 1, \ldots, \mu_i\}$.
\end{theorem}
\subsection{A class of subadditive functions}\label{subsec:sub-func}
The main goal of this subsection is to introduce two classes of subadditive functions that contain $\eta^L$ and $\eta^U$ in \eqref{eta-explicit} and \eqref{etaU} as special cases, thereby establishing their subadditivity.
Given 
\begin{equation}\label{def:b}
	\begin{aligned}
		& \epsilon \geq 0, \ \tau \in \Z_{++},  \\
		& b_i \geq 0, \ b_i \geq b_{i + 1}, 
		& \quad ~\text{for} ~i = 1, 2, \ldots \\
		& v_i \in \mathbb{R} ~\text{with}~ b_{i + 2} + v_{i + 1} \geq b_{i + 1} + v_{i}, 
		& ~\text{for} ~i = 0, 1, \ldots
	\end{aligned}
\end{equation}
let 
\begin{equation}\label{ABdef}
	\begin{aligned}
		&  a_i = b_i + \tau \epsilon , 
		&& ~\text{for} ~i = 1, 2, \ldots \\
		& A_0 = 0, \ A_i = \sum_{j=1}^{i}a_j, \ \text{and} \ B_i = A_{i - 1} + b_i,
		&& ~ \text{for} ~ i = 1,2,\ldots.
	\end{aligned}
\end{equation}
For a concave function  $g: \R \rightarrow \R$,  
we define a continuous piecewise concave function $\omega  : \R_- \rightarrow \R$ as follows:
\begin{equation}\label{omega}
	\begin{footnotesize}
		\omega (\delta) : = \left\{
		\begin{aligned}
			&0,	&&~\text{if}~ \delta = 0, \\
			&g(\delta + A_i  - v_i) + \omega (-A_i) - g(-v_i),
			&&~\text{if}~ -B_{i+1} \leq \delta < -A_{i}, \ i = 0, 1, \dots, \\
			& g(\delta + A_i  + \ell \epsilon - v_i) + \omega (-A_i ) + \ell \psi_{i} - g(- v_i), 
			&&~\text{if}~ - B_{i+1}  - (\ell + 1) \epsilon \leq \delta < - B_{i+1}  - \ell \epsilon, \\
			& && ~~i = 0, 1, \dots \ \text{and} \ \ell = 0,\ldots,\tau - 1, 
		\end{aligned}
		\right.
	\end{footnotesize}
\end{equation}
where $\psi_i : = g\left( - b_{i + 1} - v_{i} - \epsilon\right) - g( - b_{i + 1} - v_{i})$ for $i \in \Z_+$.
Given $\omega$, we provide another continuous piecewise concave function $\bar{\omega} \,:\, \R_- \rightarrow \R$ as follows:
\begin{equation}\label{baromega}
	\begin{footnotesize}
		\bar{\omega} (\delta) : = \left\{
		\begin{aligned}
			& \omega (\delta), 
			\hspace{3.5cm} ~\text{if}~ - B _{m} - (\gamma + 1) \epsilon \leq \delta \leq 0, \\
			& g(\delta + A _{m - 1} + \gamma \epsilon - v_{m - 1}) + \omega (-A _{m - 1}) + \gamma \psi_{m - 1} - g\left(-v_{m - 1} \right), \\
			& \hspace{4.2cm} ~\text{if}~ \delta < - B _{m} - (\gamma + 1) \epsilon,
		\end{aligned}
		\right.
	\end{footnotesize}
\end{equation}
where $m \in \Z_{++}$ and $\gamma \in \{0, \ldots, \tau - 1\}$.
\cref{fig:omega} depicts a realization of the functions $\omega$ and $\bar{\omega}$.

In addition to the above two functions, we also investigate their reflections $\chi(\delta) := \omega(-\delta)$ and $\bar{\chi}(\delta) := \bar{\omega}(-\delta)$ (where $\delta \in \mathbb{R}_+$), that is,
\begingroup
\footnotesize
\makeatletter
\let\orig@tagform@\tagform@
\def\tagform@#1{\maketag@@@{\normalsize(\ignorespaces#1\unskip\@@italiccorr)}}
\makeatother

\begin{align}
	& \chi(\delta) = \left\{
	\begin{aligned}
		&0, &&~\text{if}~ \delta = 0, \\
		&h(\delta - A_i  + v_i) + \chi (A_i ) - h(v_i),
		&&~\text{if}~ A_i  < \delta \leq B_{i+1}, \ i = 0, 1, \dots, \\
		& h(\delta - A_i  -\ell \epsilon + v_i) + \chi (A_i ) + \ell \psi_{i} - h(v_i), 
		&&~\text{if}~ B_{i+1}  + \ell \epsilon < \delta \leq B_{i+1}  + (\ell + 1) \epsilon, \\
		& && ~~i = 0, 1, \dots \ \text{and} \ \ell = 0,\ldots,\tau - 1, 
	\end{aligned}
	\right.
	\label{chi}\\[1mm]
	& \bar{\chi} (\delta) = \left\{
	\begin{aligned}
		& \chi(\delta), 
		\hspace{5.2cm} \text{if}~ 0 \leq \delta \leq B _{m} + (\gamma + 1) \epsilon, \\
		& h(\delta - A _{m - 1} - \gamma \epsilon + v_{m - 1}) + \chi (A _{m - 1}) + \gamma \psi_{m - 1} - h\left(v_{m - 1} \right), \\
		& \hspace{5.9cm} \text{if}~ \delta > B _{m} + (\gamma + 1) \epsilon,
	\end{aligned}
	\right.
	\label{barchi}
\end{align}
\makeatletter
\let\tagform@\orig@tagform@
\makeatother
\endgroup
where $h(z) := g(-z)$ is the reflection of $g$ and is therefore concave (as $g$ is concave).
Moreover, $\psi_i$ can be computed from $h$ as follows:
\begin{equation}\label{def-rho}
	\psi_i = h\left(b_{i + 1} + v_{i} + \epsilon\right) - h(b_{i + 1} + v_{i}),
	\quad ~\text{for} ~i = 0, 1, \ldots.
\end{equation}

The following theorem establishes the subadditivity of the above four functions.  
The proof can be found in \cref{proof_4.21}.
\begin{theorem} \label{thm:chi-sub}
	(i) Functions $\omega$ and $\bar{\omega}$ are subadditive on $\R_-$; and (ii)
	functions $\chi$ and $\bar{\chi}$ are subadditive on $\R_+$.
\end{theorem}

It deserves to mention that \cref{thm:chi-sub} also provides a framework to construct subadditive functions or prove subadditivity of known functions 
(though our primary focus is to establish the subadditivity of $\eta^L$ and $\eta^U$).
In particular, as applications of  \cref{thm:chi-sub}, we show that several known subadditive functions in the literature \cite{Atamturk2003,Shi2022} can be seen as special cases of functions $\chi$ and $\bar{\chi}$ (i.e., the reflection of $\omega$ and $\bar{\omega}$).
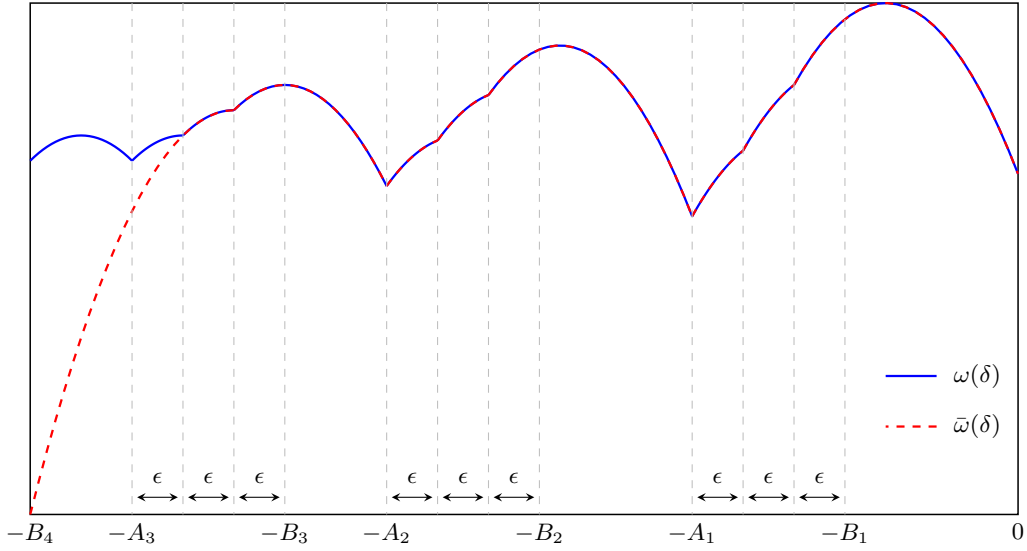
\begin{figure}[t]
	\centering
	\begin{tikzpicture}[scale=\textwidth/10cm]

		\draw[line width=0.2mm] (-9.7,1.69) rectangle (0,-3.37);

		\draw[domain=-2.2:0, color=blue, line width=0.3mm]   plot (\x, {-( -\x-1.3)^2+1.69});
		\draw[domain=-2.7:-2.2, color=blue, line width=0.3mm] plot (\x, {-( -\x-1.8)^2+1.69 -0.65});
		\draw[domain=-3.2:-2.7, color=blue, line width=0.3mm] plot (\x, {-( -\x-0.5*2+1.7-3)^2+1.69 -0.65*2});
		
		\draw[domain=-5.2:-3.2, color=blue, line width=0.3mm] plot (\x, {-( -\x-3.2+1.7-3)^2-0.42+1.3^2});
		\draw[domain=-5.7:-5.2, color=blue, line width=0.3mm] plot (\x, {-( -\x-3.2+1.7-0.5-3)^2-0.42-0.45+1.3^2});
		\draw[domain=-6.2:-5.7, color=blue, line width=0.3mm] plot (\x, {-( -\x-3.2+1.7-0.5*2-3)^2-0.42-0.45*2+1.3^2});
		
		\draw[domain=-7.7:-6.2, color=blue, line width=0.3mm] plot (\x, {-( -\x-6.2+2-3)^2-0.12+1});
		\draw[domain=-8.2:-7.7, color=blue, line width=0.3mm] plot (\x, {-( -\x-6.2-0.5+2-3)^2-0.12-0.25+1});
		\draw[domain=-8.7:-8.2, color=blue, line width=0.3mm] plot (\x, {-( -\x-6.2-0.5*2+2-3)^2-0.12-0.25*2+1});
		
		\draw[domain=-9.7:-8.7, color=blue, line width=0.3mm] plot (\x, {-( -\x-8.7+2.5-3)^2+0.13+0.25});

		\draw[domain=-2.2:0, dashed, color=red, line width=0.3mm]   plot (\x, {-( -\x-1.3)^2+1.69});
		\draw[domain=-2.7:-2.2, dashed, color=red, line width=0.3mm] plot (\x, {-( -\x-1.8)^2+1.69 -0.65});
		\draw[domain=-3.2:-2.7, dashed, color=red, line width=0.3mm] plot (\x, {-( -\x-0.5*2+1.7-3)^2+1.69 -0.65*2});
		
		\draw[domain=-5.2:-3.2, dashed, color=red, line width=0.3mm] plot (\x, {-( -\x-3.2+1.7-3)^2-0.42+1.3^2});
		\draw[domain=-5.7:-5.2, dashed, color=red, line width=0.3mm] plot (\x, {-( -\x-3.2+1.7-0.5-3)^2-0.42-0.45+1.3^2});
		\draw[domain=-6.2:-5.7, dashed, color=red, line width=0.3mm] plot (\x, {-( -\x-3.2+1.7-0.5*2-3)^2-0.42-0.45*2+1.3^2});
		
		\draw[domain=-7.7:-6.2, dashed, color=red, line width=0.3mm] plot (\x, {-( -\x-6.2+2-3)^2-0.12+1});
		\draw[domain=-9.7:-7.7, dashed, color=red, line width=0.3mm] plot (\x, {-( -\x-6.2-0.5+2-3)^2-0.12-0.25+1});

		\draw[gray!50, dashed]      (-1.7,1.69) -- (-1.7,-3.37);
		\draw[gray!50, dashed,thin] (-2.2,1.69) -- (-2.2,-3.37);
		\draw[gray!50, dashed,thin] (-2.7,1.69) -- (-2.7,-3.37);
		\draw[gray!50, dashed,thin] (-3.2,1.69) -- (-3.2,-3.37);
		
		\draw[gray!50, dashed,thin] (-4.7,1.69) -- (-4.7,-3.37);
		\draw[gray!50, dashed,thin] (-5.2,1.69) -- (-5.2,-3.37);
		\draw[gray!50, dashed,thin] (-5.7,1.69) -- (-5.7,-3.37);
		\draw[gray!50, dashed,thin] (-6.2,1.69) -- (-6.2,-3.37);
		
		\draw[gray!50, dashed,thin] (-7.2,1.69) -- (-7.2,-3.37);
		\draw[gray!50, dashed,thin] (-7.7,1.69) -- (-7.7,-3.37);
		\draw[gray!50, dashed,thin] (-8.2,1.69) -- (-8.2,-3.37);
		\draw[gray!50, dashed,thin] (-8.7,1.69) -- (-8.7,-3.37);

		\node at(-0.4,-2) {\small $\omega(\delta)$};
		\draw[color=blue,line width=0.3mm] (-0.8,-2) -- (-1.3,-2);
		\node at(-0.4,-2.5) {\small $\bar{\omega}(\delta)$};
		\draw[color=red,dashed,line width=0.3mm] (-0.8,-2.5) -- (-1.3,-2.5);
		
		\node[below] at(0,-3.37) {\small $0$};
		
		\node[below] at(-1.7,-3.37) {\small $-B_1$};
		\draw[<->,>=stealth] (-1.75,-3.2) -- (-2.15,-3.2);
		\node[above] at(-1.95,-3.15) {\small $\epsilon$};
		\draw[<->,>=stealth] (-2.25,-3.2) -- (-2.65,-3.2);
		\node[above] at(-2.45,-3.15) {\small $\epsilon$};
		\draw[<->,>=stealth] (-2.75,-3.2) -- (-3.15,-3.2);
		\node[above] at(-2.95,-3.15) {\small $\epsilon$};
		\node[below] at(-3.2,-3.37) {\small $-A_1$};
		
		\node[below] at(-4.7,-3.37) {\small $-B_2$};
		\draw[<->,>=stealth] (-4.75,-3.2) -- (-5.15,-3.2);
		\node[above] at(-4.95,-3.15) {\small $\epsilon$};
		\draw[<->,>=stealth] (-5.25,-3.2) -- (-5.65,-3.2);
		\node[above] at(-5.45,-3.15) {\small $\epsilon$};
		\draw[<->,>=stealth] (-5.75,-3.2) -- (-6.15,-3.2);
		\node[above] at(-5.95,-3.15) {\small $\epsilon$};
		\node[below] at(-6.2,-3.37) {\small $-A_2$};
		
		\node[below] at(-7.2,-3.37) {\small $-B_3$};
		\draw[<->,>=stealth] (-7.25,-3.2) -- (-7.65,-3.2);
		\node[above] at(-7.45,-3.15) {\small $\epsilon$};
		\draw[<->,>=stealth] (-7.75,-3.2) -- (-8.15,-3.2);
		\node[above] at(-7.95,-3.15) {\small $\epsilon$};
		\draw[<->,>=stealth] (-8.25,-3.2) -- (-8.65,-3.2);
		\node[above] at(-8.45,-3.15) {\small $\epsilon$};
		\node[below] at(-8.7,-3.37) {\small $-A_3$};
		
		\node[below] at(-9.7,-3.37) {\small $-B_4$};
	\end{tikzpicture}
	\caption{Functions $\omega$ and $\bar{\omega}$ of the example: $g(z) = - (z - c)^2$, $\tau = 3$, $m = 3$, and $\gamma = 1$.}
	\label{fig:omega}
\end{figure}

\begin{example}[Section 2.3 in \cite{Shi2022}]
	Suppose that $\tau = 1$.
	Then, $B_{i + 1} + \epsilon = A_i + b_{i + 1} + \epsilon =  A_i + a_{i + 1} =  A_{i + 1}$
	for $i \in \Z_+$.
	As a result, $\chi$ reduces to 
	\begin{equation*}
		\chi(\delta) = \left\{
		\begin{aligned}
			& 0, 
			&& ~\text{if}~ \delta = 0, \\
			& h(\delta - A_i + v_i) + \chi(A_i) - h(v_i),
			&& ~\text{if}~ A_i < \delta \leq A_{i+1}, \ i = 0,1, \ldots,
		\end{aligned}
		\right.
	\end{equation*}
	which is the subadditive function proposed in \cite[Eq. (11)]{Shi2022}.
	For $\bar{\chi}$, as $\gamma \in \{0, \ldots, \tau - 1\}$ and $\tau = 1$, $\gamma = 0$ must hold.
	Thus, $\bar{\chi}$ reduces to
	\begin{equation*}
		\begin{small}
			\bar{\chi}(\delta) = \left\{
			\begin{aligned}
				& 0, 
				&& ~\text{if}~ \delta = 0, \\
				& h(\delta - A_i + v_i) + \chi(A_i) - h(v_i),
				&& ~\text{if}~ A_i < \delta \leq A_{i+1}, \ i = 0, \ldots, m - 1, \\
				& h(\delta - A_{m - 1} + v_{m - 1}) + \chi(A_{m - 1}) - h(v_{m - 1}),
				&& ~\text{if}~ \delta > A_{m}, \\
			\end{aligned}
			\right.
		\end{small}
	\end{equation*}
	which matches another subadditive function proposed in \cite[Eq. (12)]{Shi2022}. \qedhere
\end{example}

Next, we show that $\chi$ and $\bar{\chi}$ generalize the subadditive functions in \cite{Atamturk2003}. We first compute $\chi(A_i)$:
\begin{equation*}
	\begin{small}
		\begin{aligned}
			\chi (A_i )  & \stackrel{(a)}{=} \chi (B_i  + \tau \epsilon) \\
			& = h(B_i  + \tau \epsilon - A_{i-1}  - (\tau - 1) \epsilon + v_{i - 1}) + \chi (A_{i-1} ) + (\tau - 1) \psi_{i - 1} - h(v_{i - 1}) \\
			& = h(b_i + v_{i-1} + \epsilon) + \chi (A_{i-1} ) + (\tau - 1) \psi_{i - 1} - h(v_{i - 1}) \\
			& = \chi (A_{i-1} ) + (\tau - 1) \psi_{i - 1} + (h(b_i + v_{i-1} + \epsilon) - h(b_i + v_{i-1})) + (h(b_i + v_{i-1}) - h(v_{i - 1})) \\
			& = \chi (A_{i-1} ) + \tau \psi_{i - 1} + (h(b_i + v_{i-1}) - h(v_{i - 1})),
		\end{aligned}
	\end{small}
\end{equation*}
where (a) follows from $A_i = A_{i - 1} + a_i = A_{i - 1} + b_i + \tau \epsilon = B_i  + \tau \epsilon$.
Together with $\chi(A_0) = \chi(0) = 0$, we have 
\begin{equation}\label{chi_Ai}
	\chi(A_i) = \tau \sum_{j = 1}^{i} \psi_{j - 1} + \sum_{j = 1}^{i} (h(b_j + v_{j - 1}) - h(v_{j - 1})).
\end{equation}

\begin{example}[Section 3.3 in \cite{Atamturk2003}]\label{example1}
	Suppose that $h(z) = \min\{0, -z\}$ and $b_{i+1} + v_{i}= -\sigma$ holds for all $i =0, 1,\ldots $, where $0 \leq \sigma \leq \epsilon$.
	By definition, for $i =0, 1,\ldots$, it follows that
	\begin{align*}
		& h(b_{i + 1} + v_{i}) = h(-\sigma) = 0, ~
		h(v_{i}) = h(-\sigma - b_{i + 1}) = 0, \\
		& \psi_{i} = h(b_{i + 1} + v_{i} + \epsilon) - h(b_{i + 1} + v_{i}) = h(\epsilon - \sigma) - h(-\sigma) = -(\epsilon - \sigma).
	\end{align*}
	Thus, from \eqref{chi_Ai}, we have $\chi (A_i ) = -i\tau(\epsilon - \sigma)$.\\[5pt]
	{\bf Claim 1} The function $\chi$ in \eqref{chi} reduces to
	\begin{equation*}
		\chi (\delta) = \left\{
		\begin{aligned}
			& 0, 
			&& ~\text{if}~\delta = 0, \\
			& -i\tau(\epsilon - \sigma), 
			&& ~\text{if}~A_i  < \delta \leq B_{i+1} ,~i=0, 1,\ldots,  \\
			& -(i \tau + \ell) (\epsilon - \sigma), 
			&& ~\text{if}~B_{i+1}  + \ell \epsilon < \delta \leq B_{i+1}  + \ell \epsilon + \sigma, \\
			&&& ~~ i = 0, 1, \dots \ \text{and} \ \ell = 0,\ldots,\tau - 1, \\ 
			& -i \tau(\epsilon - \sigma) - \delta + B_{i+1}  + (\ell + 1)\sigma,
			&& ~\text{if}~B_{i+1}  + \ell \epsilon + \sigma < \delta \leq B_{i+1}  + (\ell + 1) \epsilon,\\
			&&& ~~ i = 0, 1, \dots \ \text{and} \  \ell = 0,\ldots,\tau - 1,
		\end{aligned}
		\right.
	\end{equation*}
	which matches the subadditive function proposed in \cite[Section 3.3]{Atamturk2003}.
	\begin{proof}
		We prove the claim by considering the following two cases.
		\begin{itemize}
			\item [(i)] $A_i  < \delta \leq B_{i+1} $.
			Then $\delta - A_i + v_i \leq B_{i+1} - A_i + v_i 
			= b_{i+1} + v_i = -\sigma \leq 0$, and thus $h(\delta - A_i + v_i)= 0$.
			As a result,
			\begin{align*}
				\chi (\delta)  = h(\delta - A_i + v_i) + \chi (A_i) - h(v_i)
				= \chi (A_i) - h(v_i) \stackrel{(a)}{=} -i\tau(\epsilon - \sigma).
			\end{align*}
			where (a) follows from $\chi (A_i) = -i\tau(\epsilon - \sigma)$ and $h(v_i) = 0$.
			\item [(ii)] $B_{i+1}  + \ell \epsilon < \delta \leq B_{i+1}  + (\ell + 1) \epsilon$ for some $\ell \in \{0, \ldots, \tau - 1\}$.
			Then
			\begin{equation*}
				\begin{aligned}
					h(\delta - A_i  - \ell \epsilon + v_i) & = \min\{0, - (\delta - A_i  - \ell \epsilon + v_i)\} \\
					& \stackrel{(a)}{=} \min\{0, - (\delta - B_{i + 1}  - \ell \epsilon - \sigma)\} \\
					& = \left\{
					\begin{aligned}
						& 0, 
						&& ~\text{if}~ \delta \leq B_{i + 1}  + \ell \epsilon + \sigma, \\
						& -\delta + B_{i + 1}  + \ell \epsilon + \sigma,
						&& ~\text{if}~ \delta > B_{i + 1}  + \ell \epsilon + \sigma.
					\end{aligned}
					\right.
				\end{aligned}
			\end{equation*}
			where (a) follows from $v_i = - \sigma - b_{i + 1}$.
			This, together with $\chi (A_i ) = -i\tau(\epsilon - \sigma)$, $\psi_{i} = -(\epsilon - \sigma)$, and $h(v_i) = 0$, yields
			\begin{equation*}
				\small
				\begin{aligned}
					\chi (\delta) & = h(\delta - A_i  -\ell \epsilon + v_i) + \chi (A_i ) + \ell \psi_{i} - h(v_i) \\
					& = h(\delta - A_i  -\ell \epsilon + v_i) - (i \tau + \ell)(\epsilon - \sigma)  \\
					& = \left\{
					\begin{aligned}
						& - (i \tau + \ell)(\epsilon - \sigma) ,
						&& \text{ if } B_{i+1}  + \ell \epsilon < \delta \leq B_{i+1}  + \ell \epsilon + \sigma, \\
						& -i \tau(\epsilon - \sigma) - \delta + B_{i+1}  + (\ell + 1)\sigma,
						&& \text{ if } B_{i+1}  + \ell \epsilon + \sigma < \delta \leq B_{i+1}  + (\ell + 1) \epsilon. \qedhere
					\end{aligned}
					\right.
				\end{aligned}
			\end{equation*}
		\end{itemize}
	\end{proof}
	\noindent {\bf Claim 2} 
	Let $\gamma = \tau - 1$ in the definition of $\bar{\chi}$.
	Then the function $\bar{\chi}$ in \eqref{barchi} reduces to:
	\begin{equation*}
		\begin{aligned}
			\bar{\chi} (\delta) & = \left\{
			\begin{aligned}
				& \chi (\delta), && \text{if}~ 0 \leq \delta \leq A_{m }, \\
				& -\delta + A_m + \chi(A_m),
				&& \text{if}~ \delta > A_{m}.
			\end{aligned}
			\right. 
		\end{aligned}
	\end{equation*}
	which matches another subadditive function proposed in \cite[Section 3.3]{Atamturk2003}.
	\begin{proof}
		Observe that $B_m + (\gamma + 1) \epsilon = B_m + \tau \epsilon = A_{m - 1} + b_m + \tau \epsilon = A_{m - 1} + a_m = A_{m}$.
		Then for $\delta > A_m$, we have
		\begin{align*}
			\delta - A_{m - 1} - (\tau - 1)\epsilon + v_{m - 1} 
			& >  A_{m} - A_{m - 1} - (\tau - 1)\epsilon + v_{m - 1} \\
			& = a_m  - (\tau - 1)\epsilon + v_{m - 1} \\
			& \stackrel{(a)}{=} b_m + v_{m - 1} + \epsilon\\
			& \stackrel{(b)}{=} \epsilon - \sigma \geq 0,
		\end{align*}
		where (a)  and (b) follows from $a_m = b_m + \tau \epsilon$ and $b_m + v_{m - 1} = - \sigma$, respectively.
		As a result,
		\begin{align*}
			& \quad ~ h(\delta - A _{m - 1} - (\tau - 1) \epsilon + v_{m - 1}) + \chi (A _{m - 1}) + (\tau - 1) \psi_{m - 1} - h\left(v_{m - 1} \right) \\
			& = - \delta + A _{m - 1} + (\tau - 1) \epsilon - v_{m - 1} + \chi (A _{m - 1}) + (\tau - 1) \psi_{m - 1} - h\left(v_{m - 1} \right) \\
			& \stackrel{(a)}{=} - \delta + A _{m - 1} + (\tau - 1) \epsilon - v_{m - 1} + \chi (A _{m - 1}) - (\tau - 1) (\epsilon - \sigma) \\
			& \stackrel{(b)}{=} - \delta + A _{m} - a_m + (\tau - 1) \epsilon - (- \sigma - b_m) + \chi (A _{m - 1}) - (\tau - 1) (\epsilon - \sigma) \\
			& \stackrel{(c)}{=} - \delta + A _{m}  + \chi (A _{m - 1}) - \tau  (\epsilon - \sigma) \\
			& \stackrel{(d)}{=} - \delta + A_m + \chi(A_m),
		\end{align*}
		where (a) follows from $\psi_{m - 1} = - (\epsilon - \sigma)$ and $h(v_{m - 1}) = 0$,
		(b) follows from $A_m = A_{m - 1} + a_m$ and $v_{m - 1} = - \sigma - b_m$,
		(c) follows from $a_m - b_m = \tau \epsilon$,
		and (d) follows from $\chi(A_i) = -i \tau (\epsilon - \sigma)$ for any $i \in \mathbb{Z}_+$. \qedhere
	\end{proof}
\end{example}
\section{Applications}\label{sec:special-case}
In this section, we apply our polyhedral results on the mixed-integer nonlinear set $\Fset$ to several previously investigated sets, including the submodular maximization set $\FsetB$, the mixed-integer knapsack set $\Kset$, the mixed-integer polyhedral conic set $\Cset$, and the two-row mixed-integer set.
Our results show that some previous strong valid inequalities for these sets can be seen as special cases of our proposed single- and two-phase lifted inequalities in \eqref{eq:single-phase}, \eqref{facet-ineq1}, \eqref{eq:lifted-ineq1}, \eqref{eq:lifted-ineq2}, and \eqref{facet-ineq2}.
In addition, our results may also identify new strong valid inequalities for these sets.

\subsection{Strong valid inequalities for the submodular maximization set}\label{subsec:special-binary}

{Recall that when the integer variables $x$ in the mixed-integer nonlinear set $\Fset$ are binary (i.e., $\mu_i = 1$ for all $i \in [n]$), $\Fset$ reduces to the submodular maximization set
\begin{equation}
	\FsetB = \left\{(w,x) \in \R \times \{0,1\}^n \,:\, w \leq f(a^\top x)\right\}.
\end{equation} 
\citet{Ahmed2011} first established the submodularity of the function $f(a^\top x)$ and employed the sequence-independent lifting technique to strengthen the classical submodular inequalities of \citet{Nemhauser1988}.
\citet{Shi2022} further developed two families of valid inequalities which are facet-defining for $\conv(\FsetB)$ and stronger than those in \cite{Ahmed2011}.
In the following, we apply our polyhedral results for $\conv(\Fset)$ in \cref{sec:twosteplifting,sec:lift-one-phase} to $\conv(\FsetB)$. 
Our results show that for $\conv(\FsetB)$ (i) the proposed single-phase lifting procedure in \cref{sec:lift-one-phase} can identify new facet-defining  inequalities \eqref{eq:single-phase} for $\conv(\FsetB)$ and (ii) the proposed two-phase lifting procedure in \cref{sec:twosteplifting} can identify the same facet-defining inequalities as those of \citet{Shi2022}.
\\[0.5em]{\it \textbf{Single-phase lifted inequalities for conv$(\FsetB)$}}\\[0.5em]
\indent Observe that in this binary case, $k$ takes the value $1$ in inequality \eqref{eq:seedInq};
$g(z) = f(z + \sum_{i \in S_1} a_i \mu_i) = f(z + a(S_1))$, where $a(T) : = \sum_{i \in T} a_i$ for $T \subseteq [n]$  and 
$\rho_s(k) = f(a(S_1 \cup s)) - f(a(S_1))$. 
As a result, the subadditive approximation $Z$ in \eqref{funZ} reduces to
\begin{small}
	\begin{equation*}
		Z(\delta) = f(\delta + a(S_1) - \ell a_s) + \ell (f(a(S_1 \cup s)) - f(a(S_1))) - f(a(S_1)), 
		\ \text{if} \ \ell a_s \leq \delta < (\ell + 1) a_s, \ \ell \in \Z,
	\end{equation*}
\end{small}
and the single-phase lifted inequality \eqref{eq:single-phase} is given by
\begin{equation}\label{eq:newineq}
	w \leq f(a(S_1)) + \sum_{i \in S_0}Z(a_i)x_i + \sum_{ i \in S_1}Z(-a_i)(1 - x_i) + \rho_s(1) x_s.
\end{equation}
From \cref{cor:facet-cond}, if $0 \leq a_i \leq a_s$ holds for all $i \in [n]\backslash s$, then \eqref{eq:newineq} is facet-defining for $\conv(\FsetB)$.
However, even when this condition does not hold,  \eqref{eq:newineq} could still be facet-defining for $\conv(\FsetB)$, as shown in \cref{ex1}.
\\[0.5em]{\it \textbf{Two-phase lifted inequalities for conv$(\FsetB)$}}\\[0.5em]
\indent We now present the two-phase lifted inequalities and establish their connections with the inequalities of \citet{Shi2022}.
Note that function $\zeta$ in \eqref{calc:zeta} reduces to
\begin{equation*}
	\zeta(\delta) = \left\{
	\begin{aligned}
		& f(\delta + a(S_1 \cup s)) - f(a(S_1 \cup s)), 
		&& \text{if}~ \delta < 0, \\
		& f(\delta + a(S_1)) - f(a(S_1)),
		&& \text{if}~ \delta \geq 0.
	\end{aligned}
	\right.
\end{equation*}
We next present the two-phase lifted inequalities \eqref{facet-ineq1} and \eqref{facet-ineq2} for $\conv(\FsetB)$, separately.
\begin{itemize}
	\item [(1)]	By $k = 1$ and \cref{lem:eta} (ii), the closed formula for $\eta$ is given by:
	\begin{equation*}
		\begin{footnotesize}
			\eta(\delta) = \left\{
			\begin{aligned}
				& {f}(\delta + A_{i + 1} + a(S_1)) - \sum_{j = 1}^{i + 1}\zeta(a_j) - {f}(a(S_1)), 
				&& \text{if}~  - A_{i + 1} < \delta \leq - A_{i}, ~ i = 0,\ldots,r, \\
				& {f}(\delta + a([n])) - \sum_{j = 1}^{r+1}\zeta(a_j) - {f}(a(S_1)), 
				&& \text{if}~ \delta \leq  - A_{r + 1},
			\end{aligned}
			\right.
		\end{footnotesize}
	\end{equation*}
	where we assume $S_0 \cup s = [r + 1]$ (with $r \in \Z_+$) with $a_1 \geq a_2 \geq \cdots \geq a_{r + 1}$, and let $A_0 = 0$ and $A_i = \sum_{j = 1}^{i} a_i$ for $i \in [r + 1]$.
	By \cref{thm:facet1} and $k = 1$, we obtain the facet-defining inequality for $\conv(\FsetB)$:
	\begin{equation}\label{eq:twophase1}
		w \leq f\left(a(S_1)\right) + \sum_{i \in S_0 \cup s} \left[ f(a(S_1 \cup i)) - f(a(S_1)) \right] x_i 
		+ \sum_{i \in S_1}\eta(-a_i) (1 - x_i),
	\end{equation}
	which is indeed the inequality proposed by \citet[Eq. (19)]{Shi2022}.
	\item [(2)] By $k = \mu_s = 1$ and \cref{thm:facet2} (ii), the closed formula for $\phi$ is given by:
	\begin{equation*}
		\begin{footnotesize}
			\phi(\delta) = \left\{
			\begin{aligned}
				& {f}(\delta - A_{i+1} + A_{d + 1}) - \sum_{j = 1}^{i+1}\zeta(-a_j) - {f}(A_{d + 1}), 
				&& ~\text{if}~ A_{i} \leq \delta < A_{i+1}, ~ i = 0 ,\ldots,d, \\ 
				& {f}(\delta) - \sum_{j = 1}^{d+1}\zeta(-a_j) - f(A_{d + 1}), 
				&& ~\text{if}~ \delta \geq A_{d+1},
			\end{aligned}
			\right.
		\end{footnotesize}
	\end{equation*}
	where we assume $S_1 \cup s = [d + 1]$ (with $d \in \Z_{+}$) with $a_1 \geq a_2 \geq \cdots \geq a_{d + 1}$, and let $A_0 = 0$ and $A_i = \sum_{j = 1}^{i} a_i$ for $i \in [d + 1]$.	
	By \cref{thm:facet2} and $k = \mu_s = 1$, we obtain the facet-defining inequality for $\conv(\FsetB)$:
	\begin{equation}\label{eq:twophase2}
		w \leq f\left(a(S_1 \cup s)\right) 
		+ \sum_{i \in S_0}\phi(a_i) x_i
		+ \sum_{i \in S_1 \cup s} [ f(a(S_1\cup s) - a_i) - f(a(S_1\cup s)) ] (1 - x_i) ,
	\end{equation}
	which is equivalent to the inequality proposed by \citet[Eq. (16)]{Shi2022}.
\end{itemize}

Two remarks on inequalities \eqref{eq:newineq}--\eqref{eq:twophase2} are in order.
First, for an arbitrary 0-1 vector $x \in \{0, 1\}^n$, letting $S_1 = \{ i \in [n]  \,:\, x_i = 1 \}$, $s \in [n]\backslash S_1$, and 
$S_0 = [n] \backslash(S_1 \cup s) $,  
then the right-hand sides of inequalities \eqref{eq:newineq}--\eqref{eq:twophase2}
at the point $x$ are all equal to $f(a(S_1))= f(a^\top x)$.
As a result,  $\FsetB$ can be completely characterized by either of the classes of inequalities \eqref{eq:newineq}, \eqref{eq:twophase1}, or \eqref{eq:twophase2}, that is, 
\begin{equation}\label{tmpEq}
	\begin{aligned}
		\FsetB &= \{ (w,x) \in \mathbb{R}\times\{ 0,1 \}^n  :  \eqref{eq:newineq} \}\\
		& = \{ (w,x) \in \mathbb{R}\times\{ 0,1 \}^n  :  \eqref{eq:twophase1} \} \\
		& =\{ (w,x) \in \mathbb{R}\times\{ 0,1 \}^n :  \eqref{eq:twophase2} \}.
	\end{aligned}
\end{equation}
\noindent Second, although the two-phase lifting procedure guarantees to produce facet-defining inequalities \eqref{eq:twophase1} and \eqref{eq:twophase2}, the single-phase lifting procedure can construct facet-defining inequalities \eqref{eq:newineq} for $\conv(\FsetB)$, which are different from \eqref{eq:twophase1} and \eqref{eq:twophase2}; see the following example for an illustration.
\begin{example}\label{ex2}
	We continue with \cref{ex1}. 
	Recall that $\FsetB = \{(w,x) \in \R \times \{0,1\}^4 \,:\,
	w \leq f(x_1+2x_2+2x_3+3x_4)\}$ and that the single-phase lifting procedure yields the facet-defining inequality \eqref{eq:facet-ex} for $\conv(\FsetB)$, which can be rewritten as:
	\begin{equation}\label{eq:singlephase}
		w \leq f(3) + (f(4) - f(3)) (x_1 + 2 x_2 + 2 x_3 + 3 x_4 - 3).
	\end{equation}
	Letting $f(z) = - e^{- z}$, we demonstrate that for any selection of $s$, $S_0$, and $S_1$, the two-phase lifted inequality \eqref{eq:twophase1} is different from \eqref{eq:singlephase}:
	\begin{itemize}
		\item [(i)] $(S_0 \cup s) \cap \{2, 3, 4\} = \varnothing$.
		It follows that $s = 1$, $S_0 = \varnothing$, and $S_1 = \{2, 3, 4\}$. Then $f(a(S_1 \cup 1)) - f(a(S_1)) = f(8) - f(7) \neq f(4) - f(3)$, which implies that the coefficient of $x_1$ in \eqref{eq:twophase1} is different from  that in \eqref{eq:singlephase}.
		\item [(ii)] $(S_0 \cup s) \cap \{2, 3, 4\} \neq \varnothing$.
		Given any $i \in (S_0 \cup s) \cap \{2, 3, 4\}$ and $S_1 = [4] \backslash (S_0 \cup s)$, it can be easily verified that $f(a(S_1 \cup i)) - f(a(S_1)) \neq a_i (f(4) - f(3))$, and therefore, the coefficient of $x_i$ in \eqref{eq:twophase1} differs from that in \eqref{eq:singlephase}.
	\end{itemize}
	By a similar argument, we can show that for any selection of $s$, $S_0$, and $S_1$, the two-phase lifted inequality \eqref{eq:twophase2} is different from inequality \eqref{eq:singlephase}.
	This shows that the single-phase lifting procedure can identify facet-defining  inequalities \eqref{eq:newineq} for $\conv(\FsetB)$ that are different from the two-phase inequalities \eqref{eq:twophase1} and \eqref{eq:twophase2}.
\end{example}

\subsection{Strong valid inequalities for the mixed-integer knapsack, mixed-integer polyhedral conic, and two-row mixed-integer sets} \label{subsec:CMIR}
Consider the mixed-integer set defined by two linear inequalities and box constraints:
\begin{equation*}
	\Tset : = \left\{(v, x) \in \R \times \Z^n \,:\, v \leq b^1 - (a^1)^\top x, \ v \leq b^2 - (a^2)^\top x, \ 0 \leq x_i \leq \mu_i, \ \forall \ i \in [n] \right\},
\end{equation*}
where $a^1, a^2 \in \R^n$ and $b^1, b^2 \in \R$.
The set $\Tset$ has a wide range of applications.
In particular,
(i) when $a^2 = \mathbf{0}$ and $b^2 = 0$, $\Tset$ reduces to the \emph{mixed-integer knapsack set} $\Kset$ \cite{Atamturk2003, Atamturk2005,Marchand1999,Atamturk2010a};
(ii) when $a^2 = -a^1$ and $b^2 = -b^1$, 
the two linear inequalities can be written as $v \leq -|b^1- (a^1)^\top x|$ 
and $\Tset$ becomes the mixed-integer polyhedral conic set $\Cset$ \cite{Atamturk2010b}.
Moreover, $\Tset$ can be seen as a two-constraint relaxation/substructure of mixed-integer programming (\MIP) problems.
\citet{Nemhauser1990} derived a class of valid inequalities for $\Tset$ (known as \MIR inequalities) and 
\citet{Bodur2017} applied them to solve the \MIP reformulation of two-stage stochastic integer programming problems.

In this subsection, we first establish a theoretical relation between the polyhedral structures of $\conv(\Tset)$ and the mixed-integer knapsack polyhedron $\conv(\Kset)$; that is,
there exists a one-to-one correspondence between the faces of $\conv(\Tset)$ and those of $\conv(\Kset)$ with
$a= a^1 - a^2$ and $b = b^1 - b^2$.
This relation enables to develop strong valid inequalities for $\conv(\Tset)$ (and its special case $\conv(\Cset)$) from the strong valid inequalities for $\conv(\Kset)$. 
We then apply our single- and two-phase lifting procedures in \cref{sec:twosteplifting,sec:lift-one-phase} 
to derive strong valid inequalities for the mixed-integer knapsack polyhedron $\conv(\Kset)$. 
Our results show that (i) the proposed single-phase lifting procedure in \cref{sec:lift-one-phase} can identify the \MIR \cite{Nemhauser1990,Marchand2001}, conic \MIR \cite{Atamturk2010b}, two-row \MIR inequalities \cite{Bodur2017} for $\conv(\Kset)$, $\conv(\Cset)$, and $\conv(\Tset)$, respectively;
(ii) the proposed two-phase lifting procedure in \cref{sec:twosteplifting} can not only identify the mixed-integer  cover and pack inequalities \cite{Atamturk2003} for $\conv(\Kset)$ but also enable to develop new facet-defining  inequalities \eqref{eq:single-phase} for $\conv(\Cset)$ and $\conv(\Tset)$. 

\subsubsection{The relation between the faces of $\conv(\Tset)$ and $\conv(\Kset)$}
The following proposition characterizes that there exists a one-to-one correspondence between the faces of $\conv(\Tset)$ and $\conv(\Kset)$ with $a = a^1 - a^2$ and $b = b^1 - b^2$ (where we recall that $\Kset$ is defined in \eqref{Kset}).
\begin{proposition}\label{face-correspond}
	Inequality $\beta v \leq \pi^\top x - \pi_0$
	defines a $k$-dimensional face of $\conv(\Tset)$ if and only if $\beta w \leq (\beta a^2 + \pi)^\top x  - \beta b^2 - \pi_0$
	defines a $k$-dimensional face of $\conv(\Kset)$ with $a = a^1 - a^2$ and $b = b^1 - b^2$.
\end{proposition}
\begin{proof}
	Observe that applying an affine transformation on $\Tset$ where $v$ is transformed by 
	\[
		w = v+(a^2)^\top x-b^2,
	\]
	we obtain the set $\Kset$ with $a = a^1 - a^2$ and $b = b^1 - b^2$.
	The statement then follows immediately from the facts that
	(i) $(v, x) \in \Tset$ holds if and only if $(v+(a^2)^\top x-b^2, x) \in \Kset$ holds; 
	(ii) inequality $\beta v \leq \pi^\top x - \pi_0$ holds with equality at point $(v, x)$ if and only if inequality $\beta w \leq (\beta a^2 + \pi)^\top x - \beta b^2 - \pi_0$ holds with equality at point $(v+(a^2)^\top x-b^2, x)$;
	and (iii) points $(v^1, x^1), \ldots, (v^{k+1}, x^{k+1}) $ are affinely independent 
	if and only if points $(v^1+(a^2)^\top x^1-b^2, x^1), \ldots, (v^{k+1}+(a^2)^\top x^{k+1}-b^2, x^{k+1})$ are affinely independent.
	\qedhere
\end{proof}
\cref{face-correspond} establishes a general method for deriving strong valid inequalities for $\conv(\Tset)$ (and also its special case $\conv(\Cset)$) from those for $\conv(\Kset)$.

\subsubsection{Single- and two-phase inequalities for $\conv(\Kset)$, $\conv(\Cset)$, and $\conv(\Tset)$}\label{subsubsec:MIR-Knap}
Let $f(z) = \min\{0, b - z\}$, which is concave on $\R$. 
Then we can write $\Kset$ in the form of $\Fset$ as in \eqref{Fset}:
\begin{equation*}
	\Kset = \left\{
	(w, x) \in \R \times \Z^n \,:\,  w \leq f\left(a^\top x \right), ~ 0 \leq x_i \leq \mu_i, ~ \forall ~ i \in [n]
	\right\}.
\end{equation*}
Without loss of generality, we assume that $a_i \geq 0$ holds for all $i \in [n]$, since any bounded variables can be complemented.
In the following, we will apply our single- and two-phase lifting procedures 
in  \cref{sec:lift-one-phase,sec:twosteplifting} to derive strong valid inequalities for $\conv(\Kset)$
(thereby for $\conv(\Tset)$) and establish their connections with known results in the literature.
To do this, let $s$, $S_0$, and $S_1$ be such that 
$\theta : = b - \sum_{i \in S_1} a_i \mu_i > 0$, $\sum_{i \in S_1 \cup s}a_i \mu_i - b = a_s \mu_s - \theta > 0$,
$\theta / a_s \notin \Z_+$, $k := \ceil{\theta / a_s}$, and $\sigma : = \theta - \floor{\theta / a_s} a_s$.
Then, it follows that 
$g(z) = f(z + \sum_{i \in S_1}a_i \mu_i) = f(z + b - \theta) = \min\{0, \theta - z\}$,
$k \in [\mu_s]$, $\sigma = \theta - (k - 1) a_s$, and
\begin{equation}\label{coef-compute}
	\begin{aligned}
		& g(k a_s) 
		= \min \left\{0, \theta - k a_s\right\} 
		= \theta - k a_s
		= \sigma - a_s, \\
		& g((k - 1) a_s) 
		= \min \left\{0, \theta - (k - 1) a_s\right\} 
		= 0,\\
		& \rho_s(k) = g(k a_s) - g((k - 1) a_s) = \sigma - a_s.
	\end{aligned}
\end{equation}
{\bf \textit{Single-phase lifted inequalities for $\conv(\Kset)$, $\conv(\Cset)$, and $\conv(\Tset)$}}\\[0.5em]
\noindent 
Letting $\ell \in \Z$ be such that $\ell a_s \leq \delta < (\ell + 1) a_s$, 
or equivalently, $\ell = \floor{\delta / a_s}$, then 
\begin{equation*}
	\begin{aligned}
		g(\delta + (k - \ell - 1) a_s)
		& = \min \left\{0, 
		\theta - (k - \ell - 1) a_s - \delta 
		\right\} \\
		& = \min \left\{0, 
		\theta - (k - 1) a_s + \floor{\delta / a_s} a_s - {\delta}
		\right\} \\
		&= \min \left\{0, 
		\sigma - \sigma_\delta
		\right\},
	\end{aligned}
\end{equation*}
where we define $\sigma_\delta := \delta - \floor{\delta / a_s} a_s$.
Together with \eqref{coef-compute}, the subadditive  approximation $Z$ in \eqref{funZ} reduces to
\begin{equation*}
	\begin{aligned}
		Z(\delta) & =
		\min \left\{0, 
		\sigma - \sigma_\delta
		\right\} + \floor{\delta / a_s} (\sigma - a_s)
		\\
		& = 
		\min \left\{0, 
		\sigma - \sigma_\delta
		\right\} + \floor{\delta / a_s} \sigma + (\delta - \floor{\delta / a_s} a_s) - \delta
		\\
		& = -\delta + \min \left\{ \sigma, \sigma_\delta \right\} +  \floor{ {\delta} / {a_s} } \sigma.
	\end{aligned}
\end{equation*}
The single-phase lifted inequality \eqref{eq:single-phase} is given by
\begin{equation}\label{eq:single-Kset}
		\begin{aligned}
			{w} &\leq \sum_{i \in S_0 \cup s} \left( - a_i + \min \left\{ \sigma, \sigma_{a_i} \right\} +  \floor{ {a_i} / {a_s} } \sigma \right) x_i 
			\\
			& \qquad 
			+ \sum_{i \in S_1} \left( a_i + \min \left\{ \sigma, \sigma_{-a_i} \right\} +  \floor{ {-a_i} / {a_s} } \sigma \right) (\mu_i - x_i) - (k - 1) (\sigma - a_s),
		\end{aligned}
\end{equation}
which is valid for $\conv(\Kset)$ and equivalent to the well-known MIR inequality \cite{Nemhauser1990,Marchand2001}.
\begin{remark}\label{rem:single-phase}
	Applying \cref{face-correspond} to inequality \eqref{eq:single-Kset} yields the single-phase lifted inequality for $\conv(\Tset)$, which is equivalent to the inequality proposed by \citet[Eq. (16)]{Bodur2017}.
	Moreover, when $\conv(\Tset)$ is a from of $\conv(\Cset)$ where $\Cset$ is defined in \eqref{Cset}, then the single-phase lifted inequality specializes to the conic \MIR inequality of \citet{Atamturk2010b}.
\end{remark}

\noindent{\bf \textit{Two-phase lifted inequalities for $\conv(\Kset)$, $\conv(\Cset)$, and $\conv(\Tset)$}}\\[0.5em]
\noindent
Using a derivation analogous to that for $Z$, the first-phase lifting function $\zeta$ in \eqref{calc:zeta} can be given by
\begin{equation*}
	\zeta(\delta) = \left\{
	\begin{aligned}
		& (k - \mu_s - 1) (\sigma - a_s),
		&& \text{if}~ \floor{\delta / a_s} < k - \mu_s - 1, \\
		& -\delta + \min \left\{ \sigma, \sigma_\delta \right\} +  \floor{ {\delta} / {a_s} } \sigma,
		&& \text{if}~ k - \mu_s - 1 \leq \floor{\delta / a_s} < k, \\
		& -\delta + k \sigma,
		&& \text{if}~ \floor{\delta / a_s} \geq k,
	\end{aligned}
	\right.
\end{equation*}
{We next present the two-phase lifted inequalities \eqref{facet-ineq1}--\eqref{eq:lifted-ineq1} and \eqref{eq:lifted-ineq2}--\eqref{facet-ineq2} for $\conv(\Kset)$, separately.}
\begin{itemize}
	\item [(1)] Recall that $S_0^+ = \{i \in S_0 \,:\, a_i \geq k a_s\} = [r]$ with  
	$a_1 \geq \cdots \geq a_{r}$.
	Then, $\zeta(a_j) = - a_j + k \sigma$ for $j \in [r]$.
	\begin{claim}
		The function $\eta^U$ in \eqref{etaU} reduces to
		{\footnotesize
			\begin{equation*}
				\eta^U(\delta)= \left\{
				\begin{aligned}
					& -\delta-\mu_{(i+1)(q-1)}k\sigma, \\
					& \hspace{1cm} \text{if} \ \theta-A_{(i+1)q}<\delta\le -A_{(i+1)(q-1)}, \\
					& \hspace{1cm} \ i=0,\ldots,r-1 \ \text{and} \ q=1,\ldots,\mu_{i+1}, \\
					& A_{(i+1) q}-\mu_{(i+1) (q-1)}k\sigma-\theta, \\
					& \hspace{1cm} \text{if} \ k a_s-A_{(i+1) q}<\delta\le \theta-A_{(i+1) q}, \\
					& \hspace{1cm} \ {i=0,\ldots,r-1 \ \text{and} \ q=1,\ldots,\mu_{i+1}},\\
					& -\delta-(\mu_{(i+1)(q-1)}k+\ell)\sigma, \\
					& \hspace{1cm} \text{if} \ \theta-A_{(i+1) q}-\ell a_s<\delta\le (k-\ell)a_s-A_{(i+1) q},\\
					& \hspace{1cm} \ {i=0,\ldots,r-1, \ q=1,\ldots,\mu_{i+1}, \ \text{and} \ \ell=0,\ldots,k-1},\\
					& \hspace{1cm} \ \text{or} \ {i=r-1, \ q=\mu_{r}, \ \text{and} \ \ell=k, k+1, \ldots},\\
					& A_{(i+1) q}-\mu_{(i+1) (q-1)}k\sigma-\theta-\ell(\sigma-a_s), \\
					& \hspace{1cm} \text{if} \ (k-\ell-1)a_s-A_{(i+1) q}<\delta\le \theta-A_{(i+1) q}-\ell a_s,\\
					& \hspace{1cm} \ {i=0,\ldots,r-1, \ q=1,\ldots,\mu_{i+1}, \ \text{and} \ \ell=0,\ldots,k-1},\\
					& \hspace{1cm} \ \text{or} \ {i=r-1, \ q=\mu_{r}, \ \text{and} \ \ell=k, k+1, \ldots},
				\end{aligned}
				\right.
			\end{equation*}
		}%
		where 
		{$\mu_{i q} = \sum_{j = 1}^{i - 1} \mu_j + q$, $A_{i q} = \sum_{j = 1}^{i - 1} \mu_j a_j + q a_i$ 
		for $i \in [r]$ and $q \in \{0, 1, \ldots, \mu_i\}$}.
	\end{claim}
	\begin{proof}
		For $i \in \{0, \ldots, r-1\}$ and $q \in [\mu_{i+1}]$, we have
		\begin{align*}
			& - \sum_{j = 1}^{i} \mu_j \zeta(a_j) - q \zeta(a_{i+1}) + k \rho_s(k) - g(k a_s)\\
			& \qquad = - \sum_{j = 1}^{i} \mu_j (-a_j + k \sigma) - q (-a_{i+1} + k \sigma) + k  (\sigma - a_s)-  (\sigma - a_s)\\
			& \qquad= \left( \sum_{j = 1}^{i}\mu_j a_j + q a_{i+1} \right)
			- k \sigma \left( \sum_{j = 1}^{i }\mu_j + q - 1 \right) -(\sigma + (k - 1) a_s) \\
			& \qquad= A_{(i+1) q} - \mu_{(i+1) (q - 1)} k \sigma - \theta.
		\end{align*}
		If $k a_s - A_{(i+1) q} < \delta \leq - A_{(i+1) (q - 1)}$ holds for some $i \in \{0, \ldots, r - 1\}$ and $q \in [\mu_{i+1}]$ (i.e., the first case in \eqref{etaU}),
		then we have
		\begin{equation*}
			\begin{aligned}
				\eta^U(\delta) & = g(\delta + A_{(i+1) q}) - \sum_{j = 1}^{i} \mu_j \zeta(a_j) - q \zeta(a_{i+1}) + k \rho_s(k) - g(k a_s) \\
				& = \min \left\{ 0, \theta - \delta - A_{(i+1) q} \right\} + A_{(i+1) q} - \mu_{(i+1) (q - 1)} k \sigma - \theta \\
				& = \left\{
				\begin{aligned}
					& - \delta- \mu_{(i+1) (q - 1)} k \sigma, 
					&& ~\text{if}~ \theta - A_{(i+1) q} < \delta \leq - A_{(i+1) (q-1)}, \\
					& A_{(i+1) q} - \mu_{(i+1) (q - 1)} k \sigma - \theta, 
					&& ~\text{if}~ k a_s - A_{(i+1) q} < \delta \leq \theta - A_{(i+1) q}. 
				\end{aligned}
				\right.
			\end{aligned}
		\end{equation*}
		For the second case of \eqref{etaU}, the proof is similar and omitted for brevity.
	\end{proof}
	Using the upper subadditive approximation $\eta^U$, we obtain the following two-phase lifted inequality: 
	\begin{equation}
		\label{eq:two1-Kset}
		\begin{aligned}
			w & \leq 
			\sum_{i \in S_0 \backslash S_0^+} \left( -a_i 
			+ \min\left\{\sigma, \sigma_{a_i} \right\} + \floor{a_i / a_s} \sigma  \right) x_i 
			+ \sum_{i \in S_0^+} (-a_i + k \sigma) x_i\\
			& \qquad
			+ \sum_{i \in S_1} \eta^U(-a_i)(\mu_i - x_i)
			+ (\sigma - a_s) x_s - (k - 1) (\sigma - a_s),
		\end{aligned}
	\end{equation}
	which is valid for $\conv(\Kset)$ and facet-defining under the mild condition in \cref{thm:facet1-cond}.
	If $S_0^+ = S_0$ or $k = 1$, then the exact lifting function $\eta$ admits a closed formula derived from \eqref{etaL} or \eqref{etaL-k=1} (omitted here for brevity). 
	This enables to obtain a strengthened version of \eqref{eq:two1-Kset}:
	\begin{equation}\label{eq:two1.2-Kset}
			\begin{aligned}
				w & \leq 
				\sum_{i \in S_0 \backslash S_0^+} \left( -a_i 
				+ \min\left\{\sigma, \sigma_{a_i} \right\} + \floor{a_i / a_s} \sigma  \right) x_i 
				+ \sum_{i \in S_0^+} (-a_i + k \sigma) x_i\\
				& \qquad
				+ \sum_{i \in S_1} \eta(-a_i)(\mu_i - x_i)
				+ (\sigma - a_s) x_s - (k - 1) (\sigma - a_s),
			\end{aligned}
	\end{equation}
	which is facet-defining for $\conv(\Kset)$.
	Both \eqref{eq:two1-Kset} and its strengthened version \eqref{eq:two1.2-Kset} are equivalent to the mixed-integer pack inequalities of \citet{Atamturk2003}.	
	\item [(2)] Recall that $S_1^+ = \{i \in S_1 \,:\, a_i \geq (\mu_s - k + 1) a_s \} = [d]$ with $a_1 \geq \cdots \geq a_{d}$.
	Thus, $\zeta(-a_j) = (k - \mu_s - 1) (\sigma - a_s)$ for $j \in [d]$.
	\begin{claim}
		The function $\phi^U$ in \eqref{omegaU} reduces to
		\begin{footnotesize}
			\begin{equation*}
				\phi^U(\delta) = \left\{ \ 
				\begin{aligned}
					& (\mu_s - k + 1) (\sigma - a_s) \mu_{(i+1) (q - 1)}, \\
					& \qquad ~\text{if}~ A_{(i+1) (q - 1)} \leq \delta < A_{(i+1) q} + \theta - \mu_s a_s, \\
					& \qquad ~~ i = 0, \ldots, d - 1 \ \text{and} \ q = 1, \ldots, \mu_{i+1}, \\
					& - \delta + \theta - \mu_s a_s + A_{(i+1) q} + (\mu_s - k + 1) (\sigma - a_s) \mu_{(i+1) (q - 1)}, \\
					& \qquad ~\text{if}~ A_{(i+1) q} + \theta - \mu_s a_s \leq \delta < A_{(i+1) q} + (k - \mu_s - 1) a_s, \\
					& \qquad ~~ i = 0, \ldots, d - 1 \ \text{and} \ q = 1, \ldots, \mu_{i+1}, \\
					& ((\mu_s - k + 1) \mu_{(i+1) (q - 1)} + \ell) (\sigma - a_s), \\
					& \qquad ~\text{if}~ A_{(i+1) q} + (k - \mu_s + \ell - 1) a_s \leq \delta < A_{(i+1) q} + \theta - (\mu_s - \ell) a_s, \\
					& \qquad ~~ i = 0, \ldots, d-1, ~ q = 1, \ldots, \mu_{i+1}, \ \text{and} \ \ell = 0, \ldots, \mu_s - k, \\
					& \qquad ~~ \text{or} \ i = d-1, ~ q = \mu_{d}, \ \text{and} \ \ell = \mu_s - k + 1, \mu_s - k + 2, \\
					& -\delta + \theta - (\mu_s - \ell) a_s + A_{(i+1) q} + ((\mu_s - k + 1) \mu_{(i+1) (q - 1)} + \ell) (\sigma - a_s), \\
					& \qquad ~\text{if}~  A_{(i+1) q} + \theta - (\mu_s - \ell) a_s \leq \delta < A_{(i+1) q} + (k - \mu_s + \ell) a_s, \\
					& \qquad ~~ i = 0, \ldots, d - 1, ~ q = 1, \ldots, \mu_{i+1}, \ \text{and} \ \ell = 1, \ldots, \mu_s - k, \\
					& \qquad ~~  \text{or} \ i = d-1, ~ q = \mu_{d}, \ \text{and} \ \ell = \mu_s - k + 1, \mu_s - k + 2, \\
				\end{aligned}
				\right.
			\end{equation*}
		\end{footnotesize}%
		where 
		{$\mu_{i q} = \sum_{j = 1}^{i-1} \mu_j + q$, $A_{i q} = \sum_{j = 1}^{i-1} \mu_j a_j + q a_i$ }
		for $i \in [d]$ and $q \in \{0, 1, \ldots, \mu_i \}$.
	\end{claim}
	\begin{proof}
		For $i \in \{0, \ldots, d - 1\}$ and $q \in [\mu_{i+1}]$, we have
			\begin{align*}
				& - \sum_{j = 1}^{i} \mu_j \zeta(- a_j) - q \zeta(-a_{i+1}) + (k - \mu_s) \rho_s(k) - g(k a_s) \\
				& \qquad = - \sum_{j = 1}^{i} \mu_j (k - \mu_s - 1) (\sigma - a_s) - q (k - \mu_s - 1) (\sigma - a_s) + (k - \mu_s - 1) (\sigma - a_s) \\
				& \qquad = (\mu_s - k + 1) (\sigma - a_s) \left(
				\sum_{j = 1}^{i} \mu_j + q - 1
				\right)
				= (\mu_s - k + 1) (\sigma - a_s) \mu_{(i+1) (q - 1)} .
			\end{align*}
		If $A_{(i+1) (q - 1)} \leq \delta < A_{(i+1) q} + (k - \mu_s - 1) a_s$ holds for some $i \in \{0,\ldots,d-1\}$ and $q \in [\mu_{i+1}]$ (i.e., the first case in \eqref{omegaU}), then 
		\begin{equation*}
			\begin{aligned}
				\phi^U(\delta) & = g(\delta + \mu_s a_s - A_{(i+1) q}) - \sum_{j = 1}^{i} \mu_j \zeta(-a_j) - q \zeta(-a_{i+1}) + (k - \mu_s) \rho_s(k) - g(k a_s) \\
				& = \min \left\{ 0, \theta - \delta - \mu_s a_s + A_{(i+1) q}\right\} 
				+ (\mu_s - k + 1) (\sigma - a_s) \mu_{(i+1) (q - 1)} \\
				& = \left\{
				\begin{aligned}
					& (\mu_s - k + 1) (\sigma - a_s) \mu_{(i+1) (q - 1)}, \\
					& \hspace{3cm}
					~\text{if}~ A_{(i+1) (q - 1)} \leq \delta < A_{(i+1) q} + \theta - \mu_s a_s, \\
					& -\delta + \theta - \mu_s a_s + A_{(i+1) q} + (\mu_s - k + 1) (\sigma - a_s) \mu_{(i+1) (q - 1)}, \\
					& \hspace{3cm} ~\text{if}~ A_{(i+1) q} + \theta - \mu_s a_s \leq \delta < A_{(i+1) q} + (k - \mu_s - 1) a_s. 
				\end{aligned}
				\right.
			\end{aligned}
		\end{equation*}
	For the second case of \eqref{omegaU}, the proof is similar and omitted for brevity.
	\end{proof}
	
	Using the upper subadditive approximation $\phi^U$, we obtain the following two-phase lifted inequality: 
	\begin{equation}\label{eq:two2-Kset}
			\begin{aligned}
				w & \leq \sum_{i \in S_1 \backslash S_1^+} \left( a_i + \min\left\{ \sigma, \sigma_{-a_i} \right\} + \floor{-a_i / a_s} \sigma \right) (\mu_i - x_i) \\
				& \qquad
				+ \sum_{i \in S_1^+} (k - \mu_s - 1) (\sigma - a_s) (\mu_i - x_i) \\
				& \qquad
				+ \sum_{i \in S_0} \phi^U(a_i) x_i 
				+ (\sigma - a_s) x_s - (k - 1) (\sigma - a_s),
			\end{aligned}
	\end{equation}
	which is valid for $\conv(\Kset)$ and facet-defining under mild conditions.
	If $S_1^+ = S_1$ or $k = \mu_s$, then the exact lifting function $\phi$ admits a closed formula derived from \eqref{phiL} or \eqref{eq:phi} (omitted here for brevity). 
	This leads to the strengthened version of \eqref{eq:two2-Kset}:
	\begin{equation}\label{eq:two2.2-Kset}
		\begin{aligned}
			w & \leq \sum_{i \in S_1 \backslash S_1^+} \left( a_i + \min\left\{ \sigma, \sigma_{-a_i} \right\} + \floor{-a_i / a_s} \sigma \right) (\mu_i - x_i) \\
			& \qquad
			+ \sum_{i \in S_1^+} (k - \mu_s - 1) (\sigma - a_s) (\mu_i - x_i) \\
			& \qquad
			+ \sum_{i \in S_0} \phi(a_i) x_i 
			+ (\sigma - a_s) x_s - (k - 1) (\sigma - a_s),
		\end{aligned}
	\end{equation}
	which is facet-defining for $\conv(\Kset)$.
	Both \eqref{eq:two2-Kset} and its strengthened version \eqref{eq:two2.2-Kset} are equivalent to the mixed-integer cover inequality of Atamt\"urk \cite{Atamturk2003, Atamturk2005}.
\end{itemize}
\begin{remark}
	Applying \cref{face-correspond} to inequalities \eqref{eq:two1-Kset}--\eqref{eq:two2.2-Kset} yields two-phase lifted inequalities for $\conv(\Tset)$ and its specialization $\conv(\Cset)$.
	The resultant inequalities are strong valid inequalities for both polyhedra and therefore may applied to improve the computational performance of branch-and-cut algorithms for solving related problems in which $\Tset$ or $\Cset$ arises as a substructure.
\end{remark}

\section{Computational experiments}\label{sec:computational-result}
In this section, we evaluate the effectiveness of our proposed single- and two-phase lifted inequalities by using them as cutting planes to solve two classes of practical problems: 
the expected utility maximization \cite{Ahmed2011,Yu2017,Shi2022}
and weapon target assignment problems \cite{Manne1958,Ahuja2007,Andersen2022,Lu20211,Bertsimas2025}.
Our computational experiments are performed on a cluster of Intel(R) Xeon (R) Gold 6230R CPU @ 2.10GHz computers using CPLEX 20.1.0 as the solver.
We implement our proposed inequalities as cuts using the callback function of CPLEX.
CPLEX is set to run in a single-threaded mode, with a time limit of $3600$ seconds and 
a relative gap tolerance of $0\%$.
Unless otherwise stated, all other CPLEX's parameters are set to their default values.

In the following, we first detail how the proposed lifted inequalities
are dynamically separated within a branch-and-cut algorithm, and then report computational results using the proposed lifted inequalities to solve the two problems.
Computational instances used in this paper are available at \url{https://doi.org/10.5281/zenodo.18296729}.

\subsection{Separation}\label{subsec:separation}
Given a point $(\bar{w}, \bar{x}) \in \R \times \R^n$ with $0 \leq \bar{x}_i \leq \mu_i$ for all $i \in [n]$ (encountered in the branch-and-cut algorithm), 
the separation problem of the single-phase lifted inequalities \eqref{eq:single-phase}
or the two-phase lifted inequalities \eqref{eq:lifted-ineq1} and \eqref{eq:lifted-ineq2} 
(or their strengthened versions \eqref{facet-ineq1} and \eqref{facet-ineq2}), respectively, 
attempts to find corresponding inequalities violated by $(\bar{w}, \bar{x})$ or prove that none exists.
In this subsection, we will detail the procedure to separate these inequalities.
For simplicity, we will focus on the separation of the single-phase lifted inequalities \eqref{eq:single-phase};
the separation of the two-phase lifted inequalities \eqref{eq:lifted-ineq1} and \eqref{eq:lifted-ineq2} can be described in a similar manner.

The separation problem of the single-phase lifted inequalities \eqref{eq:single-phase} is equivalent to determining parameters $s$, $k$, $S_0$, and $S_1$ 
that minimize the right-hand side of \eqref{eq:single-phase}.
We begin by discussing the following two cases in which exact separation of the single-phase lifted inequalities \eqref{eq:single-phase} can be achieved.
\begin{itemize}
	\item [(i)] If all components of $\bar{x}$ are at their bounds, i.e., 
	$\bar{x}_i \in \{0, \mu_i\}$ for all $i \in [n]$, 
	then we let $s \in \argmax_{i\in [n]} a_i$, 
	$k = 1$ if $\bar{x}_s = 0$ and $k = \mu_s$ otherwise,
	$S_0 = \left\{ i \in [n] \backslash s \,:\, \bar{x}_i = 0 \right\}$, 
	and $S_1 = \left\{ i \in [n] \backslash s \,:\, \bar{x}_i = \mu_i \right\}$.
	\item [(ii)] If there exists exactly one index $s_0 \in [n]$ such that $\bar{x}_{s_0} \in \{1, \ldots, \mu_{s_0} - 1\}$
	and $\bar{x}_i \in \{0, \mu_i\}$ for all $i \in [n]\backslash s_0$,
	then we set $s=s_0$, $k = \bar{x}_{s_0}$, $S_0 = \left\{ i \in [n] \,:\, \bar{x}_i = 0 \right\}$, 
	and $S_1 = \left\{ i \in [n] \,:\, \bar{x}_i = \mu_i \right\}$.
\end{itemize}
For both cases, the right-hand side of the single-phase lifted inequality \eqref{eq:single-phase} at point $\bar{x}$ reduces to $f(a^\top \bar{x})$.
This, together with the fact that if $\bar{x}_i \in \Z \cap [0, \mu_i]$ for all $i \in [n]$,
then $(\bar{w}, \bar{x}) \notin \conv(\Fset)$ if and only if $\bar{w} > f(a^\top \bar{x})$, 
implies that if $\bar{w} > f(a^\top \bar{x})$, 
then the corresponding single-phase lifted inequality \eqref{eq:single-phase} is violated by $(\bar{w}, \bar{x})$; otherwise, no violated one exists.
Moreover, under case (i), the selection of $s$ and $k$ can yield facet-defining single-phase lifted inequalities \eqref{eq:single-phase} for $\conv(\Fset)$; see \cref{cor:facet-cond}.

In general, for a point $(\bar{w}, \bar{x}) \in \R \times \R^n$ with $0 \leq \bar{x}_i \leq \mu_i$ for all $i \in [n]$,
determining the most violated single-phase lifted inequalities, however, is challenging.
To the best of our knowledge, even for the special case $f(a^\top x)= \min\{0, b-a^\top x\}$ in \cref{subsubsec:MIR-Knap}, 
it is still unclear whether the separation of  the single-phase lifted inequalities \eqref{eq:single-phase} (i.e., the \MIR inequalities for the mixed-integer knapsack set) can be conducted in polynomial time.
To address this, we extend the heuristic algorithm of \cite{Marchand2001} for the separation of \MIR inequalities to our case. 
Specifically, we first construct the candidate set $C=\{i \in [n] \,:\, 0 < \bar{x}_i < \mu_i \}$ for parameter $s$.
Then, for each $s\in C$, we let $k = \ceil{\bar{x}_s}$,
{$S_1 = \left\{ i \in [n] \backslash s \,:\,\bar{x}_i \geq \mu_i / 2 \right\}$},
and $S_0 = [n] \backslash (S_1 \cup s)$ and check whether the corresponding single-phase lifted inequality \eqref{eq:single-phase} is violated by $(\bar{w}, \bar{x})$.
In the end, if some violated inequalities are identified, the most violated one will be added to the formulation.

\subsection{Expected utility maximization problem}\label{subsec:result-EUM}
The first data set for our computational study comes from an expected utility maximization problem,
which is the benchmark used in \cite{Ahmed2011,Yu2017,Shi2022} 
to evaluate the effectiveness of the (approximate or exact) lifted submodular inequalities for $\FsetB$ or its variants.
The expected utility maximization problem involves a set $[n]$ of investment options with the capital requirement $a_i \in \R_+$ for each option $i \in [n]$,
a set $[m]$ of scenarios with a probability $\pi_j \in [0, 1]$
and a future value of investment $v_j \in \R^n_+$ for each scenario $j \in [m]$.
Let  $f(z) = 1 - \exp(-z / \lambda)$ be the utility function representing an investor's risk preferences where $\lambda$ is the risk tolerance parameter.
The problem attempts to maximize the expected utility subject to an available budget (normalized as 1) and can be formulated as follows:
\begin{equation}\label{prob:EUM}
	\max_{w, \, x}\left\{
	\sum_{j = 1}^{m} \pi_j w_j \,:\, w_j \leq 1 - \exp\left( - \frac{v_j^\top x}{\lambda} \right), \ \forall \ j \in [m], \ a^\top x \leq 1, \ x \in \{0, 1\}^n
	\right\}.
\end{equation}

In our experiments, we adopt instance generation settings similar to those described in \cite{Ahmed2011,Yu2017,Shi2022}.
Specifically, the capital requirements $a_i$ are uniformly generated from $[0.1, 0.15]$
and the probability of each scenario is $\pi_j = \frac{1}{m}$.
The value of investment $i$ under scenario $j$ is given by: 
\begin{equation*}
	v_{j i} = p_i \cdot \exp\left( \alpha_i + \beta_i \ln f_{j} + \epsilon_{j i} \right),
\end{equation*}
where $p_i$, $\alpha_i$, and $\beta_i $ are uniformly generated from $[0, 0.2]$,  $[0.05, 0.1]$, and  $[0, 1]$, respectively, 
and $\ln f_i$ and  $\epsilon_{j i}$ are uniformly chosen from the normal distributions $N(0.05, 0.0025)$ and $N(0, 0.0025)$, respectively.

Note that as shown in \cref{subsec:special-binary}, for the 0-1 case (i.e., $\mu_i = 1$ for all $i \in [n]$),
the single- 
and two-phase lifted inequalities 
reduce to \eqref{eq:newineq} and \eqref{eq:twophase1}--\eqref{eq:twophase2}, respectively.
Also note that by \eqref{tmpEq} and the fact that \eqref{eq:newineq}--\eqref{eq:twophase2} can be separated exactly for integral points $x \in \{0, 1\}^n$ (see \cref{subsec:separation}),
our proposed branch-and-cut algorithm based on single- or two-phase lifted inequalities will terminate with a correct solution for the expected utility maximization problem \eqref{prob:EUM}.

\cref{tab:EUM} presents the computational results using the single-phase lifted inequalities \eqref{eq:newineq}, 
the two-phase lifted inequalities \eqref{eq:twophase1}--\eqref{eq:twophase2}, 
and the combination of the two classes of inequalities to solve the expected utility maximization problem \eqref{prob:EUM}.
For each fixed parameter $(n,m,\lambda)$, 
we randomly generate 10 instances.
The first three columns of the table specify these problem parameters.
Under each setting,
we report the following statistics, averaged across the 10 trials:
the total number of cuts added (C), 
the number of explored branch-and-cut nodes (N), 
the solution time in seconds (T),
the time spent in separating cuts in seconds (ST),
and the linear programming (\LP) relaxation gap at the root node (Rgap), 
computed as $(z_{\text{R}} - z_{\text{OPT}}) / \vert z_{\text{OPT}} \vert \times 100\%$.
Here, $z_{\text{OPT}}$ denotes the objective value of the best feasible solution found across all three branch-and-cut implementations,
and $z_{\text{R}}$ denotes the LP relaxation bound obtained at the root node.
For cases that some instances are not solved to optimality within the time limit of $3600$ seconds,
we also report the average end gap (Egap), computed by $(z_{\text{UB}} - z_{\text{OPT}}) / \vert z_{\text{OPT}} \vert \times 100\%$,
where $z_{\text{UB}}$ denotes the best known upper bound achieved by the corresponding setting at termination.
This average end gap is computed across all 10 trials, 
including the instances solved to optimality (for which the end gap is $0\%$).
A superscript next to the ``Egap'' value indicates the number of instances (out of 10) that cannot be solved to optimality under the corresponding setting.
At the end of the table,
we report the average results for instances that can be solved by all three settings.

\begin{sidewaystable}[htbp]
	\centering
	\caption{Computational results of the single-phase lifted inequalities \eqref{eq:newineq},
		two-phase lifted inequalities \eqref{eq:twophase1}--\eqref{eq:twophase2}, 
		and the combination of the two classes of inequalities for the expected utility maximization problem \eqref{prob:EUM}.}
	\label{tab:EUM}
	\fontsize{7.5pt}{9}\selectfont
	\renewcommand{\arraystretch}{1.5}
	\addtolength{\tabcolsep}{-2pt}
	\begin{tabular}{ccc
			ccccc
			cccccc
			ccccc}
		\toprule
		\multicolumn{1}{c}{$n$} & 
		\multicolumn{1}{c}{$m$} & 
		\multicolumn{1}{c}{$\lambda$} & 
		\multicolumn{5}{c}{\textbf{Single-phase lifted ineqs. \eqref{eq:newineq}}} &
		\multicolumn{6}{c}{
			\begin{tabular}{@{}c@{}}
				\textbf{Two-phase lifted ineqs. \eqref{eq:twophase1}--\eqref{eq:twophase2}} 
				(Ineqs. from \cite{Shi2022})
			\end{tabular}
		} &
		\multicolumn{5}{c}{\textbf{Combination}} \\
		\cmidrule(lr){4-8}\cmidrule(lr){9-14}\cmidrule(lr){15-19}
		& & &
		\multicolumn{1}{c}{C} & 
		\multicolumn{1}{c}{N} & 
		\multicolumn{1}{c}{T} & 
		\multicolumn{1}{c}{ST} & 
		\multicolumn{1}{c}{Rgap (\%)} &
		\multicolumn{1}{c}{C} & 
		\multicolumn{1}{c}{N} & 
		\multicolumn{1}{c}{T} & 
		\multicolumn{1}{c}{ST} & 
		\multicolumn{1}{c}{Rgap (\%)} &
		\multicolumn{1}{c}{Egap (\%)} &
		\multicolumn{1}{c}{C} & 
		\multicolumn{1}{c}{N} & 
		\multicolumn{1}{c}{T} & 
		\multicolumn{1}{c}{ST} & 
		\multicolumn{1}{c}{Rgap (\%)} \\
		\midrule 
		
		1000 & 500 & 0.4 & 10802 & 26 & 7.98 & 2.59 & 0.23 & 53072 & 40494 & 1350.37 & 13.78 & 12.44 & 3.06$^{3}$ & 12408 & 30 & 12.26 & 5.74 & 0.32 \\ 
		& & 0.6 & 9391 & 11 & 6.92 & 2.33 & 0.13 & 41940 & 477 & 72.22 & 19.62 & 0.68 &  & 10046 & 22 & 9.31 & 4.65 & 0.14 \\ 
		& & 0.8 & 7053 & 12 & 4.59 & 1.53 & 0.03 & 21115 & 91 & 18.50 & 9.06 & 0.28 &  & 8705 & 9 & 7.25 & 3.79 & 0.04 \\ 
		& & 1.0 & 6175 & 9 & 4.35 & 1.40 & 0.03 & 17908 & 88 & 15.90 & 7.73 & 0.18 &  & 7120 & 8 & 6.03 & 3.02 & 0.04 \\ 
		
		1000 & 1000 & 0.4 & 20584 & 293 & 34.05 & 4.62 & 2.43 & 95716 & 1585 & 1117.26 & 30.92 & 5.66 & 0.34$^{1}$ & 19917 & 16 & 20.89 & 8.55 & 0.15 \\ 
		& & 0.6 & 14335 & 4 & 11.61 & 3.38 & 0.04 & 58192 & 125 & 58.27 & 27.70 & 0.35 &  & 16862 & 4 & 15.18 & 6.88 & 0.03 \\ 
		& & 0.8 & 12586 & 5 & 9.41 & 2.75 & 0.07 & 50843 & 56 & 43.25 & 27.31 & 0.22 &  & 13862 & 4 & 12.23 & 5.60 & 0.04 \\ 
		& & 1.0 & 16818 & 14 & 12.38 & 4.08 & 0.06 & 49781 & 86 & 47.25 & 24.21 & 0.16 &  & 18469 & 19 & 16.52 & 8.38 & 0.06 \\ 
		
		2000 & 500 & 0.4 & 12750 & 144 & 20.16 & 5.74 & 0.42 & 100063 & 3274 & 1332.21 & 48.62 & 2.03 & 0.31$^{3}$ & 14489 & 113 & 27.93 & 12.54 & 0.35 \\ 
		& & 0.6 & 9052 & 32 & 11.25 & 4.03 & 0.09 & 51471 & 1369 & 386.99 & 36.08 & 0.61 &  & 12323 & 48 & 18.94 & 10.32 & 0.08 \\ 
		& & 0.8 & 7389 & 9 & 10.49 & 3.77 & 0.04 & 27415 & 121 & 48.36 & 27.61 & 0.20 &  & 9011 & 11 & 15.68 & 8.31 & 0.05 \\ 
		& & 1.0 & 6460 & 4 & 8.48 & 2.91 & 0.02 & 22890 & 58 & 31.28 & 21.76 & 0.09 &  & 7632 & 3 & 12.70 & 6.98 & 0.02 \\ 
		
		2000 & 1000 & 0.4 & 22144 & 40 & 38.05 & 10.30 & 0.26 & 241443 & 2900 & 2397.31 & 108.45 & 2.33 & 0.46$^{4}$ & 26893 & 30 & 53.75 & 24.67 & 0.25 \\ 
		& & 0.6 & 26150 & 274 & 50.68 & 12.54 & 2.23 & 107615 & 1478 & 669.43 & 80.96 & 2.61 &  & 24878 & 267 & 54.10 & 19.89 & 2.23 \\ 
		& & 0.8 & 14955 & 16 & 20.56 & 7.00 & 0.05 & 47413 & 104 & 77.54 & 43.16 & 0.25 &  & 18906 & 27 & 30.43 & 16.04 & 0.06 \\ 
		& & 1.0 & 14627 & 12 & 22.02 & 6.64 & 0.05 & 41359 & 112 & 79.54 & 39.66 & 0.15 &  & 18123 & 11 & 35.51 & 17.83 & 0.05 \\ 
		
		3000 & 500 & 0.4 & 10056 & 64 & 20.88 & 6.80 & 0.18 & 108720 & 3062 & 1241.24 & 73.30 & 1.40 & 0.20$^{2}$ & 11187 & 38 & 31.50 & 16.60 & 0.14 \\ 
		& & 0.6 & 7095 & 12 & 15.68 & 4.78 & 0.06 & 25618 & 272 & 65.28 & 33.35 & 0.30 &  & 8419 & 7 & 22.46 & 10.86 & 0.05 \\ 
		& & 0.8 & 7232 & 15 & 16.10 & 4.93 & 0.05 & 25077 & 265 & 73.04 & 33.58 & 0.20 &  & 9015 & 24 & 25.79 & 12.17 & 0.05 \\ 
		& & 1.0 & 7767 & 34 & 14.03 & 5.07 & 0.04 & 21674 & 235 & 57.45 & 24.58 & 0.10 &  & 8666 & 34 & 20.39 & 11.08 & 0.04 \\ 
		
		3000 & 1000 & 0.4 & 16679 & 42 & 47.41 & 12.04 & 0.19 & 127064 & 1537 & 1174.58 & 100.64 & 1.21 & 0.26$^{2}$ & 23120 & 57 & 74.88 & 31.87 & 0.15 \\ 
		& & 0.6 & 14844 & 16 & 38.62 & 10.29 & 0.05 & 58724 & 506 & 370.05 & 80.85 & 0.29 &  & 17957 & 14 & 56.23 & 25.58 & 0.06 \\ 
		& & 0.8 & 14328 & 11 & 38.63 & 10.01 & 0.03 & 50275 & 303 & 156.46 & 63.58 & 0.21 &  & 20097 & 11 & 62.65 & 29.15 & 0.05 \\ 
		& & 1.0 & 14482 & 21 & 33.32 & 11.14 & 0.06 & 50468 & 251 & 167.38 & 69.02 & 0.16 &  & 18810 & 20 & 51.58 & 27.46 & 0.05 \\ 
		
		\multicolumn{3}{l}{All solved}  & 12181 & 41 & 19.87 & 5.61 & 0.28 & 49743 & 774 & 251.15 & 39.04 & 1.04 &  & 14375 & 29 & 27.77 & 13.14 & 0.17 \\ 
		
		\bottomrule
	\end{tabular}
\end{sidewaystable}

We first compare the single-phase lifted inequalities \eqref{eq:newineq} with the two-phase lifted inequalities \eqref{eq:twophase1}--\eqref{eq:twophase2} (which are the inequalities developed by \citet{Shi2022}).
From \cref{tab:EUM}, we observe that the setting using the single-phase lifted inequalities \eqref{eq:newineq}
generates  much  fewer cuts than that using the two-phase lifted inequalities \eqref{eq:twophase1}--\eqref{eq:twophase2}
($12181$ vs $49743$).
However, compared to that using the two-phase lifted inequalities \eqref{eq:twophase1}--\eqref{eq:twophase2},
the setting using the single-phase lifted inequalities \eqref{eq:newineq} yields a much tighter \LP relaxation gap at the root node ($0.28\%$ vs $1.04\%$).
In particular, for instances with $n = 1000$, $m = 500$, and $\lambda = 0.4$,
the setting using the single-phase lifted inequalities \eqref{eq:newineq} can yield a gap of $0.23\%$ while that using the two-phase lifted inequalities \eqref{eq:twophase1}--\eqref{eq:twophase2} yields a gap of $12.44\%$.
These results indicate that the (newly proposed) single-phase lifted inequalities \eqref{eq:newineq} are more effective than the two-phase lifted inequalities \eqref{eq:twophase1}--\eqref{eq:twophase2}
in terms of strengthening the \LP relaxation of the expected utility maximization problem \eqref{prob:EUM}.
This advantage directly translates into superior computational efficiency of the setting using the single-phase lifted inequalities \eqref{eq:newineq}.
In particular, using the single-phase lifted inequalities \eqref{eq:newineq}, we can successfully solve all $240$ instances to optimality.
For instances that can be solved by all three settings,
this method requires an average CPU time of $19.87$ seconds and an average number of explored nodes of $41$.
In contrast, using the two-phase lifted inequalities \eqref{eq:twophase1}--\eqref{eq:twophase2},  $15$ out of $240$ instances cannot be solved to optimality within the time limit of $3600$ seconds, 
and the average CPU time  and number of explored nodes are 
$251.15$ seconds and $774$, respectively. 

Next, we evaluate the performance effect of the combination of the single- and two-phase inequalities.
As shown in \cref{tab:EUM}, the combination outperforms the setting using the single-phase lifted inequalities \eqref{eq:newineq}
in terms of achieving a smaller \LP relaxation gap at the root node  ($0.17\%$ vs $0.28\%$) 
and exploring fewer nodes ($29$ vs $41$).
However, this advantage cannot compensate for the 
additional overhead spent in separating  cuts 
and the overhead led by adding more cuts.
Overall, for the expected utility maximization problem \eqref{prob:EUM}, 
the combination is slightly outperformed by the setting using the single-phase lifted inequalities \eqref{eq:newineq}.

\subsection{Weapon target assignment problem}\label{subsec:result-WTA}
The second data set for our computational study consists of the weapon target assignment problem,
which is one of the most popular classes of nonlinear assignment problems \cite{Manne1958,Ahuja2007,Andersen2022,Lu20211,Bertsimas2025}.
Let $n$ denote the number of weapon types with $\mu_i \in \Z_{++}$ representing the available quantity for each $i \in [n]$,
and $m$ be the number of targets with $V_j > 0$ representing the importance value for each  $j \in [m]$.
The probability of using a weapon of type $i$ to destroy target $j$ is denoted as $p_{ij} \in (0,1)$.
The weapon target assignment problem attempts to determine the number of weapons of type $i$ to be assigned to target $j$, 
denoted by $x_{i j}$, 
to maximize the expected value of damage across targets
respect the resource constraints on 
 the numbers of available weapons.
It can be formulated as the following nonlinear integer programming problem:
\begin{equation}\label{prob:WTA}
	\max_{x} \left\{ 
	\sum_{j = 1}^{m} V_j \left( 1 - \prod_{i = 1}^{n} (1 - p_{i j})^{x_{i j}} \right)
	\,:\,
	\sum_{j = 1}^{m} x_{i j} \leq \mu_i, \ \forall \ i \in [n], \
	x \in \Z_+^{n \times m}
	\right\}.
\end{equation}

In our experiments, we use the same scheme in \cite{Andersen2022,Bertsimas2025} to construct the instances.
Specifically, each $p_{i j}$ is independently drawn from the uniform distribution on $(0,1)$
and each $V_j$ is independently drawn from an integer uniform distribution on $[1, 100]$.
For the available quantity $\{\mu_i\}_{i \in [n]}$,  we set $\mu_i = 1$ with probability $1 - \rho$ and $\mu_i = 2$ with probability $\rho$.
In our computational experiment, $\rho$ is taken from $\{0.3, 0.4, 0.5\}$.

Letting $f(z) := 1 - \exp(-z)$ and $a_{i j} : = - \ln (1 - p_{i j})$ for $i \in [n]$ and $j \in [m]$, 
then $f$ is concave on $\R$ and $a_{i j} > 0$.
Therefore, the nonlinear term in the objective function can be rewritten as:
\begin{align*}
	1 - \prod_{i = 1}^{n} (1 - p_{i j})^{x_{i j}}
	& =  1 - \exp \left( \sum_{i = 1}^{n} x_{i j} \ln (1 - p_{i j}) \right) 
	= f \left(\sum_{i = 1}^{n} a_{i j} x_{i j}\right).
\end{align*}
Introducing a continuous variable $w_j \in \R$ for each target $j \in [m]$, 
we can rewrite the problem equivalently as
\begin{equation}\label{prob:WTA-refor}
	\begin{small}
		\begin{aligned}
			\max_{w, \, x} \left\{ 
			\sum_{j = 1}^{m} V_j w_j
			\,:\,
			w_j \leq f\left( \sum_{i = 1}^{n} a_{i j} x_{i j} \right), \ \forall \ j \in [m], \
			\sum_{j = 1}^{m} x_{i j} \leq \mu_i, \ \forall \ i \in [n], \
			x \in \Z_+^{n \times m}
			\right\}.
		\end{aligned}
	\end{small}
\end{equation}
Problem \eqref{prob:WTA-refor} is a convex \MINLP problem
and can therefore be solved by a branch-and-cut algorithm based on the  outer approximation (OA) cuts \cite{Duran1986,Fletcher1994,Bonami2008}:
\begin{equation}\label{eq:OA}
	w_j \leq f \left( \sum_{i = 1}^{n} a_{i j} \bar{x}_{i j} \right) 
	+ \sum_{i = 1}^{n} f' \left( \sum_{i = 1}^{n} a_{i j} \bar{x}_{i j} \right) a_{i j} (x_{i j} - \bar{x}_{i j}), \ \forall \ \bar{x} \in \Z_+^{n \times m},~\forall~j \in [m],
\end{equation}
where $f'(z) = \exp(-z)$.
Specifically, letting $(w^*, x^*)$ be a solution of the current \LP relaxation encountered during the branch-and-cut process, 
if $x^*$ is integral, then we add the corresponding OA cuts \eqref{eq:OA} into the formulation 
if $w_j^* > f \left( \sum_{i = 1}^{n} a_{i j} x^*_{i j} \right)$;
otherwise, we first round the fractional vector $x^*$ to its nearest integral vector $\bar{x}$, and then add the OA cuts \eqref{eq:OA} defined by $\bar{x}$ into the formulation if it is violated by $(w^*, x^*)$.
In addition, since problem \eqref{prob:WTA-refor} contains substructures in a form of $\Fset$
(i.e., $\{(w_j, x_{\cdot j}) \in \R \times \Z^n \, : \, 	w_j \leq f\left( \sum_{i = 1}^{n} a_{i j} x_{i j} \right),~0 \leq x_{ij} \leq \mu_i , \ \forall \ i \in [n] \}$ for $j \in [m]$), 
we can apply the proposed single-phase lifted inequalities \eqref{eq:single-phase} and two-phase lifted inequalities \eqref{eq:lifted-ineq1} and \eqref{eq:lifted-ineq2} (or their strengthened versions \eqref{facet-ineq1} and \eqref{facet-ineq2}) to strengthen the \LP relaxation of problem \eqref{prob:WTA-refor}. 

To evaluate the strength of the proposed single- or/and two-phase inequalities,  we compare the performance of the following four settings, 
\begin{itemize}
	\item[$\bullet$] \settingO: solve the weapon target assignment problem \eqref{prob:WTA-refor} by the branch-and-cut algorithm based on OA cuts \eqref{eq:OA};
	\item[$\bullet$] \settingS: \settingO with the single-phase lifted inequalities \eqref{eq:single-phase};
	\item[$\bullet$] \settingT: \settingO with the two-phase lifted inequalities \eqref{eq:lifted-ineq1} and \eqref{eq:lifted-ineq2} (or their strengthened versions \eqref{facet-ineq1} and \eqref{facet-ineq2});
	\item[$\bullet$] \settingC:  \settingO with the combination of the single- and two-phase lifted inequalities.
\end{itemize}

\cref{tab:WTA} summarizes the computational results.
We first observe that the formulation using the OA cuts \eqref{eq:OA} yields an extremely poor \LP relaxation bound of the weapon target assignment problem \eqref{prob:WTA-refor};
the \LP relaxation gap at the root node of the \settingO ranges from $79.98\%$ to $96.58\%$.
Consequently, none of the instances can be solved by \settingO within the time limit of 3600 seconds.
In sharp contrast, 
the single- and two-phase lifted inequalities enable to return a much tighter \LP relaxation gap at the root node (ranging from $2.11\%$ to $14.37\%$ for \settingS, and from $0.34\%$ to $3.13\%$ for \settingT),
which indicates that our single- and two-phase lifted inequalities 
can indeed significantly strengthen the \LP relaxation of the weapon target assignment problem \eqref{prob:WTA-refor}.
This result directly translates to superior solution efficiency for settings using our proposed lifted inequalities.
In particular, \settingS and \settingT can solve $97$ and $104$ out of the $120$ instances, respectively, to optimality within the time limit of $3600$ seconds.
Moreover, for instances that cannot be solved  to optimality, 
\settingS and \settingT can still provide a very small optimality gap on the best solution obtained (ranging from $0.02\%$ to $1.02\%$ for \settingS, and from $0.02\%$ to $0.77\%$ for \settingT).

Next, we compare the performance of \settingS and \settingT.
Recall that, for the expected utility maximization problem \eqref{prob:EUM}, the single-phase lifted inequalities enable the branch-and-cut algorithm to achieve much better performance than the two-phase lifted inequalities.
This is, however, not the case for the weapon target assignment problem \eqref{prob:WTA-refor}.
Indeed, \settingT outperforms \settingS for this problem, as shown in \cref{tab:WTA}.
More specifically, \settingT yields a tighter \LP relaxation gap at the root node ($1.28\%$ vs $6.35\%$) while generating fewer cuts ($33457$ vs $58873$).
Consequently, \settingT can solve five more instances than \settingS within the time limit of $3600$ seconds;
for the instances that cannot be solved to optimality, 
\settingT always returns a smaller end gap than \settingS.
For instances that can be solved to optimality by at least one setting,
\settingT can return a smaller CPU time and number of explored nodes than those returned by \settingS.

Finally, we observe from \cref{tab:WTA} that the combination of the single- and two-phase lifted inequalities can even achieve a better \LP relaxation bound for the weapon target assignment problem \eqref{prob:WTA-refor},
as demonstrated in columns RGap $(\%)$ under  \settingS, \settingT, and \settingC.
Therefore, 
compared with \settingS and \settingT,
\settingC achieves a better performance in terms of smaller CPU time and number of explored nodes.

\begin{sidewaystable}[htbp]
	\centering
	\caption{Computational results of the four settings:  
		\settingO, {\settingS}, {\settingT}, and {\settingC}
		for the weapon target assignment problem \eqref{prob:WTA-refor}.}
	\label{tab:WTA}
	\fontsize{7.5pt}{9}\selectfont
	\renewcommand{\arraystretch}{1.5}
	\addtolength{\tabcolsep}{-4.5pt}

		\begin{tabular}{ccc
				cc
				cccccc
				cccccc
				cccccc}
			\toprule
			\multicolumn{1}{c}{$n$} & 
			\multicolumn{1}{c}{$m$} & 
			\multicolumn{1}{c}{$\rho$} & 
			\multicolumn{2}{c}{\settingO} &
			\multicolumn{6}{c}{\settingS} &
			\multicolumn{6}{c}{
				\begin{tabular}{@{}c@{}}
					\settingT \\
				\end{tabular}
			} & 
			\multicolumn{6}{c}{\settingC} \\
			\cmidrule(lr){4-5}\cmidrule(lr){6-11}\cmidrule(lr){12-17}\cmidrule(lr){18-23}
			& & &

			\multicolumn{1}{c}{Rgap (\%)} &
			\multicolumn{1}{c}{Egap (\%)} &

			\multicolumn{1}{c}{C} & 
			\multicolumn{1}{c}{N} & 
			\multicolumn{1}{c}{T} & 
			\multicolumn{1}{c}{ST} & 
			\multicolumn{1}{c}{Rgap (\%)} &
			\multicolumn{1}{c}{Egap (\%)} &

			\multicolumn{1}{c}{C} & 
			\multicolumn{1}{c}{N} & 
			\multicolumn{1}{c}{T} & 
			\multicolumn{1}{c}{ST} & 
			\multicolumn{1}{c}{Rgap (\%)} &
			\multicolumn{1}{c}{Egap (\%)} &

			\multicolumn{1}{c}{C} & 
			\multicolumn{1}{c}{N} & 
			\multicolumn{1}{c}{T} &
			\multicolumn{1}{c}{ST} & 
			\multicolumn{1}{c}{Rgap (\%)} &
			\multicolumn{1}{c}{Egap (\%)} \\
			\midrule 
			
			75 & 100 & 0.3 & 83.03 & 7.61$^{10}$ & 2937 & 597 & 2.81 & 0.08 & 2.82  & & 1512 & 78 & 1.28 & 0.13 & 0.58 &  & 1346 & 58 & 1.07 & 0.12 & 0.72 &  \\ 
			& & 0.4 & 80.40 & 12.11$^{10}$ & 6101 & 823 & 4.96 & 0.14 & 6.47  & & 2471 & 242 & 3.05 & 0.34 & 1.33 &  & 1888 & 159 & 2.24 & 0.27 & 1.12 &  \\ 
			& & 0.5 & 79.98 & 14.89$^{10}$ & 9541 & 1501 & 9.36 & 0.26 & 8.54  & & 3211 & 496 & 4.47 & 0.61 & 1.91 &  & 2822 & 353 & 3.54 & 0.51 & 1.53 &  \\ 
			
			150 & 200 & 0.3 & 87.92 & 10.56$^{10}$ & 10075 & 1584 & 29.76 & 0.49 & 3.29  & & 5894 & 965 & 21.45 & 1.88 & 1.11 &  & 3868 & 552 & 12.67 & 1.23 & 0.69 &  \\ 
			& & 0.4 & 89.55 & 19.21$^{10}$ & 39566 & 6180 & 94.14 & 1.77 & 7.78  & & 31237 & 4629 & 76.84 & 9.12 & 1.78 &  & 16025 & 2588 & 51.25 & 6.04 & 1.71 &  \\ 
			& & 0.5 & 90.58 & 26.25$^{10}$ & 102025 & 14842 & 232.58 & 4.93 & 13.48  & & 59660 & 7819 & 130.68 & 19.93 & 1.60 &  & 51922 & 7431 & 132.79 & 22.87 & 1.71 &  \\ 
			
			225 & 300 & 0.3 & 92.09 & 12.65$^{10}$ & 12956 & 1342 & 71.60 & 0.88 & 2.52  & & 6093 & 352 & 29.95 & 1.85 & 0.44 &  & 4819 & 467 & 32.37 & 2.00 & 0.41 &  \\ 
			& & 0.4 & 92.67 & 24.94$^{10}$ & 279861 & 46023 & 1171.45 & 22.25 & 7.91  & 0.02$^{1}$& 160278 & 27910 & 1122.76 & 109.10 & 2.14 & 0.02$^{1}$ & 145529 & 26252 & 938.26 & 106.31 & 1.93 & 0.06$^{1}$ \\ 
			& & 0.5 & 93.86 & 29.73$^{10}$ & 853841 & 99567 & 3109.72 & 59.94 & 13.79  & 0.22$^{8}$& 458346 & 58529 & 2172.74 & 248.11 & 2.33 & 0.10$^{3}$ & 533846 & 73202 & 2472.48 & 328.40 & 2.51 & 0.10$^{4}$ \\ 
			
			300 & 400 & 0.3 & 96.14 & 9.37$^{10}$ & 13175 & 1127 & 171.13 & 1.39 & 2.11  & & 6099 & 344 & 48.59 & 2.35 & 0.34 &  & 6329 & 360 & 58.66 & 3.13 & 0.34 &  \\ 
			& & 0.4 & 95.65 & 27.34$^{10}$ & 276971 & 43432 & 2957.10 & 35.50 & 8.42  & 0.12$^{4}$& 148096 & 24759 & 1973.39 & 158.40 & 1.82 & 0.03$^{2}$ & 131849 & 24041 & 1602.42 & 145.18 & 1.85 &  \\ 
			& & 0.5 & 96.58 & 41.23$^{10}$ & 436329 & 28502 & 3600.00 & 32.47 & 14.37  & 1.02$^{10}$& 398056 & 28708 & 3600.00 & 256.11 & 3.13 & 0.77$^{10}$ & 419419 & 34371 & 3600.00 & 303.50 & 2.88 & 0.65$^{10}$ \\ 
			
			\multicolumn{5}{l}{All solved}  & 58873 & 9542 & 295.93 & 4.40 & 6.35 &  & 33457 & 5344 & 210.25 & 19.11 & 1.28 &  & 29104 & 4837 & 182.42 & 18.99 & 1.17 & \\ 
	
			\bottomrule
		\end{tabular}
\end{sidewaystable}
\section{Conclusion}\label{sec:conslusion}
In this paper, we have conducted the first comprehensive polyhedral investigation on the mixed-integer nonlinear set with box constraints: $\Fset= \left\{ (w,x) \in \R \times \Z^n \,:\, w \leq f(a^\top x), \right.$ $\left. 0 \leq x_i \leq \mu_i, ~ \forall ~ i \in [n] \right\}$ that
arises as an important substructure in many \MINLP problems.
We proposed a class of seed inequalities for a two-dimensional restriction of $\Fset$ (obtained by fixing all but one of the integer variables to their bounds), and developed two lifting procedures to derive strong valid inequalities for $\conv(\Fset)$.
In the first lifting procedure, all fixed variables are simultaneously lifted, based on a subadditive approximation for the exact lifting function of the seed inequalities. 
In the second lifting procedure, the variables fixed at their lower bounds (respectively, at their upper bounds) are lifted first using the exact  lifting function, and then the variables fixed at their upper bounds (respectively, at their lower bounds) are lifted using some subadditive approximations.
The resultant inequalities, the single- and two-phase lifted inequalities, can be efficiently computed and are able to define facets for $\conv(\Fset)$ under mild conditions.
These two key features make them particularly suitable to be embedded in a branch-and-cut framework to accelerate the computational performance of solving related \MINLP problems.
Moreover, for the special case where $\Fset$ is the submodular maximization set \cite{Nemhauser1988,Ahmed2011,Shi2022}, 
the mixed-integer knapsack set \cite{Nemhauser1990,Marchand1999,Atamturk2003}, 
the mixed-integer polyhedral conic set \cite{Atamturk2010b}, 
or the two-row mixed-integer set \cite{Wolsey1977,Bodur2017},
we demonstrate that the proposed  inequalities either unify existing inequalities in the literature or yield new facet-defining inequalities for the convex hull of the related set.
Finally, computational results on expected utility maximization  and weapon target assignment problems demonstrate the 
superiority of the proposed  single- and two-phase lifted inequalities in terms of strengthening the continuous relaxations and  improving the overall performance of the branch-and-cut algorithm.
\appendix
\section{Proof of \cref{liftzeta-new}}\label{proof_4.1}
Letting $\Delta(\delta_1, \delta_2) = \zeta(\delta_1) + \zeta(\delta_2) - \zeta(\delta_1 + \delta_2)$ for $\delta_1, \delta_2 \in \R$,
then it suffices to prove that $\Delta(\delta_1, \delta_2) \geq 0$ for (i) $\delta_1, \delta_2 \in \R_+$ and (ii) $\delta_1, \delta_2 \in \R_-$.
If $\delta_1 = 0$ or $\delta_2 = 0$, then $\Delta(\delta_1, \delta_2) = 0$ trivially holds.
Therefore, we only consider the case $\delta_1 \neq 0$ and $\delta_2 \neq 0$ in the following.

We first show that $\Delta(\delta_1, \delta_2) \geq  0$ for $\delta_1, \delta_2 \in \R_{++}$. 
By the definition of $\zeta$ in \eqref{calc:zeta}, it follows that for $\delta \in \R_{++}$,
\begin{equation*}\small
	\zeta(\delta) =  \left\{
	\begin{aligned}
		&g \left(\delta+ (k - \ell - 1) a_s\right) + (\ell + 1) \rho_s(k)-g(ka_s), &&\text{if}~\ell a_s \le \delta < (\ell + 1) a_s, \\
		& && ~ \ell = 0, \ldots, k - 1, \\
		&g \left(\delta\right) + k\rho_s(k)-g(ka_s), &&\text{if}~\delta \ge ka_s.
	\end{aligned}
	\right.
\end{equation*}
Without loss of generality, we assume that $\delta_1\le \delta_2$.
Suppose that $\ell_i a_s \leq \delta_i < (\ell_i + 1)a_s$ holds for some  $\ell_i \in \Z_+$, $i = 1,2$.
We consider the following three cases.
\begin{enumerate}
	\item[(i)] {$0 < \delta_1 \le \delta_2 < ka_s$}. 
	From $ \zeta(\delta)\leq Z(\delta)$ for $\delta\in \mathbb{R}$, $ \zeta(\delta)=Z(\delta)$ for $0 \leq \delta\leq k a_s$, and the subadditivity of $Z$ established in \cref{lem:Z-subadditive-on-R}, we have $\Delta(\delta_1, \delta_2)= Z(\delta_1)+ Z(\delta_2)- \zeta(\delta_1+\delta_2) \geq Z(\delta_1)+ Z(\delta_2)- Z(\delta_1+\delta_2) \geq  0$.
	\item [(ii)] {$0 < \delta_1 < ka_s \le \delta_2 < \delta_1 + \delta_2$}. 
	We have 
	\begin{equation*}
		\begin{aligned}
			\Delta(\delta_1, \delta_2) 
			& = [(\ell_1 + 1)(g(ka_s) - g((k-1)a_s))]
			- \left[
			g(\delta_1 + \delta_2)
			- g(\delta_2) 
			\right] \\
			& \quad 
			- \left[
			g(ka_s) 
			- g(\delta_1 - (\ell_1+1)a_s + ka_s)
			\right],
		\end{aligned}
	\end{equation*}
	and 
	\begin{equation*}
		\begin{aligned}
			& (\ell_1 + 1)a_s - \delta_1 - (-\delta_1 + (\ell_1 + 1)a_s)=0, \\
			& ka_s  \le
			\min\left\{
			\delta_1 + \delta_2,
			ka_s
			\right\}, \\
			& (k-1)a_s  \le
			\min\left\{
			\delta_2,
			\delta_1 - (\ell_1 + 1)a_s + ka_s
			\right\}.
		\end{aligned}
	\end{equation*}
	Applying \cref{cor:sum-pariwise-comparison},
	we obtain $\Delta(\delta_1, \delta_2) \ge 0$.
	\item [(iii)] {$ka_s \le \delta_1 \le \delta_2 < \delta_1 + \delta_2$}. We have 
	\begin{equation*}
		\begin{aligned}
			\Delta(\delta_1, \delta_2) 
			& = k[g(ka_s) - g((k-1)a_s)]
			+ [g(\delta_1)
			- g(ka_s)] \\
			& \quad 
			- \left[
			g(\delta_1 + \delta_2) 
			-g(\delta_2)
			\right],
		\end{aligned}
	\end{equation*}
	and 
	\begin{equation*}
		\begin{aligned}
			& ka_s + (\delta_1 - ka_s)- \delta_1 = 0, \\
			& \max\left\{
			ka_s,
			\delta_1
			\right\}
			\le
			\delta_1 + \delta_2, \\
			& \max\left\{
			(k-1)a_s,
			ka_s
			\right\}
			\le
			\delta_2.
		\end{aligned}
	\end{equation*}
	Again applying \cref{cor:sum-pariwise-comparison}, 
	we obtain $\Delta(\delta_1, \delta_2) \geq 0$.
\end{enumerate}

Next, we prove $\Delta(\delta_1, \delta_2) \geq  0$ for $\delta_1, \delta_2 \in \R_{--}$. 
By the definition of $\zeta$ in \eqref{calc:zeta}, it follows that for $\delta \in \R_{--}$,
\begin{equation*}\small
	\zeta(\delta) =  \left\{
	\begin{aligned}
		&g(\delta+\mu_sa_s)+(k-\mu_s)\rho_s(k)-g(ka_s), &&\text{if}~\delta < (k-\mu_s  - 1)a_s,\\
		&g \left(\delta+ (k - \ell - 1) a_s\right) + (\ell + 1) \rho_s(k)-g(ka_s), &&\text{if}~\ell a_s \le \delta < (\ell + 1) a_s, \\
		& && ~ \ell = k - \mu_s-1,\ldots, - 1. \\
	\end{aligned}
	\right.
\end{equation*}
Without loss of generality, we assume that $\delta_1\geq \delta_2$.
Suppose that $\ell_i a_s \leq \delta_i < (\ell_i + 1)a_s$ holds for some $\ell_i \in \Z_{--}, \ i = 1,2$.
We consider the following three cases.
\begin{enumerate}
	\item[(i)] {$ 0 > \delta_1 \ge \delta_2 \geq  (k - \mu_s - 1) a_s$}. 
	From $ \zeta(\delta)\leq Z(\delta)$ for $\delta\in \mathbb{R}$, $ \zeta(\delta)=Z(\delta)$ for $(k-\mu_s -1)a_s \leq \delta<0$, and the subadditivity of $Z$ established in   \cref{lem:Z-subadditive-on-R}, we have $\Delta(\delta_1, \delta_2)= Z(\delta_1)+ Z(\delta_2)- \zeta(\delta_1+\delta_2) \geq Z(\delta_1)+ Z(\delta_2)- Z(\delta_1+\delta_2) \geq  0$.
	\item [(ii)] {$ 0 > \delta_1  \ge (k - \mu_s - 1) a_s >
		\delta_2 > \delta_1 + \delta_2$}. 
	We have
	\begin{equation*}
		\begin{aligned}
			\Delta(\delta_1, \delta_2) 
			& = 
			[g(\delta_2+\mu_s a_s)
			- g(\delta_1 + \delta_2+\mu_s a_s) ]
			+ [ (\ell_1 + 1 ) (g(ka_s) - g((k-1)a_s))] \\
			& \quad 
			- \left[
			g(ka_s) 
			- g(\delta_1 - (\ell_1+1)a_s + ka_s)
			\right],
		\end{aligned}
	\end{equation*}
	and 
	\begin{equation*}
		\begin{aligned}
			& -\delta_1
			+ (\ell_1 + 1 ) a_s 
			- (- \delta_1 + (\ell_1 + 1)a_s)=0, \\
			&\delta_2 + \mu_s a_s \le ka_s, \\
			&
			\delta_1 + \delta_2 +\mu_s a_s \le \min\{(k-1)a_s,  \delta_1 - (\ell_1 + 1) a_s + ka_s\}.
		\end{aligned}
	\end{equation*}
	Applying \cref{cor:sum-pariwise-comparison},
	we obtain $\Delta(\delta_1, \delta_2) \ge 0$.
	\item [(iii)] {$ 0 > (k - \mu_s - 1) a_s >  \delta_1  \ge
		\delta_2 > \delta_1 + \delta_2$}.  We have 
	\begin{equation*}
		\begin{aligned}
			\Delta(\delta_1, \delta_2) 
			& = [g(\delta_2+\mu_s a_s)
			- g(\delta_1 + \delta_2+\mu_s a_s)] \\
			& \quad 
			- [ (\mu_s-k) (g(ka_s) - g((k-1)a_s))]
			- \left[g(ka_s)- g(\delta_1+\mu_s a_s)\right],
		\end{aligned}
	\end{equation*}
	and 
	\begin{equation*}
		\begin{aligned}
			&- \delta_1 - (\mu_s-k) a_s - ((k - \mu_s) a_s - \delta_1) = 0, \\
			& \delta_2 + \mu_s a_s \le ka_s, \\
			& \delta_1 + \delta_2 + \mu_s a_s
			\le
			\min\left\{ (k-1)a_s, \delta_1 + \mu_s a_s\right\}.
		\end{aligned}
	\end{equation*}
	Again applying \cref{cor:sum-pariwise-comparison}, 
	we obtain $\Delta(\delta_1, \delta_2)  \geq 0$.\qedhere
\end{enumerate}

\section{Proof of \cref{prop:sub-of-etaLU}}\label{proof_4.12}

	We only prove the subadditivity of $\eta^L$ in \eqref{eta-explicit} on $\R_-$ as
	 (i) $\eta^L$ in \eqref{etaL} is a special case of \eqref{eta-explicit} with parameters $\underbrace{a_1, \ldots , a_1}_{\mu_1}, \underbrace{a_2, \ldots , a_2}_{\mu_2},\ldots, \underbrace{a_r, \ldots , a_r}_{\mu_r}$,
	and (ii) $\eta^U$ in \eqref{etaU} is a special case of $\eta^L$ in \eqref{etaL} with $\mu_s = +\infty$.
	The proof is divided into two steps.
	First, we transform $\eta^L$ into an iterative form in \cref{bareta-rewrite1}.
	Then,
	we show that the new iterative form of $\eta^L$ is a special case of the subadditive function $\bar{\omega}$ defined in \eqref{baromega}. 
	
	Let $\bar{r} := \ceil{(\mu_s - k + 1) / k}$, $r' := r + \bar{r}$, 
	and $\kappa := \mu_s - k + 1 - (\bar{r} - 1) k$.
	Then $\bar{r} \in [\mu_s]$ and $\kappa \in [k]$ as $k \in [\mu_s]$.
	For $i = r + 1, \ldots, r'$, let $a_i : = k a_s$ and $A_i : = \sum_{j = 1}^{i} a_j$.
	It follows that $A_i = A_r + (i - r) k a_s$ for $i = r+1, \ldots, r'$.
	Using these notations, we can transform $\eta^L (\delta)$ into an iterative form.
	\begin{claim}\label{bareta-rewrite1}
		The function $\eta^L(\delta)$ in \eqref{eta-explicit} admits the following iterative form:
		\begin{equation}
			\label{etabar-recursive}
			\begin{footnotesize}
				\eta^L(\delta) = \left\{
				\begin{aligned}
					& 0, 
					\hspace{3.55cm} ~\text{if}~ \delta = 0, \\
					& g(\delta + A_{i + 1}) + \eta^L(-A_{i}) - g(a_{i + 1}), \\
					& \hspace{3.85cm} ~\text{if}~ k a_s - A_{i + 1} \leq \delta < -A_i, \\
					& \hspace{3.85cm} ~~ i =  0, \ldots, r' - 1, \\
					& g(\delta + A_{i + 1} + \ell a_s) + \eta^L(-A_{i}) - \ell \rho_s(k) - g(a_{i + 1}), \\
					& \hspace{3.85cm} ~\text{if}~ k a_s - (\ell + 1) a_s - A_{i + 1} \leq \delta < k a_s - \ell a_s - A_{i + 1}, \\
					& \hspace{3.85cm} ~~ i =  0,\ldots,r' - 2 \ \text{and} \ \ell = 0,\ldots, k-1, \\
					& \hspace{3.85cm} ~~ \text{or} \ i =  r' - 1 \ \text{and} \ \ell = 0,\ldots, \kappa - 1, \\
					& g(\delta + A_{r'} + (\kappa - 1) a_s) + \eta^L(-A_{r' - 1}) - (\kappa - 1) \rho_s(k) - g(a_{r'}), \\
					& \hspace{3.9cm} ~\text{if}~ \delta < k a_s - \kappa a_s - A_{r'},
				\end{aligned}
				\right.
			\end{footnotesize}
		\end{equation}
	\end{claim}
	\begin{proof}
		From the definition in \eqref{eta-explicit}, $\eta^L (\delta)$ is a continuous function on $\R_{-}$ with $\eta^L(0)=0$.
		We first show that 
		\begin{equation}\label{eta-Ai}
			\eta^L(-A_i) = -\sum_{j = 1}^{i}\zeta(a_j), \ \forall \ i = 1, \ldots, r' - 1.
		\end{equation}
		Let $i_0\in [r'-1]$.
		If $i_0 \leq r-1$, then the first case in \eqref{eta-explicit} applies and it follows
		\begin{align*}
			\eta^L(- A_{i_0}) 
			& = g(- A_{i_0} + A_{i_0 + 1}) - \sum_{j = 1}^{i_0  + 1} \zeta (a_j) + k \rho_s(k) - g(k a_s) \\
			& = - \sum_{j = 1}^{i_0} \zeta(a_j) - \zeta(a_{i_0 + 1}) + g(a_{i_0+1}) + k \rho_s(k) - g(k a_s)
			\stackrel{(a)}{=} - \sum_{j = 1}^{i_0} \zeta(a_j),
		\end{align*}
		where (a) follows from $\zeta(a_{i_0 + 1}) = g(a_{i_0+1}) + k \rho_s(k) - g(k a_s)$ (by ${i_0 + 1} \in [r] = S_0^+$ and \eqref{calc:zeta}).
		Otherwise, $i_0 \in \{r, r + 1, \ldots, r' - 1\}$. 
		Then $- A_{i_0} = - (i_0 - r) k a_s - A_r = (k - \bar{\ell} )a_s - A_r$,
		where $\bar{\ell} := (1 +i_0 - r) k$.
		By $r \leq i_0 \leq r'-1$ and $\bar{r}=r' - r$, 
		it follows $k \leq \bar{\ell} \leq\bar{r} k = \ceil{(\mu_s - k + 1) / k} k \leq \mu_s + 1$.
		If $\bar{\ell} \leq \mu_s$, then 
		the second case  in \eqref{eta-explicit} (with $i=r-1$ and ${\ell} =\bar{\ell}$)  applies and it follows
		\begin{equation*}
			\begin{aligned}
				\eta^L(-A_{i_0})
				& = g((k - \bar{\ell}) a_s - A_r + A_r + \bar{\ell} a_s) - \sum_{j = 1}^{r} \zeta(a_j) + (k - \bar{\ell}) \rho_s(k) - g(k a_s) \\
				& \stackrel{(a)}{=} - \sum_{j = 1}^{r} \zeta(a_j) - (i_0 - r) \zeta(ka_s)  
				\stackrel{(b)}{=} - \sum_{j = 1}^{i_0} \zeta(a_j),
			\end{aligned}
		\end{equation*}
		where the (a) follows from $k - \bar{\ell} = -(i_0-r)k$ and $\zeta(ka_s) = g(ka_s) + k \rho_s(k) - g(k a_s)=k \rho_s(k)$,
		and (b) follows from $a_j = ka_s$ for $j \in \{r+1, \ldots, i_0\}$.
		Otherwise, $\bar{\ell} = \mu_s + 1$ must hold, then $-A_{i_0} = (k - \bar{\ell})a_s - A_r = (k - \mu_s - 1) a_s - A_r$.
		Thus the last case in \eqref{eta-explicit} applies and it follows
		\begin{small}
			\begin{align*}
				\eta^L(-A_{i_0}) 
				& = g((k - \mu_s - 1) a_s - A_r + A_r + \mu_s a_s) - \sum_{j = 1}^{r}\zeta(a_j) + (k - \mu_s) \rho_s(k) - g(ka_s) \\
				& = - \sum_{j = 1}^{r}\zeta(a_j) + (k - \mu_s - 1) \rho_s(k) 
				\stackrel{(a)}{=} - \sum_{j = 1}^{r}\zeta(a_j) - (i_0 - r) \zeta(k a_s)
				\stackrel{(b)}{=} - \sum_{j = 1}^{i_0}\zeta(a_j),
			\end{align*}
		\end{small}%
		where (a) follows from $\bar{\ell} = \mu_s + 1 = (1 + i_0 - r)k$
		and $\zeta(ka_s) = g(ka_s) + k \rho_s(k) - g(k a_s)=k \rho_s(k)$,
		and (b) follows from $a_j = ka_s$ for $j \in \{r+1, \ldots, i_0\}$.
		
		Let $\tilde{\eta}^L(\delta)$ denote the right-hand-side of \eqref{etabar-recursive}.
		If $-A_r < \delta \leq 0$, 
		then $\tilde{\eta}^L(\delta)={\eta}^L(\delta)$ follows from \eqref{eta-Ai}
		and $\zeta(a_{i + 1}) = g(a_{i + 1}) + k \rho_s(k) - g(k a_s)$.
		Now consider the case $\delta \leq - A_r$. 
		From the definition of $\kappa$,
		we have $\mu_s + 1= \bar{r} k + \kappa$, 
		and thus 
		\begin{equation}\label{eq:Ar-r2}
			(k - \mu_s - 1) a_s - A_r
			= - \bar{r} ka_s +  ka_s - \kappa a_s - A_r=  k a_s - \kappa a_s - A_{r'}.
		\end{equation}
		 We further consider the following two cases.
		\begin{itemize}
			\item [(i)] $(k - \mu_s - 1) a_s - A_r < \delta \leq - A_r$, i.e., the second case in \eqref{eta-explicit} with $i = r - 1$.
			Letting $\ell \in \{k, \ldots, \mu_s\}$ be such that  $\delta \in ((k - \ell - 1) a_s - A_r, (k - \ell) a_s - A_r]$, then 
			\begin{equation}\small\label{tmp}
				\eta^L(\delta)  =g(\delta + A_r + \ell a_s) - \sum_{j = 1}^{r} \zeta(a_j) + (k - \ell) \rho_s(k) - g(k a_s) .
			\end{equation}		
			Let $i' : = \floor{\ell / k} + r - 1$ and $\ell' : = \ell - \floor{\ell / k} k$.
			Then it follows that $i' \in \{r, \ldots, r' - 1\}$ (as $\floor{\ell / k}  \leq \ell/k  \leq \mu_s/k <  \lceil (\mu_s +1)/ k\rceil= \bar{r}+1= r'-r+1$), $\ell' \in \{0, \ldots, k - 1\}$,
			and
			\begin{footnotesize}
				\begin{align*}
					((k - \ell - 1) a_s - A_r, (k - \ell) a_s - A_r]
					& =(-(i'-r)ka_s - (\ell' + 1) a_s-A_r,  -(i'-r)ka_s -\ell' a_s-A_r] \\
					& = (k a_s - (\ell' + 1) a_s - A_{i' + 1}, k a_s - \ell'a_s - A_{i' + 1}].
				\end{align*}
			\end{footnotesize}%
			Together with \eqref{eq:Ar-r2} and $(k - \mu_s - 1) a_s - A_r < \delta \leq -A_r$, $\delta$ must belong to the third case in \eqref{etabar-recursive}.
			Therefore, 
			\begin{align*}\small
				\tilde{\eta}^L(\delta)& = g(\delta + A_{i' + 1} + \ell' a_s) + \eta^L(-A_{i'}) - \ell' \rho_s(k) - g(a_{i'+1})\\
				& = g(\delta + A_{i' + 1} + \ell' a_s) - \sum_{j = 1}^{r} \zeta(a_j) 
				- (i' - r)\zeta(k a_s) - \ell' \rho_s(k) - g(a_{i' + 1}) \\
				& \overset{(a)}{=}g(\delta + A_r + \ell a_s) - \sum_{j = 1}^{r} \zeta(a_j) + (k - \ell) \rho_s(k) - g(k a_s)\overset{(b)}=	\eta^L(\delta)
			\end{align*}
			where (a) follows from $\ell = (i' + 1- r) k + \ell'$,
			$ A_{i' + 1} + \ell' a_s =A_r + \ell a_s$,
			$ - (i' - r)\zeta(k a_s) - \ell' \rho_s(k)=(-(i' - r) k - \ell') \rho_s(k) = (k - \ell) \rho_s(k)$, 
			and $a_{i' + 1} = ka_s$, 
			and (b) follows from \eqref{tmp}.
			\item [(ii)] $\delta \leq (k - \mu_s - 1) a_s - A_r$,  i.e.,  the third case in \eqref{eta-explicit}.
			From \eqref{eq:Ar-r2}, $\delta \leq k a_s - \kappa a_s - A_{r'}$ belongs the last case in \eqref{etabar-recursive} or the third case in \eqref{etabar-recursive} with $\delta = k a_s - \kappa a_s - A_{r'}$.
			As $\tilde{\eta}^{L}(\delta)$ is a continuous function, it follows from \eqref{eta-Ai} that
			\begin{align*}
				\tilde{\eta}^L(\delta)
				& = g(\delta + A_{r'} + (\kappa - 1) a_s) + \eta^L(-A_{r' - 1}) - (\kappa - 1) \rho_s(k) - g(a_{r'})\\
				& = g(\delta + A_{r'} + (\kappa - 1) a_s)  - \sum_{j = 1}^{r}\zeta(a_j) - (\bar{r} - 1) \zeta(k a_s) 
				- (\kappa - 1) \rho_s(k) - g(k a_s) \\
				& \overset{(a)}= g(\delta + A_r + \mu_s a_s) - \sum_{j = 1}^{r}\zeta(a_j) + (k - \mu_s) \rho_s(k) - g(ka_s)=\eta^L(\delta),
			\end{align*}
			where (a) follows from 
			$A_{r'} + (\kappa - 1) a_s= (A_r + \bar{r} k a_s) + (\kappa - 1) a_s
			=  A_r + \mu_s a_s $ 
			and $- (\bar{r} - 1) \zeta(k a_s) - (\kappa - 1) \rho_s(k) = -((\bar{r} - 1) k + \kappa - 1) \rho_s(k) 
			= (k - \mu_s) \rho_s(k) $.
			\qedhere
		\end{itemize}
	\end{proof}
	\begin{claim}\label{bareta-rewrite2}
		The function $\eta^L(\delta)$ in \eqref{etabar-recursive} is special case of $\bar{\omega}(\delta)$ in \eqref{baromega}. 
	\end{claim}
	\begin{proof}
		We choose parameters in \eqref{def:b} as
		$\epsilon = a_s$, $\tau = k$, 
		$b_i = a_i - k a_s$, and $v_{i - 1} = - a_i$ for $i \in [r']$, 
		and parameters that define $\bar{\omega}$ as $m=r'$ and $\gamma= \kappa -1$.
		Then it follows
		$B_i = A_{i - 1} +b_i = A_i - k a_s$
		and $\psi_i = g(-b_{i+1} - v_i - \epsilon) - g(-b_{i + 1} - v_i) 
		= g((k -1) a_s) - g(k a_s) = -\rho_s(k)$
		for $i \in [r']$.  
		Under these parameters, $\bar{\omega}(\delta) = \eta^L(\delta)$ must hold.
	\end{proof}

\section{Proof of \cref{thm:chi-sub}}\label{proof_4.21}
Observe that for any $\delta_1, \delta_2 \in \R_+$, it follows
\begin{align*}
	& \chi(\delta_1) + \chi(\delta_2) - \chi(\delta_1 + \delta_2)
	= \omega(- \delta_1) + \omega(-\delta_2) - \omega(-\delta_1 - \delta_2), \\
	& \bar{\chi}(\delta_1) + \bar{\chi}(\delta_2) - \bar{\chi}(\delta_1 + \delta_2)
	= \bar{\omega}(- \delta_1) + \bar{\omega}(-\delta_2) - \bar{\omega}(-\delta_1 - \delta_2).
\end{align*}
As a result, to prove \cref{thm:chi-sub}, it suffices to show that
\begin{align}
	\chi(\delta_1) + \chi(\delta_2) - \chi(\delta_1 + \delta_2) \geq 0, \label{goal1} \\
	\bar{\chi}(\delta_1) + \bar{\chi}(\delta_2) - \bar{\chi}(\delta_1 + \delta_2) \geq 0, \label{goal2}
\end{align}
holds for any $\delta_1, \delta_2 \in \R_+$. 
We proceed with the following four lemmas.

\begin{lemma}\label{properties}
	The following statements hold: (i) $a_i \geq a_{i+1}$; (ii) $B_{i + 1} = B_i + a_{i + 1}$; 
	(iii) $- a_{i+1} + v_{i-1} \leq - a_i + v_i $; (iv) $\psi_{i} \leq \psi_{i - 1}$; 
	and (v) $\chi (A_i ) = \tau \sum_{j = 1}^{i} \psi_{j - 1} + \sum_{j = 1}^{i}(h(b_j + v_{j-1}) - h(v_{j - 1}))$.
\end{lemma}
\begin{proof}
	Statements (i)-(iii) follow from the definitions in \eqref{def:b} and \eqref{ABdef}. 
	Statement (iv) follows from \eqref{def-rho} and \cref{lem:slope}.
	Finally, statement (v) follows from \eqref{chi_Ai}.
	\qedhere
\end{proof}

\begin{lemma}\label{lem:translation}
	Let $\delta, \Delta \geq 0$ be such that $A_i < \delta \leq \delta + \Delta \leq A_{i + 1}$ for some $i \in \Z_{+}$.
	Then, 
	\begin{equation}\label{ineq1}
		\chi (\delta + \Delta) - \chi (\delta) 
		\leq \chi \left( \delta - \sum_{j = t}^{i+1}a_j  + \Delta \right)
		- \chi \left( \delta - \sum_{j = t}^{i+1}a_j  \right), \ \forall \ t = 2, \ldots, i+1.
	\end{equation}
\end{lemma}
\begin{proof}
	We show the statement by induction on $t$.	
	We first consider the base case $t = i+1$. 
	In this case, \eqref{ineq1} reduces to 
	\begin{equation}\label{tmeq}
		\chi (\delta + \Delta) - \chi (\delta) \leq  \chi (\delta - a_{i+1}  + \Delta) - \chi (\delta - a_{i+1} ).
	\end{equation}
	We consider the following three cases.
	\begin{itemize}
		\item [(i)] $A_i  < \delta \leq \delta + \Delta \leq B_{i+1} $.
		Then, by the definition of $\chi$ in \eqref{chi}, we have 
		\begin{align*}
			\chi (\delta + \Delta) - \chi (\delta) & = 
			h(\delta - A_i  + v_i + \Delta) - h(\delta - A_i  + v_i) \\
			& = h(\delta - A_{i - 1} - a_i  + v_i + \Delta) - h(\delta - A_{i - 1} - a_i  + v_i) \\ 
			&  \stackrel{(a)}{\leq}  h(\delta - A_{i - 1} - a_{i + 1} + v_{i - 1} + \Delta) - h(\delta - A_{i - 1} - a_{i + 1} + v_{i - 1}) \\
			&  \stackrel{(b)}{=}  \chi (\delta - a_{i+1}  + \Delta) - \chi (\delta - a_{i+1} ),
		\end{align*}
		where (a) follows from \cref{properties} (iii), $\Delta \geq 0$, and \cref{lem:slope},
		and (b) follows from $A_{i-1}  < \delta - a_{i+1}  \leq \delta - a_{i+1}  + \Delta \leq B_{i} $ (which can be derived from $A_i  < \delta \leq \delta + \Delta \leq B_{i+1} $ and \cref{properties} (i)--(ii)).
		\item [(ii)] $A_i  < \delta \leq B_{i+1}$ 
		and $B_{i+1}  + \ell \epsilon < \delta + \Delta \leq B_{i+1}  + (\ell + 1) \epsilon$ for some $\ell \in \{0, \ldots, \tau - 1\}$. 
		Then $\Delta - \ell \epsilon \geq 0$, and by \cref{properties} (i)-(ii), it follows
		\begin{equation}\label{tmpeq1}
			A_{i-1}  < \delta - a_{i+1}  \leq B_{i}~\text{and}~ 
			B_{i}  + \ell \epsilon < \delta - a_{i+1}  + \Delta \leq B_{i}  + (\ell + 1) \epsilon.
		\end{equation}
		By the definition of $\chi$ in \eqref{chi}, 
		we have
		\begin{footnotesize}
			\begin{align*}
				\chi (\delta + \Delta) - \chi (\delta) 
				& =	h(\delta - A_i  + v_i + \Delta - \ell \epsilon) - h(\delta - A_i  + v_i) + \ell \psi_{i} \\
				& =	h(\delta - A_{i - 1} - a_i  + v_i  + \Delta - \ell \epsilon) - h(\delta - A_{i - 1} - a_i  + v_i) + \ell \psi_{i} \\
				& \stackrel{(a)}{\leq} h(\delta - A_{i - 1}  - a_{i+1} + v_{i-1}  + \Delta - \ell \epsilon) - h(\delta - A_{i - 1}  - a_{i+1} + v_{i-1}) + \ell \psi_{i - 1} \\
				& \stackrel{(b)}{=} \chi (\delta - a_{i+1}  + \Delta) - \chi (\delta - a_{i+1} ),
			\end{align*}
		\end{footnotesize}%
		where (a) follows from \cref{properties} (iii)-(iv), $\Delta - \ell \epsilon \geq 0$, and \cref{lem:slope}, and (b) follows from \eqref{tmpeq1}.
		\item [(iii)] $B_{i+1}  + \ell_1 \epsilon < \delta \leq B_{i+1}  + (\ell_1 + 1) \epsilon$ 
		and $B_{i+1}  + \ell_2 \epsilon < \delta + \Delta \leq B_{i+1}  + (\ell_2 + 1) \epsilon$ for some $\ell_1, \ell_2 \in \{0, \ldots, \tau - 1\}$. 
		Then $\ell_2 \geq \ell_1$, and by \cref{properties} (ii), it follows
		\begin{equation}\label{tmpeq2}
			B_{i}  + \ell_1 \epsilon < \delta - a_{i+1}  \leq B_{i}  + (\ell_1 + 1) \epsilon~\text{and}~ 
			B_{i}  + \ell_2 \epsilon < \delta - a_{i+1}  + \Delta \leq B_{i}  + (\ell_2 + 1) \epsilon.
		\end{equation}
		We further consider the following two subcases.
		\begin{itemize}
			\item [(iii.1)] $\ell_2 = \ell_1$. By the definition of $\chi$ in \eqref{chi},
			we have
			\begin{align*}
				& \chi (\delta + \Delta) - \chi (\delta) \\
				& \qquad =	h(\delta - A_i  + v_i - \ell_1 \epsilon + \Delta) - h(\delta - A_i  + v_i - \ell_1 \epsilon) \\
				& \qquad =	h(\delta - A_{i - 1} - a_i  + v_i - \ell_1 \epsilon + \Delta) - h(\delta - A_{i - 1} - a_i  + v_i - \ell_1 \epsilon) \\
				& \qquad \stackrel{(a)}{\leq} h(\delta - A_{i - 1} - a_{i + 1} + v_{i-1} - \ell_1 \epsilon + \Delta) - h(\delta - A_{i - 1} - a_{i + 1} + v_{i-1} - \ell_1 \epsilon) \\
				& \qquad \stackrel{(b)}{=} \chi (\delta - a_{i+1}  + \Delta) - \chi (\delta - a_{i+1} ).
			\end{align*}
			where (a) follows from \cref{properties} (iii), $\Delta \geq 0$, and \cref{lem:slope},
			and (b) follows from $\ell_2 = \ell_1$ and \eqref{tmpeq2}.
			\item [(iii.2)] $\ell_2 > \ell_1$.
			Let 
			\begin{equation}
				\begin{aligned}\label{def-Del}
					\Delta_1 := & (\delta + \Delta - A_i   - \ell_2 \epsilon + v_i) - (b_{i+1} + v_i), \\
					\Delta_2 := & (b_{i+1} + v_i + \epsilon) - (\delta - A_i  - \ell_1 \epsilon + v_i). 
				\end{aligned}
			\end{equation}
			Then 
			\begin{equation*}
				\begin{aligned}
					& \Delta_1= {(\delta + \Delta) - (A_i  + b_{i+1} + \ell_2 \epsilon)} = (\delta + \Delta) - (B_{i+1}  + \ell_2 \epsilon) \geq 0,\\
					& \Delta_2 = {(b_{i+1} + A_i  + (\ell_1 + 1) \epsilon) - \delta}
					= (B_{i+1}  + (\ell_1 + 1) \epsilon) - \delta \geq 0,
				\end{aligned}
			\end{equation*}
			and by $a_{i+1}-b_{i+1}= a_i - b_i = \tau \epsilon$, it follows
			\begin{equation}\label{lem30:cond}
				\begin{aligned}
					& b_i + v_{i - 1} + \Delta_1 = \delta - a_{i+1}  + \Delta - A_{i-1}  - \ell_2 \epsilon + v_{i - 1}, \\
					& 	b_i + v_{i - 1} + \epsilon - \Delta_2 = \delta - a_{i+1}  - A_{i - 1}  - \ell_1 \epsilon + v_{i - 1}.
				\end{aligned}
			\end{equation}
			By the definition of $\chi$ in \eqref{chi}, we have
			\begin{equation*}
				\label{lem30-cond}
				\begin{aligned}
					& \chi (\delta + \Delta) - \chi (\delta) \\
					& \qquad 
					=	h(\delta + \Delta - A_i   - \ell_2 \epsilon + v_i) - h(\delta - A_i  - \ell_1 \epsilon + v_i) + (\ell_2 - \ell_1) \psi_{i} \\
					& \qquad
					\stackrel{(a)}{=} [h(b_{i+1} + v_i + \Delta_1) - h(b_{i+1} + v_i)] \\
					& \qquad \quad
					+ [h(b_{i+1} + v_i + \epsilon) - h(b_{i+1} + v_i + \epsilon - \Delta_2)] 
					+ [(\ell_2 - \ell_1 - 1)\psi_{i}] \\
					& \qquad 
					\stackrel{(b)}{\leq} [h(b_{i} + v_{i-1} + \Delta_1) - h(b_{i} + v_{i-1})] \\
					& \qquad \quad
					+ [h(b_{i} + v_{i-1} + \epsilon) - h(b_{i} + v_{i-1} + \epsilon - \Delta_2)] 
					+ [(\ell_2 - \ell_1 - 1)\psi_{i - 1}] \\
					& \qquad 
					= h(b_{i} + v_{i-1} + \Delta_1) - h(b_{i} + v_{i-1} + \epsilon - \Delta_2) + (\ell_2 - \ell_1) \psi_{i - 1} \\ 
					& \qquad
					\stackrel{(c)}{=}
					h(\delta - a_{i+1}  + \Delta - A_{i-1}  - \ell_2 \epsilon + v_{i - 1}) - h(\delta - a_{i+1}  - A_{i - 1}  - \ell_1 \epsilon + v_{i - 1}) \\
					& \qquad \quad + (\ell_2 - \ell_1) \psi_{i - 1} \\
					& \qquad
					\stackrel{(d)}{=} \chi (\delta - a_{i+1}  + \Delta) - \chi (\delta - a_{i+1} ).
				\end{aligned}
			\end{equation*} 
			where (a) follows from \eqref{def-Del} and $\psi_{i} = h(b_{i+1} + v_i + \epsilon) - h(b_{i+1} + v_i)$,
			(b) follows from $b_i + v_{i - 1} \leq b_{i + 1} + v_i$, $\Delta_1 \geq 0$, $\Delta_2 \geq 0$, \cref{lem:slope}, and \cref{properties} (iv),
			(c) follows from \eqref{lem30:cond},
			and (d) follows from \eqref{tmpeq2} and the definition of $\chi $ in \eqref{chi}.
		\end{itemize}
	\end{itemize}
	
	Now suppose that \eqref{ineq1} holds for $t = t_0 + 1, \ldots, i+1$.
	Then,
	\begin{equation*}
		\begin{aligned}
			\chi (\delta + \Delta) - \chi (\delta) 
			& \stackrel{(a)}{\leq} 
			\chi \left( \underbrace{\delta - \sum_{j = t_0 + 1}^{i+1}a_j}_{\delta'}  + \Delta \right) 
			- \chi \left( \underbrace{\delta - \sum_{j = t_0 + 1}^{i+1}a_j}_{\delta'} \right) \\
			& \stackrel{(b)}{\leq} 
			\chi \left(\delta - \sum_{j = t_0 + 1}^{i+1}a_j - a_{t_0}  + \Delta\right) - \chi \left(\delta - \sum_{j = t_0 + 1}^{i+1}a_j - a_{t_0} \right) \\
			& = \chi \left( \delta - \sum_{j = t_0}^{i+1}a_j  + \Delta \right) 
			- \chi \left( \delta - \sum_{j = t_0}^{i+1}a_j  \right),
		\end{aligned}
	\end{equation*}
	where (a) follows from the induction hypothesis,
	and (b) follows from $A_{t_0-1} < \delta - \sum_{j = t_0 + 1}^{i+1}a_j \leq \delta - \sum_{j = t_0 + 1}^{i+1}a_j + \Delta \leq A_{t_0}$ (by $A_i  < \delta \leq \delta + \Delta \leq A_{i+1} $ and $a_{i+1} \leq a_i \leq \cdots \leq a_{t_0}$)
	and the {base case} in \eqref{tmeq} (with $i= t_0-1$). 
	Thus, \eqref{ineq1} also holds for $t = t_0$, which completes the proof.
	\qedhere
\end{proof}

\begin{lemma}\label{lem:same-segment1}
	Let $\delta, \Delta \geq 0$ be such that $A_i < \delta \leq \delta + \Delta \leq A_{i + 1}$ for some $i \in \Z_{+}$.
	Then, the following statements hold:
	(i) $\chi (\delta +\Delta) - \chi (\delta) \geq \chi (A_{i+1} ) - \chi (A_{i+1}  - \Delta)$;
	and (ii) $\chi(A_i + \Delta) - \chi(A_i) \geq \chi (\delta +\Delta) - \chi (\delta)$.
\end{lemma}

Notice that the parameters $a_i$, $A_i$, and $B_i$ in \eqref{ABdef} and function $\chi$ in \eqref{chi} all depend on $\tau$.
For clarity in the subsequent proof of \cref{lem:same-segment1}, we explicitly denote this dependence by writing them as $a_i^\tau$, $A_i^\tau$, $B_i^\tau$, and $\chi^\tau$.
As a result, the statements in \cref{lem:same-segment1} reduce to 
\begin{align}
	& \chi^\tau(\delta +\Delta) - \chi^\tau (\delta) \geq \chi^\tau (A^\tau_{i+1} ) - \chi^\tau (A^\tau_{i+1}  - \Delta),\label{tmp1}\\
	& \chi^\tau(A^\tau_i + \Delta) - \chi^\tau(A^\tau_i) \geq \chi^\tau (\delta +\Delta) - \chi^\tau (\delta)\label{tmp2} , 
\end{align}
where $\delta, \Delta \geq 0$ satisfying $A^\tau_i < \delta \leq \delta + \Delta \leq A^\tau_{i + 1}$ for some $i \in \Z_{+}$.
To prove \cref{lem:same-segment1}, or equivalently, \eqref{tmp1}--\eqref{tmp2}, we need the following results.

\begin{claim}\label{properties2}
	For $\tau \geq 2$ and $i \geq 0$, it follows that
	\begin{itemize}
		\item [(i)] $A_i^\tau = A_i^{\tau - 1} + i \epsilon$ and $B_{i+1}^{\tau} = B_{i+1}^{\tau - 1} + i \epsilon$;
		\item [(ii)] $\chi^{\tau}(\delta') - \chi^{\tau}(\delta'') = \chi^{\tau}(\delta' - \epsilon) - \chi^{\tau}(\delta'' - \epsilon),  ~ \forall~
		\delta', \delta'' \in [A_{i+1}^{\tau} - \epsilon, A_{i+1}^{\tau}]$;
		\item [(iii)] $\chi^{\tau - 1}(\delta) + \sum_{j = 1}^{i} \psi_{j - 1} = \chi^\tau(\delta + i \epsilon), 
		~ \forall ~ \delta \in [A_i^{\tau - 1}, A_{i+1}^{\tau - 1}]$;
		\item  [(iv)] $ \chi^{\tau - 1}(\delta') - \chi^{\tau - 1}(\delta'')
		= \chi^\tau(\delta' + i \epsilon) - \chi^\tau(\delta'' + i \epsilon), 
		~ \forall ~ \delta', \delta'' \in [A_i^{\tau - 1}, A_{i+1}^{\tau - 1}]$.
	\end{itemize}
\end{claim}

\begin{proof}
	Statements (i) and (ii) follow directly from \eqref{ABdef} and \eqref{chi}, respectively.
	By statement (i), $\delta \in [A_i^{\tau - 1}, B_{i+1}^{\tau - 1}]$ holds if and only if 
	$\delta + i \epsilon \in [A_i^{\tau}, B_{i+1}^{\tau}]$ holds, and 
	$\delta \in [B_{i+1}^{\tau - 1} + \ell \epsilon, B_{i+1}^{\tau - 1} + (\ell + 1) \epsilon]$ holds for some $\ell \in \{0,\ldots, \tau-1\}$ if
	and only if
	$\delta + i \epsilon \in [B_{i+1}^{\tau} + \ell \epsilon, B_{i+1}^{\tau} + (\ell + 1) \epsilon]$ holds.
	From \cref{properties} (v), 
	we have $\chi^{\tau}(A_i^{\tau}) = \tau \sum_{j = 1}^{i}\psi_{j - 1} + \sum_{j = 1}^{i}(h(b_j + v_{j - 1}) - h(v_{j - 1}))
	= \chi^{\tau - 1}(A_i^{\tau-1}) + \sum_{j = 1}^{i}\psi_{j - 1}$.
	Thus, from the definition of $\chi^\tau$ in \eqref{chi} and statement (i), statement (iii) also holds.
	Finally, statement (iv) follows directly from statement (iii).
\end{proof}

\begin{proof}[Proof of \cref{lem:same-segment1}]
	We shall prove \eqref{tmp1} and \eqref{tmp2} by induction on $\tau$.
	If $\tau = 1$, then $\ell$ must equal to zero in the third case of \eqref{chi}. 
	In this case, we have 
	$\chi^\tau(\delta) = h(\delta - A_i^\tau + v_i) + \chi^\tau(A_i^\tau) - h(v_i)$ for $\delta \in [A_i^\tau, A_{i+1}^\tau]$, which, together with the concavity of function $h$ and \cref{lem:slope}, implies
	\begin{equation*}
		\chi^\tau(A_i^\tau + \Delta) - \chi^\tau(A_i^\tau) \geq 
		\chi^\tau(\delta + \Delta) - \chi^\tau(\delta)
		\geq \chi^\tau(A_{i+1}^\tau) - \chi^\tau(A_{i+1}^\tau - \Delta).
	\end{equation*}
	Suppose that  \eqref{tmp1} and \eqref{tmp2}  hold for $\tau = \tau_0 - 1 \geq 1$.
	Then, for any $\delta', \Delta' \geq 0$ with
	$A_{i}^{\tau_0 - 1} < \delta' \leq \delta' + \Delta' \leq A_{i+1}^{\tau_0 - 1}$, it follows
	\begin{equation*}
		\begin{aligned}
			& \chi^{\tau_0 - 1}(\delta' + \Delta') - \chi^{\tau_0 - 1}(\delta')
			\geq 
			\chi^{\tau_0 - 1}(A_{i+1}^{\tau_0 - 1}) - \chi^{\tau_0 - 1}(A_{i+1}^{\tau_0 - 1} - \Delta'), \\
			& \chi^{\tau_0 - 1}(A_i^{\tau_0 - 1} + \Delta') - \chi^{\tau_0 - 1}(A_i^{\tau_0 - 1}) 
			\geq \chi^{\tau_0 - 1}(\delta '+ \Delta') - \chi^{\tau_0 - 1}(\delta').
		\end{aligned}
	\end{equation*}
	Letting $\delta'' := \delta' + i \varepsilon$, then by \cref{properties2} (i), 
	it follows $A_{i}^{\tau_0} < \delta'' \leq \delta'' + \Delta' \leq A_{i+1}^{\tau_0} - \epsilon$, and by \cref{properties2} (iv), it follows
	\begin{align}
		& \chi^{\tau_0}(\delta'' + \Delta') - \chi^{\tau_0}(\delta'')
		\geq 
		\chi^{\tau_0}(A_{i+1}^{\tau_0} - \epsilon) - \chi^{\tau_0}(A_{i+1}^{\tau_0} - \epsilon - \Delta'), \label{induction-hypothesis1} \\
		& \chi^{\tau_0}(A_i^{\tau_0} + \Delta') - \chi^{\tau_0}(A_i^{\tau_0}) 
		\geq \chi^{\tau_0}(\delta'' + \Delta') - \chi^{\tau_0}(\delta''). \label{induction-hypothesis2}
	\end{align}
	To prove  that \eqref{tmp1} and \eqref{tmp2}  hold for $\tau = \tau_0 \geq 2$, we  consider the following three cases.
	\begin{itemize}
		\item [(i)] $A_i^{\tau_0} < \delta \leq \delta + \Delta \leq B_{i+1}^{\tau_0} + (\tau_0 - 1)\epsilon = A_{i+1}^{\tau_0} - \epsilon$. 
		Then, \eqref{tmp2} follows from \eqref{induction-hypothesis2} with $\Delta' = \Delta$ and $\delta'' = \delta$,
		and by \eqref{induction-hypothesis1}, we have
		\begin{equation*}
			\chi^{\tau_0}(\delta + \Delta) - \chi^{\tau_0}(\delta)
			\geq \chi^{\tau_0}(A_{i+1}^{\tau_0} - \epsilon) - \chi^{\tau_0}(A_{i+1}^{\tau_0} - \epsilon - \Delta).
		\end{equation*}
		To show \eqref{tmp1},
		it suffices to prove
		\begin{equation*}
			\chi^{\tau_0}(A_{i+1}^{\tau_0} - \epsilon) - \chi^{\tau_0}(A_{i+1}^{\tau_0} - \epsilon - \Delta) 
			\geq
			\chi^{\tau_0}(A_{i+1}^{\tau_0}) - \chi^{\tau_0}(A_{i+1}^{\tau_0} - \Delta).
		\end{equation*}
		If $\Delta \leq \epsilon$,
		then the above inequality is implied by \cref{properties2} (ii).
		Otherwise, $\Delta > \epsilon$, and let $\Delta_1 := \Delta - \epsilon > 0$.
		Then, we obtain
			\begin{align*}
				& \chi^{\tau_0}(A_{i+1}^{\tau_0} - \epsilon) - \chi^{\tau_0}(A_{i+1}^{\tau_0} - \epsilon - \Delta) \\
				& \qquad =
				[\chi^{\tau_0}(A_{i+1}^{\tau_0} - \epsilon) - \chi^{\tau_0}(A_{i+1}^{\tau_0} - 2 \epsilon)] 
				+ [\chi^{\tau_0}(A_{i+1}^{\tau_0} - 2 \epsilon) - \chi^{\tau_0}(A_{i+1}^{\tau_0} - 2 \epsilon - \Delta_1)] \\ 
				& \qquad \stackrel{(a)}{\geq}
				[\chi^{\tau_0}(A_{i+1}^{\tau_0} - \epsilon) - \chi^{\tau_0}(A_{i+1}^{\tau_0} - 2 \epsilon)] 
				+ [\chi^{\tau_0}(A_{i+1}^{\tau_0} - \epsilon) - \chi^{\tau_0}(A_{i+1}^{\tau_0} - \epsilon - \Delta_1)] \\
				& \qquad \stackrel{(b)}{=}
				[\chi^{\tau_0}(A_{i+1}^{\tau_0}) - \chi^{\tau_0}(A_{i+1}^{\tau_0} - \epsilon)] 
				+ [\chi^{\tau_0}(A_{i+1}^{\tau_0} - \epsilon) - \chi^{\tau_0}(A_{i+1}^{\tau_0} - \epsilon - \Delta_1)] \\
				& 
				\qquad = \chi^{\tau_0}(A_{i+1}^{\tau_0}) - \chi^{\tau_0}(A_{i+1}^{\tau_0} - \Delta),
			\end{align*}
		where (a) follows from \eqref{induction-hypothesis1}
		and $A_{i+1}^{\tau_0} - 2 \epsilon - \Delta_1 = A_{i+1}^{\tau_0} - \epsilon - \Delta \geq \delta > A_i^{\tau_0}$ 
		(as $A_i^{\tau_0} < \delta \leq \delta + \Delta \leq A_{i+1}^{\tau_0} - \epsilon$),
		and (b) follows from \cref{properties2} (ii).
		\item [(ii)] $A_{i+1}^{\tau_0} - \epsilon < \delta \leq \delta + \Delta \leq A_{i+1}^{\tau_0}$.
		Then, \eqref{tmp1} follows directly from the concavity of $\chi^{\tau_0}(\delta)$ on $[A_{i+1}^{\tau_0} - \epsilon, A_{i+1}^{\tau_0}]$ (as $\chi^{\tau_0}$ is continuous piecewise concave)
		and \cref{lem:slope}.
		\eqref{tmp2} also holds since
		\begin{align*}
			\chi^{\tau_0}(\delta + \Delta) - \chi^{\tau_0}(\delta) 
			& \stackrel{(a)}{=} \chi^{\tau_0}(\delta + \Delta - \epsilon) - \chi^{\tau_0}(\delta - \epsilon) \\
			& \stackrel{(b)}{\leq} \chi^{\tau_0}(A_i^{\tau_0} + \Delta) - \chi^{\tau_0}(A_i^{\tau_0}),
		\end{align*}
		where (a) follows from \cref{properties2} (ii),
		and (b) follows from \eqref{induction-hypothesis2} 
		and $A_{i}^{\tau_0} \leq A_{i+1}^{\tau_0} - 2 \epsilon 
		< \delta - \epsilon \leq \delta + \Delta - \epsilon \leq A_{i+1}^{\tau_0} - \epsilon$ 
		(as $\tau_0 \geq 2$ and $A_{i+1}^{\tau_0} - \epsilon < \delta \leq \delta + \Delta \leq A_{i+1}^{\tau_0}$).
		\item [(iii)] $A_i^{\tau_0} < \delta \leq A_{i+1}^{\tau_0} - \epsilon < \delta + \Delta \leq A_{i+1}^{\tau_0}$.
		We show \eqref{tmp1} and \eqref{tmp2} (with $\tau=\tau_0$) hold in (iii.1) and (iii.2), respectively. 
		\begin{itemize}
			\item [(iii.1)] 
			Let $\Delta_2 : = A_{i+1}^{\tau_0} - (\delta + \Delta)$.
			Then it must follow that $0 \leq \Delta_2 < \epsilon$.
			If $\Delta \geq \epsilon$, or equivalently, $\delta \leq A_{i+1}^{\tau_0} - \epsilon - \Delta_2$,
			then we have 
			$A_{i}^{\tau_0} < \delta < \delta + \Delta_2 \leq A_{i+1}^{\tau_0} - \epsilon$, and thus
			\begin{align*}
				\chi^{\tau_0}(\delta + \Delta_2) - \chi^{\tau_0}(\delta) 
				& \stackrel{(a)}{\geq} \chi^{\tau_0}(A_{i+1}^{\tau_0} - \epsilon) - \chi^{\tau_0}(A_{i+1}^{\tau_0} - \epsilon - \Delta_2) \\
				& \stackrel{(b)}{=}
				\chi^{\tau_0}(A_{i+1}^{\tau_0}) - \chi^{\tau_0}(A_{i+1}^{\tau_0} - \Delta_2) \\
				& = \chi^{\tau_0}(A_{i+1}^{\tau_0}) - \chi^{\tau_0}(\delta + \Delta),
			\end{align*}
			where (a) follows from \eqref{induction-hypothesis1},
			and (b) follows from \cref{properties2} (ii) and $\Delta_2 < \epsilon$. 
			Therefore, \eqref{tmp1} holds since 			
			\begin{align*}
				\chi^{\tau_0}(\delta + \Delta) - \chi^{\tau_0}(\delta) & =
				[\chi^{\tau_0}(\delta + \Delta) - \chi^{\tau_0}(\delta + \Delta_2) ]
				+ [\chi^{\tau_0}(\delta + \Delta_2) - \chi^{\tau_0}(\delta)] \\
				& \geq 
				[\chi^{\tau_0}(\delta + \Delta) - \chi^{\tau_0}(\delta + \Delta_2) ]
				+ [\chi^{\tau_0}(A_{i+1}^{\tau_0}) - \chi^{\tau_0}(\delta + \Delta)] \\
				& = \chi^{\tau_0}(A_{i+1}^{\tau_0}) - \chi^{\tau_0}(A_{i+1}^{\tau_0} - \Delta).
			\end{align*}
			We now consider the case $\Delta < \epsilon$, or equivalently, $\delta > A_{i+1}^{\tau_0} - \epsilon - \Delta_2 $.
			Observe that
			\begin{equation*}
				\begin{aligned}
					\chi^{\tau_0}(\delta + \Delta) - \chi^{\tau_0}(A_{i+1}^{\tau_0} - \epsilon) 
					& =
					\chi^{\tau_0}(\delta + \Delta) - \chi^{\tau_0}(\delta + \Delta - (\epsilon - \Delta_2)) \\
					& \stackrel{(a)}{\geq} 
					\chi^{\tau_0}(\delta + \epsilon) - \chi^{\tau_0}(\delta + \epsilon - (\epsilon - \Delta_2)) \\
					& = \chi^{\tau_0}(\delta + \epsilon) - \chi^{\tau_0}(A_{i+1}^{\tau_0} - \Delta),
				\end{aligned}
			\end{equation*}
			where (a) follows from $\epsilon - \Delta_2 > 0$, $A_{i+1}^{\tau_0}-\epsilon =\delta + \Delta - (\epsilon - \Delta_2)< \delta + \Delta < \delta + \epsilon \leq A_{i+1}^{\tau_0}$ (as $\delta \leq A_{i+1}^{\tau_0} - \epsilon$), $A_{i+1}^{\tau_0} - \epsilon< \delta + \epsilon - (\epsilon - \Delta_2)= A_{i+1}^{\tau_0} - \Delta \leq A_{i+1}^{\tau_0} $, the concavity of 
			$\chi^{\tau_0}(\delta)$ on $[A_{i+1}^{\tau_0} - \epsilon, A_{i+1}^{\tau_0}]$, and \cref{lem:slope}.
			As a result, 
			\begin{align*}
				\chi^{\tau_0}(\delta + \Delta) - \chi^{\tau_0}(\delta) & =
				[	\chi^{\tau_0}(\delta + \Delta) - \chi^{\tau_0}(A_{i+1}^{\tau_0} - \epsilon) ]
				+ [\chi^{\tau_0}(A_{i+1}^{\tau_0} - \epsilon) - \chi^{\tau_0}(\delta)] \\
				& \geq 
				[ \chi^{\tau_0}(\delta + \epsilon) - \chi^{\tau_0}(A_{i+1}^{\tau_0} - \Delta)]
				+  [\chi^{\tau_0}(A_{i+1}^{\tau_0} - \epsilon) - \chi^{\tau_0}(\delta)]\\
				& \stackrel{(a)}{=} 	[ \chi^{\tau_0}(\delta + \epsilon) - \chi^{\tau_0}(A_{i+1}^{\tau_0} - \Delta)]
				+  [ \chi^{\tau_0}(A_{i+1}^{\tau_0}) - \chi^{\tau_0}(\delta + \epsilon)]\\
				& = \chi^{\tau_0}(A_{i+1}^{\tau_0}) -\chi^{\tau_0}(A_{i+1}^{\tau_0} - \Delta),
			\end{align*}
			where (a) follows from  $A_{i+1}^{\tau_0} - \epsilon < \delta + \Delta < \delta + \epsilon \leq A_{i+1}^{\tau_0}$
			and \cref{properties2} (ii).
			\item [(iii.2)] 
			We first consider the case $A_i^{\tau_0} + \Delta \geq A_{i+1}^{\tau_0} - \epsilon$.
			Letting $\Delta_3 := \delta - A_i^{\tau_0}$, then by $\delta > A_{i}^{\tau_0}$, $\Delta_3 > 0$ must hold.
			Observe that
			\begin{align*}
				\chi^{\tau_0}(\delta) - \chi^{\tau_0}(A_i^{\tau_0}) 
				& = \chi^{\tau_0}(A_i^{\tau_0} + \Delta_3) - \chi^{\tau_0}(A_i^{\tau_0}) \\
				& \stackrel{(a)}{\geq} \chi^{\tau_0}(A_i^{\tau_0} + \Delta - \epsilon + \Delta_3) - \chi^{\tau_0}(A_i^{\tau_0} + \Delta - \epsilon) \\
				& = \chi^{\tau_0}(\delta + \Delta - \epsilon) - \chi^{\tau_0}(A_i^{\tau_0} + \Delta - \epsilon) \\
				& \stackrel{(b)}{=} \chi^{\tau_0}(\delta + \Delta) - \chi^{\tau_0}(A_i^{\tau_0} + \Delta),
			\end{align*}
			where (a) follows from \eqref{induction-hypothesis2}, $\Delta_3 > 0$, and
			$A_i^{\tau_0} \leq A_{i+1}^{\tau_0} - 2 \epsilon \leq A_i^{\tau_0} + \Delta - \epsilon 
			< A_i^{\tau_0} + \Delta - \epsilon + \Delta_3
			\leq A_{i + 1}^{\tau_0} - \epsilon$ 
			(as $\tau_0 \geq 2$, $A_{i + 1}^{\tau_0} - \epsilon \leq A_i^{\tau_0} + \Delta$,
			and $A_i^{\tau_0} + \Delta_3 + \Delta = \delta + \Delta \leq A_{i + 1}^{\tau_0}$),
			and (b) follows from \cref{properties2} (ii).
			Therefore, \eqref{tmp2} holds since
			\begin{align*}
				\chi^{\tau_0}(A_i^{\tau_0} + \Delta) - \chi^{\tau_0}(A_i^{\tau_0})
				& =
				[\chi^{\tau_0}(A_i^{\tau_0} + \Delta) - \chi^{\tau_0}(\delta)] 
				+ [\chi^{\tau_0}(\delta) - \chi^{\tau_0}(A_i^{\tau_0})]\\
				& \geq 
				[\chi^{\tau_0}(A_i^{\tau_0} + \Delta) - \chi^{\tau_0}(\delta)] 
				+ [\chi^{\tau_0}(\delta + \Delta) - \chi^{\tau_0}(A_i^{\tau_0} + \Delta)]\\
				& = \chi^{\tau_0}(\delta + \Delta) - \chi^{\tau_0}(\delta).
			\end{align*}
			Now we consider the case $A_i^{\tau_0} + \Delta < A_{i+1}^{\tau_0} - \epsilon$.
			Let $\Delta_4 := \delta + \Delta - (A_{i+1}^{\tau_0} - \epsilon)$.
			Then by $\delta \leq A_{i+1}^{\tau_0} - \epsilon < \delta + \Delta \leq A_{i+1}^{\tau_0}$, $0 < \Delta_4 < \epsilon$ and $\Delta \geq \Delta_4$ must hold.
			Observe that 
			\begin{equation*}
				\begin{aligned}
					\chi^{\tau_0}(A_i^{\tau_0} + \Delta_4) - \chi^{\tau_0}(A_i^{\tau_0})
					& \stackrel{(a)}{\geq} 
					\chi^{\tau_0}(A_{i+1}^{\tau_0} - 2 \epsilon + \Delta_4) - \chi^{\tau_0}(A_{i+1}^{\tau_0} - 2 \epsilon ) \\
					& \stackrel{(b)}{=} 
					\chi^{\tau_0}(A_{i+1}^{\tau_0} - \epsilon + \Delta_4) - \chi^{\tau_0}(A_{i+1}^{\tau_0} - \epsilon) \\
					& = \chi^{\tau_0}(\delta + \Delta) - \chi^{\tau_0}(A_{i+1}^{\tau_0} - \epsilon),
				\end{aligned}
			\end{equation*}
			where (a) follows from $A_i^{\tau_0} \leq A_{i + 1}^{\tau_0} - 2 \epsilon < A_{i + 1}^{\tau_0} - 2 \epsilon + \Delta_4 < A_{i + 1}^{\tau_0} - \epsilon$ 
			(as $\tau_0 \geq 2$ and $0 < \Delta_4 < \epsilon$)
			and \eqref{induction-hypothesis2},
			and (b) follows from \cref{properties2} (ii). 
			Therefore,
			\begin{equation*}
				\small
				\begin{aligned}
					& \chi^{\tau_0}(A_i^{\tau_0} + \Delta) - \chi^{\tau_0}(A_i^{\tau_0}) \\
					& \qquad= [\chi^{\tau_0}(A_i^{\tau_0} + \Delta) - \chi^{\tau_0}(A_i^{\tau_0} + \Delta_4)]
					+ [\chi^{\tau_0}(A_i^{\tau_0} + \Delta_4) - \chi^{\tau_0}(A_i^{\tau_0})] \\
					& \qquad
					\geq [\chi^{\tau_0}(A_i^{\tau_0} + \Delta) - \chi^{\tau_0}(A_i^{\tau_0} + \Delta_4)]
					+ [\chi^{\tau_0}(\delta + \Delta) - \chi^{\tau_0}(A_{i+1}^{\tau_0} - \epsilon)] \\
					&  \qquad 
					\stackrel{(a)}{\geq} 
					[\chi^{\tau_0}(A_{i + 1}^{\tau_0} - \epsilon) - \chi^{\tau_0}(A_{i + 1}^{\tau_0} - \epsilon - (\Delta - \Delta_4))]
					+ [\chi^{\tau_0}(\delta + \Delta) - \chi^{\tau_0}(A_{i+1}^{\tau_0} - \epsilon)] \\
					& \qquad =
					[\chi^{\tau_0}(A_{i + 1}^{\tau_0} - \epsilon) - \chi^{\tau_0}(\delta)]
					+ [\chi^{\tau_0}(\delta + \Delta) - \chi^{\tau_0}(A_{i+1}^{\tau_0} - \epsilon)] \\
					& \qquad = 
					\chi^{\tau_0}(\delta + \Delta) - \chi^{\tau_0}(\delta),
				\end{aligned}
			\end{equation*}
			where (a) follows from $A_i^{\tau_0} < A_i^{\tau_0} + \Delta_4 \leq A_i^{\tau_0} + \Delta 
			< A_{i+1}^{\tau_0} - \epsilon$ and \eqref{induction-hypothesis1}. \qedhere
		\end{itemize}
	\end{itemize}
\end{proof}

\begin{lemma}\label{lem:translation2}
	Let $\Delta \in [0, a_{i+1}]$ for some integer $i \in \Z_+$.
	Then for any $\delta \geq A_i $, we have
	\begin{equation}\label{3.37-result}
		\chi (A_i  + \Delta) - \chi (A_i ) \geq 
		\chi (\delta + \Delta) - \chi (\delta).
	\end{equation}
\end{lemma}
\begin{proof}
	Suppose that $\delta \in [A_{i_1} , A_{i_1 + 1} ]$ 
	and $\delta + \Delta \in [A_{i_2} , A_{i_2 + 1} ]$, where $i \leq i_1 \leq i_2$.
	We shall establish \eqref{3.37-result} by induction on $i_2 - i_1$.
	If $i_2 - i_1 = 0$, then
	\begin{align*}
		\chi (\delta + \Delta) - \chi (\delta) 
		& \stackrel{(a)}{\leq} 
		\chi \left(\underbrace{\delta - \sum_{j = i + 2}^{i_1 + 1}a_j}_{\delta'} + \Delta\right) 
		- \chi \left( \underbrace{\delta - \sum_{j = i + 2}^{i_1 + 1}a_j}_{\delta'} \right) \\
		& \stackrel{(b)}{\leq} \chi (A_i  + \Delta) - \chi (A_i ),
	\end{align*}
	where (a) follows from \cref{lem:translation},
	and (b) follows from $A_i  \leq \delta - \sum_{j = i + 2}^{i_1 + 1}a_j  \leq \delta - \sum_{j = i + 2}^{i_1 + 1}a_j  + \Delta \leq A_{i+1} $ (as $\delta, \delta + \Delta \in [A_{i_1}, A_{i_1 + 1}]$ and $a_{i +1} \geq \cdots \geq a_{i_1 + 1}$)
	and \cref{lem:same-segment1} (ii).
	Now suppose that \eqref{3.37-result} holds for $i_2 - i_1 = m \geq 0$.
	Consider the case $i_2 - i_1 = m + 1 \geq 1$.
	Let $\Delta_1 = A_{i_1 + 1}  - \delta$. 
	By $A_{i_1 + 1} \leq A_{i_2} \leq \delta + \Delta$, $\Delta \geq \Delta_1$ must hold.
	Observe that
	\begin{equation*}
		\begin{aligned}
			\chi (A_{i_1+1} ) - \chi (\delta)
			& = \chi (A_{i_1+1} ) - \chi (A_{i_1 + 1} - \Delta_1) \\
			& \stackrel{(a)}{\leq}
			\chi \left(A_{i_1 + 1} - \sum_{j = i+2}^{i_1 + 1}a_j  \right) 
			- \chi \left(A_{i_1 + 1} - \sum_{j = i+2}^{i_1 + 1}a_j  - \Delta_1 \right) \\
			& = \chi (A_{i+1} ) - \chi (A_{i+1}  - \Delta_1) \\
			& \stackrel{(b)}{\leq} 
			\chi (A_i  + \Delta) - \chi (A_i  + \Delta - \Delta_1),
		\end{aligned}
	\end{equation*}
	where (a) follows from \cref{lem:translation},
	and (b) follows from $A_i \leq A_i  + \Delta - \Delta_1 \leq A_i  + \Delta \leq A_{i+1} $ (as  $\Delta \geq \Delta_1$ and $\Delta \in [0, a_{i+1} ]$)
	and \cref{lem:same-segment1} (i).
	Therefore, 
	\begin{equation*}
		\begin{aligned}
			\chi(\delta + \Delta) - \chi(\delta) & = 
			[\chi(\delta + \Delta) - \chi(A_{i_1 + 1})]
			+ [\chi(A_{i_1 + 1}) - \chi(\delta)] \\
			& \leq [\chi(\delta + \Delta) - \chi(A_{i_1 + 1})]
			+ [\chi (A_i  + \Delta) - \chi (A_i  + \Delta - \Delta_1)] \\
			& = [\chi(A_{i_1 + 1} + \Delta - \Delta_1) - \chi(A_{i_1 + 1})]
			+ [\chi (A_i  + \Delta) - \chi (A_i  + \Delta - \Delta_1)] \\
			& \stackrel{(a)}{\leq} [\chi (A_i  + \Delta - \Delta_1) - \chi (A_i )]
			+ [\chi (A_i  + \Delta) - \chi (A_i  + \Delta - \Delta_1)] \\
			& = \chi (A_i  + \Delta) - \chi (A_i),
		\end{aligned}
	\end{equation*}
	where (a) follows from $A_{i_1 + 1} \in [A_{i_1 + 1}, A_{i_1 + 2}]$, 
	$A_{i_1 + 1} + \Delta - \Delta_1 = \delta + \Delta \in [A_{i_2}, A_{i_2 + 1}]$, 
	$i_2 - (i_1 + 1) = m$, $\Delta -\Delta_1 \in [0,a_{i+1}]$,
	and the induction hypothesis.
\end{proof}
We are now ready to prove \eqref{goal1}--\eqref{goal2}.
\begin{proof}[Proof of \eqref{goal1}]
	Given any $j \in \Z_+$, observe that
	\begin{equation}\label{eq:cond1}
		\begin{small}
			\begin{aligned}
				\chi (A_{j+1} ) - \chi (A_j) & = \chi (A_{j} + a_{j + 1} ) - \chi (A_j) \\
				& \stackrel{(a)}{\geq} \chi (A_{j}  + \delta_2 + a_{j + 1}) - \chi (A_j  + \delta_2) 
				= \chi (A_{j+1}  + \delta_2) - \chi (A_j  + \delta_2),
			\end{aligned}
		\end{small}
	\end{equation}
	where (a) follows from $a_{j + 1} \in [0, a_{j + 1}]$, $A_j + \delta_2 \geq A_j$, and \cref{lem:translation2}.
	Suppose that $\delta_1 \in [A_i , A_{i+1} ]$ holds for some $i \in \Z_+$.
	Letting $\Delta : = \delta_1 - A_i \in [0, a_{i + 1}]$, then we have
	\begin{align*}
		\chi(\delta_1) & \stackrel{(a)}{=}
		\chi(A_i + \Delta) - \chi(A_i) + \sum_{j = 0}^{i - 1}[\chi(A_{j + 1}) - \chi(A_j)] \\
		& \stackrel{(b)}{\geq}
		\chi(A_i + \Delta) - \chi(A_i) + \sum_{j = 0}^{i - 1}[\chi (A_{j+1}  + \delta_2) - \chi (A_j  + \delta_2)] \\
		& = \chi(A_i + \Delta) - \chi(A_i) + \chi(A_i + \delta_2) - \chi(\delta_2)\\
		& \stackrel{(c)}{\geq} 
		\chi(A_i + \Delta + \delta_2) - \chi(A_i + \delta_2) + \chi(A_i + \delta_2) - \chi(\delta_2) \\
		& =\chi(\delta_1 + \delta_2) - \chi(\delta_2),
	\end{align*}
	where (a) follows from $\chi(A_0)= \chi(0)=0$,
	(b) follows from \eqref{eq:cond1},
	and (c) follows from $\delta_2 \geq 0$ and \cref{lem:translation2}.
	This completes the proof.
\end{proof}	
\begin{proof}[Proof of \eqref{goal2}]
	By $\bar{\chi}(\delta) = \chi(\delta)$ for $\delta \in [0, B_{m}  + (\gamma + 1)\epsilon]$ and \eqref{goal1}, it suffices to show that \eqref{goal2} holds
	for all $\delta_1, \delta_2 \in \R_+$ with $\delta_1 + \delta_2 > B_m  + (\gamma + 1) \epsilon$.
	We consider the following three cases.
	\begin{itemize}
		\item [(i)] $B_m  + (\gamma + 1) \epsilon < \delta_1 \leq \delta_2$.
		Observe that 
		\begin{equation*}
			\begin{small}
				\begin{aligned}
					& \bar{\chi}(\delta_1) + \bar{\chi}(\delta_2) - \bar{\chi}(\delta_1 + \delta_2) \\
					& \qquad = h(\delta_1 - A_{m-1}  - \gamma \epsilon + v_{m - 1}) 
					+ \chi (A_{m-1} )  
					+ \gamma \psi_{m - 1} - h(v_{m - 1}) \\
					& \qquad \quad
					+ h(\delta_2 - A_{m-1}  - \gamma \epsilon + v_{m - 1})
					- h(\delta_1 + \delta_2 - A_{m-1}  - \gamma \epsilon + v_{m - 1})  \\
					& \qquad  \stackrel{(a)}{=} 
					[h(\delta_1 - A_{m-1}  - \gamma \epsilon + v_{m - 1})  - h(v_{m - 1})]
					+ \tau \sum_{j = 1}^{m - 1}[h(b_j + v_{j - 1} + \epsilon) - h(b_ j + v_{j-1})] \\
					& \qquad \quad
					+ \sum_{j = 1}^{m - 1}[h(b_j + v_{j - 1}) - h(v_{j-1})]  
					+ \gamma [h(b_m + v_{m - 1} + \epsilon) - h(b_m + v_{m - 1})] \\
					& \qquad \quad
					- [h(\delta_1 + \delta_2 - A_{m-1}  - \gamma \epsilon + v_{m - 1}) - h(\delta_2 - A_{m-1}  - \gamma \epsilon + v_{m - 1})],
				\end{aligned}
			\end{small}
		\end{equation*}
		where (a) follows from \cref{properties} (v) and the definition of $\psi_{m - 1}$ in \eqref{def-rho}.
		Below we use \cref{cor:sum-pariwise-comparison} to show \eqref{goal2}.
		By 
		\begin{equation*}
			(\delta_1 - A_{m-1}  - \gamma \epsilon )
			+ \tau \sum_{j = 1}^{m - 1} \epsilon 
			+ \sum_{j = 1}^{m - 1} b_j 
			+ \gamma \epsilon 
			- \delta_1 = 0,
		\end{equation*}
		the condition \eqref{ccondition} in \cref{cor:sum-pariwise-comparison} holds.
		From $B_m  + (\gamma + 1) \epsilon < \delta_1 \leq \delta_2$ and $B_m = A_{m-1}+b_m$,
		we have
		\begin{small}
			\begin{equation*}
				b_m + v_{m - 1} + \epsilon
				< \delta_1 - A_{m-1}  - \gamma \epsilon + v_{m - 1}
				\leq \delta_2 - A_{m-1}  - \gamma \epsilon + v_{m - 1}
				\leq \delta_1 + \delta_2 - A_{m-1}  - \gamma \epsilon + v_{m - 1}.
			\end{equation*}
		\end{small}%
		Thus conditions (i) and (ii) in \cref{cor:sum-pariwise-comparison} hold and \eqref{goal2} follows.
		\item [(ii)] $0 < \delta_1 \leq B_m  + (\gamma + 1) \epsilon < \delta_2$.
		Suppose that $A_i  < \delta_1 \leq A_{i+1} $ for some $i \in \{0, \ldots, m - 1\}$,
		and let
		\begin{equation}\label{ell-def}
			\ell = \left\{
			\begin{aligned}
				& 0, && ~\text{if}~ A_i  < \delta_1 \leq B_{i+1} , \\
				& t, && ~\text{if}~ B_{i+1}  + t \epsilon < \delta_1 \leq B_{i+1}  + (t + 1) \epsilon, \ t = 0, \ldots, \tau - 1.
			\end{aligned}
			\right.
		\end{equation}
		Observe that
		\begin{equation*}
			\begin{aligned}
				& \bar{\chi}(\delta_1) + \bar{\chi}(\delta_2) - \bar{\chi}(\delta_1 + \delta_2) \\
				& \qquad = 
				h(\delta_1 - A_i  - \ell \epsilon + v_i)  + \chi (A_i ) + \ell \psi_{i} - h(v_i) \\
				& \qquad \quad
				+ h(\delta_2 - A_{m - 1}  - \gamma \epsilon + v_{m - 1}) - h(\delta_1 + \delta_2 - A_{m - 1}  - \gamma \epsilon + v_{m - 1}) \\
				& \qquad \stackrel{(a)}{=} 
				[h(\delta_1 - A_i  - \ell \epsilon + v_i) - h(v_i)] 
				+ \tau \sum_{j = 1}^{i}[h(b_j + v_{j - 1} + \epsilon) - h(b_j + v_{j - 1})] \\
				& \qquad \quad
				+ \sum_{j = 1}^{i}[h(b_j + v_{j - 1}) - h(v_{j - 1})] 
				+ \ell [h(b_{i + 1} + v_{i} + \epsilon) - h(b_{i + 1} + v_{i})] \\
				& \qquad \quad
				- [h(\delta_1 + \delta_2 - A_{m - 1}  - \gamma \epsilon + v_{m - 1}) - h(\delta_2 - A_{m - 1}  - \gamma \epsilon + v_{m - 1})],
			\end{aligned}
		\end{equation*}
		where (a) follows from \cref{properties} (v) and the definition of $\psi_{i}$ in \eqref{def-rho}.
		Below we use \cref{cor:sum-pariwise-comparison} to show \eqref{goal2}.
		By
		\begin{equation*}
			(\delta_1 - A_i  - \ell \epsilon)
			+ \tau \sum_{j = 1}^{i} \epsilon
			+ \sum_{j = 1}^{i} b_j 
			+ \ell \epsilon 
			- \delta_1 
			= 0,
		\end{equation*}
		the condition \eqref{ccondition} in \cref{cor:sum-pariwise-comparison} holds.
		By \eqref{ell-def},
		we have
		\begin{equation*}
			v_i \leq \delta_1 - A_i  - \ell \epsilon + v_i 
			\leq b_{i+1} + v_i + \epsilon.
		\end{equation*}
		Using the fact that $B_m + (\gamma + 1)\epsilon < \delta_2 < \delta_1 + \delta_2$
		and $b_{i + 1} + v_i \leq b_m + v_{m - 1}$ for any $i \in \{0, \ldots, m - 1\}$,
		it follows that
		\begin{equation*}
			b_{i + 1} + v_i + \epsilon \leq b_{m} + v_{m-1} + \epsilon
			< \delta_2 - A_{m - 1}  - \gamma \epsilon + v_{m - 1}
			< 
			\delta_1 + \delta_2 - A_{m - 1}  - \gamma \epsilon + v_{m - 1}.
		\end{equation*}
		Thus conditions (i) and (ii) in \cref{cor:sum-pariwise-comparison} hold
		and \eqref{goal2} follows.
		\item [(iii)] 
		$0 \leq \delta_1 \leq \delta_2 \leq B_{m}  + (\gamma + 1) \epsilon$
		and $\delta_1 + \delta_2 > B_{m}  + (\gamma + 1)\epsilon$.
		Observe that different from function $\chi(\delta)$, which depends on $b_i$ and $v_{i-1}$ for all $i =1, 2, \ldots$, function $\bar{\chi} (\delta)$ only depends on $b_i$ and $v_{i-1}$ for $i = 1, \ldots, m$.
		Consider a special case  of $\chi(\delta)$ where we enforce 
		\begin{equation}\label{def-b-v}
			b_i := 0, ~ v_{i - 1} := b_{m} + v_{m - 1}  ~\text{for}~ i = m + 1, m + 2, \ldots. 
		\end{equation}
		To prove the desired result, it suffices to show 
		\begin{equation}\label{chi-barchi}
			\bar{\chi} (\delta) \leq \chi (\delta), ~ \forall ~ \delta > B_{m}+(\gamma+1) \epsilon.
		\end{equation}
		Indeed, 
		\begin{equation*}
			\bar{\chi} (\delta_1) + \bar{\chi} (\delta_2) - \bar{\chi} (\delta_1 + \delta_2)
			\stackrel{(a)}{=}  \chi (\delta_1) + \chi (\delta_2) - \bar{\chi} (\delta_1 + \delta_2)
			\stackrel{(b)}{\geq}  \chi (\delta_1) + \chi (\delta_2) - \chi (\delta_1 + \delta_2) \geq 0, 
		\end{equation*}
		where (a) follows from $\chi(\delta) = \bar{\chi}(\delta)$ for $\delta \leq B_{m}+(\gamma+1) \epsilon$,
		and (b) follows from \eqref{chi-barchi}.
		
		Next, we prove \eqref{chi-barchi}.
		Let $\delta > B_{m}+(\gamma+1) \epsilon$.
		Since $B_{i + 1} = A_i + b_{i + 1} = A_i$ for $i = m, m + 1, \ldots$,
		$\delta$ falls into the third case of \eqref{chi}, and thus, 
		\begin{equation*}
			\chi (\delta) = h(\delta - A_i  -\ell \epsilon + v_i) + \chi (A_i ) + \ell \psi_{i} - h(v_i),
		\end{equation*}
		where $\delta \in (B_{i+1} + \ell \epsilon, B_{i+1} + (\ell + 1) \epsilon]$ 
		for some $i \in \{m - 1, m, \ldots\}$ and $\ell \in \{0, \ldots, \tau - 1\}$.
		Using the definition of $\bar{\chi}$ in \eqref{barchi}, we have
		\begin{equation}\label{tmp3}
			\begin{aligned}
				\chi(\delta) - \bar{\chi}(\delta) 
				& = [ h(\delta - A_i - \ell \epsilon + v_i)
				- h(\delta - A_{m - 1} - \gamma \epsilon + v_{m - 1}) ] \\
				& \quad
				+ [
				\chi(A_i) + \ell \psi_{i} - h(v_i)
				- (\chi(A_{m - 1}) + \gamma \psi_{m - 1} - h(v_{m - 1}))
				].
			\end{aligned}
		\end{equation}
		Note that
		\begin{equation*}
			\begin{aligned}
				& \chi(A_i) + \ell \psi_{i} - h(v_i)
				- (\chi(A_{m - 1}) + \gamma \psi_{m - 1} - h(v_{m - 1})) \\
				& \qquad 
				= [\chi(A_i) - \chi(A_{m - 1})] + \ell \psi_{i} - \gamma \psi_{m - 1}
				- h(v_i) + h(v_{m - 1}) \\
				& \qquad
				\stackrel{(a)}{=} 
				\tau \sum_{j = m}^{i} \psi_{j - 1}
				+ \sum_{j = m}^{i}(h(b_j + v_{j- 1}) - h(v_{j - 1})) 
				+ \ell\psi_{i} - \gamma \psi_{m - 1}
				- h(v_i) + h(v_{m - 1}) \\
				& \qquad
				\stackrel{(b)}{=} 	\tau \sum_{j = m}^{i} \psi_{j - 1} +\ell\psi_{i} - \gamma \psi_{m - 1} \\
				& \qquad
				\stackrel{(c)}{=}(\tau (i - m + 1) + \ell - \gamma)[ h(b_m + v_{m- 1} + \epsilon) - h(b_m + v_{m - 1})]
			\end{aligned}
		\end{equation*}
		where (a) follows from \cref{properties} (v),
		(b) holds trivially if $i = m - 1$, and follows from \eqref{def-b-v} otherwise,
		and (c) follows from $	\psi_{j - 1} = h(b_j + v_{j- 1} + \epsilon) - h(b_j + v_{j - 1})
		= h(b_m + v_{m- 1} + \epsilon) - h(b_m + v_{m - 1}) $ for $j \geq m$.
		Together with \eqref{tmp3}, we have
		\begin{align*}
			\chi (\delta) - \bar{\chi} (\delta)
			& = (\tau(i - m + 1) + \ell - \gamma) [h(b_m + v_{m - 1} + \epsilon) - h(b_m + v_{m - 1})] \\
			& \quad
			- [h(\delta - A_{m - 1}  - \gamma \epsilon + v_{m - 1}) - h(\delta - A_i  -\ell \epsilon + v_i)].
		\end{align*}
		Below we use \cref{cor:sum-pariwise-comparison} to show $\chi(\delta) - \bar{\chi}(\delta) \geq 0$.
		First, we have 
		\begin{equation*}
			(\tau(i - m + 1) + \ell - \gamma) \epsilon
			- [- A_{m - 1}  - \gamma \epsilon + v_{m - 1} - (- A_i  -\ell \epsilon + v_i)] = 0.
		\end{equation*}
		Indeed, if $i = m - 1$, then the above equality is trivial;
		otherwise, the above equality follows from
		$A_i - A_{m - 1} + v_{m - 1} - v_i = \sum_{j = m}^{i} a_j + v_{m - 1} - (b_m + v_{m - 1})
		= (a_m - b_m) + \sum_{j = m + 1}^{i} (b_j + \tau \epsilon) = \tau (i - m + 1) \epsilon$.
		Thus the condition \eqref{ccondition} in \cref{cor:sum-pariwise-comparison} holds.
		By $\delta \in (B_{i+1} + \ell \epsilon, B_{i+1} + (\ell + 1) \epsilon]$ and  $B_m + (\gamma + 1) \epsilon < \delta$, either $i \geq m$ or $\ell \geq \gamma +1$ must hold, which together with $\gamma \leq \tau -1$, implies
		\begin{equation*}
			\tau(i - m + 1) + \ell - \gamma > 0.
		\end{equation*}
		Finally, by $\delta \in (B_{i+1} + \ell \epsilon, B_{i+1} + (\ell + 1) \epsilon]$,
		it follows
		\begin{equation*}
			b_{i + 1} + v_{i}
			< \delta - A_i  -\ell \epsilon + v_i 
			\leq b_{i + 1} + v_{i} + \epsilon.
		\end{equation*}
		Combining it with $i \geq m - 1$, \eqref{def-b-v},
		and $B_m + (\gamma + 1) \epsilon < \delta$ yields
		\begin{equation*}
			b_m + v_{m - 1}
			< \delta - A_i  -\ell \epsilon + v_i 
			\leq b_m + v_{m - 1} + \epsilon 
			< \delta - A_{m - 1}  - \gamma \epsilon + v_{m - 1}.
		\end{equation*}
		Thus, conditions (i) and (ii) in \cref{cor:sum-pariwise-comparison} hold and the result follows. \qedhere
	\end{itemize}
\end{proof}

\renewcommand{\refname}{\normalsize References}\small
\bibliographystyle{apalike}

\bibliography{shorttitles,wta}
\appendix

\end{document}